\documentclass[moor,sglanonrev]{informs5}
\usepackage{eqndefns-left} % For checking the display equation width and equation environment definitions %
\RequirePackage{tgtermes}
\RequirePackage{newtxtext}
\RequirePackage{newtxmath}
\RequirePackage{bm}
\RequirePackage{endnotes}

\OneAndAHalfSpacedXI
\usepackage[ruled,vlined,linesnumbered]{algorithm2e}
\usepackage{algpseudocode}
\usepackage{tikz}

\usepackage{booktabs}
\usepackage{makecell}
\usepackage{hyperref}

\usepackage[sort&compress]{natbib}
 \bibpunct[, ]{[}{]}{,}{n}{}{,}%
 \def\bibfont{\small}%
\EquationsNumberedThrough    % Default: (1), (2), ...
\TheoremsNumberedThrough     % Preferred (Theorem 1, Lemma 1, Theorem 2)
\ECRepeatTheorems  %  

\MANUSCRIPTNO{MOOR-0001-2025.00}

\def\tr{^{\intercal}}

\def\Z{\mathbb{Z}}
\def\Re{{\mathbb R}}

\def\conv{\mathop{\rm conv}}

\def\proj{\mathop{\rm proj}}

\def\Re{{\mathbb R}}

\def\BLP{BL\&P }

\newcommand{\vc}[1]{\bm{#1}}	%defines vector format
\newcommand{\mac}[1]{\mathcal{#1}}	%defines font format
\newcommand{\blue}[1]{\textcolor{black}{#1}}

\begin{document}
%%%%%%%%%%%%%%%%

%\setcounter{page}{1}
%\DOI{10.1287/xxxx.0000.0000}%

\RUNAUTHOR{Davarnia and Rahimian}

\RUNTITLE{A Unifying Convexification Framework for Chance-Constrained Programs}

\TITLE{\blue{A Unifying Convexification Framework for Chance-Constrained Programs with Finite Support via Bilinear Formulations over a Simplex}}

\ARTICLEAUTHORS{%
%\AUTHOR{John Doe,\textsuperscript{a} Jane Smith,\textsuperscript{b}}
%\AFF{\textsuperscript{a}Department of Industrial Engineering, University of XYZ, \EMAIL{john.doe@xyz.edu; \textsuperscript{b}Department of Computer Science, University of ABC, \EMAIL{jane.smith@abc.edu}} 
\AUTHOR{Danial Davarnia}
\AFF{Edwardson School of Industrial Engineering, Purdue University, West Lafayette, IN 47907, \EMAIL{ddavarn@purdue.edu}}

\AUTHOR{Hamed Rahimian}
\AFF{Department of Industrial Engineering, Clemson University, Clemson, SC 29634, \EMAIL{hrahimi@clemson.edu}}
} % end of the block

\ABSTRACT{%
% Enter your abstract
Chance-constrained programming is a widely used framework for decision-making under uncertainty, yet its mixed-integer reformulations \blue{over a finite support} involve nonconvex mixing sets with a knapsack constraint, leading to weak relaxations and computational challenges. Most existing approaches for strengthening the relaxations of these sets rely primarily on extensions of a specific class of valid inequalities, limiting both convex hull coverage and the discovery of fundamentally new structures. In this paper, we develop a novel convexification framework that reformulates chance-constrained sets as bilinear sets over a simplex in a lifted space and employs a step-by-step aggregation procedure to derive facet-defining inequalities in the original space of variables. Our approach generalizes and unifies established families of valid inequalities in the literature while introducing new ones that capture substantially larger portions of the convex hull. Main contributions include: 
%(i) a new aggregation-based convexification technique for bilinear sets with binary variables constrained to a simplex,
(i) the development of a new aggregation-based convexification technique for bilinear sets over a simplex in a lower-dimensional space;
(ii) the introduction of a novel bilinear reformulation of mixing sets with a knapsack constraint---arising from single-row relaxations of chance constraints---over a simplex, which enables the systematic derivation of strong inequalities in the original variable space; (iii) the characterization of facet-defining inequalities within a unified framework that contains both existing and new families; and (iv) the development of an efficient exact separation algorithm for an important subclass of the derived families. Preliminary computational experiments demonstrate that our inequalities describe over 90\% of the facet-defining inequalities of the convex hull of benchmark instances, significantly strengthening existing relaxations and advancing the polyhedral understanding of chance-constrained programs.
}%

\FUNDING{The first author was supported in part by the U.S. Air Force Office of Scientific Research 
grant FA9550-23-1-0183, and the National Science Foundation 
grant CMMI-2547883. The second author was supported in part by the U.S. Air Force Office of Scientific Research 
grant FA9550-24-1-0241.}

%Supplemental Material:
%Data Ethics & Reproducibility Note:

%\KEYWORDS{Stochastic programming, Decision support,Uncertainty, Disaster response, Optimization}

%\HISTORY{Received: Month DD, YYYY; Accepted: Month DD, YYYY; Published Online: Month DD, YYYY}

\maketitle
%%%%%%%%%%%%%%%%%%%%%%%%%%%%%%%%%%%%%%%%%%%%%%%%%%%%%%%%%%%%%%%%%%%%%%

\section{Introduction} \label{sec:introduction}
%%%%%%%%%%%%%%%%%%%%%%%%%%%%%%%%%%%%%%%%%%%%%%%%%%%%%%%%%%%%%%%%%%%%%%%%%%%%%%%%%%%%%%%%%%%%
%%%%%%%%%%%%%%%%%%%%%%%%%%%%%%%%%%%%%%%%%%%%%%%%%%%%%%%%%%%%%%%%%%%%%
%%%%%%%%%%%%%%%%%%%%%%%%%%%%%%%%%%%%%%%%%%%%%%
%%%%%%%%%%%%%%%%%%%%%%%

Many decision-making problems in uncertain environments, arising in engineering and operation, must account for reliability requirements, service levels, or risk measures \cite{andrieu2010model,pathy2024value}. Chance-constrained programming (CCP), first introduced by \cite{charnes1959chance}, has become a widely used modeling framework to address such needs. This research area continues to attract significant interest, with recent advances in both theory and computation \citep{pena2020solving,rahimian2023data,shenglong2024}; see also \cite{kucukyavuz2023survey} for a comprehensive survey.

\smallskip
In this paper, we study the polyhedral structure of an important substructure that arises in mixed-integer programming (MIP) reformulations of CCPs \blue{over a finite support \cite{ruszczynski2002probabilistic}}, known as the \textit{mixing set with a knapsack constraint}. The nonconvexity of these sets has motivated extensive research aimed at understanding their polyhedral properties. A fundamental direction in this line of work is the development of explicit valid inequalities for their convex hull in the original space of variables. Such inequalities can be embedded directly into mathematical models to strengthen relaxations or incorporated into branch-and-cut algorithms \cite{luedtke2014branch,liu2016decomposition,pathy2024value}, thereby improving both computational efficiency and solution quality.

\smallskip
\blue{
While extended formulations based on the disjunctive programming \cite{balas:1979,balas:1985} approach exist for some variants of the mixing set \cite{kuccukyavuz2012mixing}, the literature has primarily focused on developing closed-form valid inequalities in the original variable space for four main reasons.
First, explicit inequalities provide a more direct characterization of the convex hull in the original variable space, whereas extended formulations do not explicitly reveal this structure. Second, they avoid the lifting required by disjunctive programming formulations, which introduces additional variables and constraints and can result in substantially larger formulations whose computational cost increases with problem size. Third, closed-form inequalities often admit a parametric representation, whereby specifying the parameters immediately yields a valid inequality in the original variable space, whereas extended formulations require identifying a solution to the projection problem before such an inequality can be generated. Fourth, explicit inequalities often admit more efficient separation algorithms than the cut-generating linear programs associated with lift-and-project formulations that are typically used to separate points in the original variable space \cite{balas:ce:co:1993,rajabalizadeh:da:2023}.
}

\smallskip
An important class of valid inequalities for CCPs is the \textit{strengthened mixing} inequalities. In their seminal work, \citet{luedtke2010integer} studied a \textit{single-row} relaxation of the chance constraint, modeled as a mixing set with a 0-1 knapsack constraint, and derived a family of valid (and sometimes facet-defining) inequalities \blue{that generalize the classical \textit{star} or \textit{mixing} inequalities, see, e.g., \cite{atamturk2000mixed,gunluk2001mixing,sanjeevi2012mixed,atamturk2012n,atamturk2010mingling} and references therein.}
Building on this, \citet{kuccukyavuz2012mixing} developed a new family of valid inequalities for the mixing set with a cardinality constraint that extended earlier inequalities in \cite{luedtke2010integer}. 
%It proposed compact extended formulations for joint chance constraints, 
% and introduced a blending procedure to transmit the strength of single-row relaxations to multi-row intersections, thereby unifying and extending earlier inequalities in \cite{luedtke2010integer}. 
Later, \citet{abdi2016mixing} characterized a broader class of valid inequalities for the general-probability case, subsuming those in \cite{luedtke2010integer,kuccukyavuz2012mixing}, while \citet{zhao2017polyhedral} derived an even more general family encompassing all known explicit inequalities for the single-row relaxation. % and proposed blending mechanisms that, under certain conditions, yield facet-defining inequalities for intersections of multiple mixing sets. More recently, \citet{liu2019intersection} analyzed intersections of two mixing sets with shared continuous variables---a key substructure in joint CCPs---and derived valid (sometimes facet-defining) inequalities. 
%They defined the \emph{quantile cut closure} and established several key properties: (i) the closure has a finite description; (ii) for linear CCPs, the closure is polyhedral; and (iii) repeated application of quantile closures converges to the convex hull, with finite convergence in purely integer cases. 

%{\color{red} We note that a line of research on the polyhedral study of CCPs focuses on \textit{quantile} cuts  \cite{xie2018quantile,qiu2014covering}, derived in the original variable space. 
%It is shown that for a linear CCP, studied in this paper, the closure of the infinite family of quantile cuts has a finite polyhedral description \cite[Corollary~1]{xie2018quantile}. %%Moreover, a recursive application of such closure operations exactly characterizes the projection of the convex hull of the mixing set with a knapsack constraint in the limit. 
%While such cuts lead to valid inequalities for the convex hull of the mixing set with a knapsack constraint in the original space of variables, this differs from our work, as the desired set is characterized by a recursive application of the closure operations in the limit \cite[Proposition~2]{xie2018quantile}. Moreover, these cuts are not explicit as deriving them requires solving several single-scenario optimization problems.}
%However, different from our work, even their separation in the first round of  is NP-hard \cite[Theorem~7]{xie2018quantile}.} % and \cite[Proposition~5]{qiu2014covering}.}

\smallskip
Although the literature provides several families of valid inequalities for mixing sets, two main limitations remain. First, computational evidence shows that existing families capture only a limited portion of the convex hull, yielding relatively weak relaxations that often fail to deliver notable bound improvements when added to the initial relaxation of CCPs.
Second, the prevailing approach to deriving new inequalities has largely focused on a specific class, typically by starting from a known family and expanding it through the addition of new elements. This constrains the discovery of fundamentally different classes of inequalities with distinct forms that fall outside previously established families. As a result, much of the existing work presents extended variants of known inequalities and relies on case-by-case analysis to prove validity. This approach not only lacks an intuitive or constructive methodology to inspire the derivation of new families with different structures, but also limits coverage of the convex hull---contributing to the relatively weak relaxation bounds noted earlier.
These limitations highlight the need for continued efforts in designing systematic procedures to identify new and richer classes of valid inequalities, with the potential to improve optimality gap closure more meaningfully during the solution process.

\smallskip
To address these limitations, we propose a novel convexification approach through a different lens: reformulating the original mixing set with a knapsack constraint as a bilinear set over a simplex in a lifted space, which allows for employing a systematic aggregation method to characterize important families of facet-defining inequalities for the convex hull of the set in the original space of variables. 
The motivation for this work stems from \cite{davarnia2017simultaneous} and \cite{khademnia:da:2024}, which introduced step-by-step aggregation methods for bilinear sets to obtain a closed-form representation of their convex hull in the original space of variables. Computational experiments across various application areas, from network interdiction to transportation, have demonstrated two main advantages of this convexification approach compared to classical methods such as disjunctive programming \cite{balas:1979,balas:1985} and the reformulation-linearization technique \cite{sherali:ad:1990,sherali:al:1992}: (i) it enables the rapid generation of valid inequalities for bilinear sets to tighten convex relaxations, and (ii) it provides a systematic process to derive explicit facet-defining inequalities of the convex hull in the original space of variables.
We refer the reader to \cite{tawarmalani:ri:ch:2010,locatelli:sch:2014,boland:de:ka:mo:ri:2017,muller:se:gl:2020,gupte:ka:ri:wa:2020,fampa:lee:2021,delpia:kh:2021,davarnia:ki:2025,rahimian:me:2024} for a collection of works on the polyhedral study and convexification approaches for bilinear problems under various settings.

\smallskip
Building on this foundation, the present paper extends the existing literature by introducing a new convexification method for a bilinear set with a different structure from those previously studied.
The main differences between the structures we study here and those in \cite{davarnia2017simultaneous,khademnia:da:2024} are as follows: (i) we focus on sets that contain disjoint bilinear \textit{constraints}, whereas previous works examined the graph of disjoint bilinear functions over linear constraints; (ii) referring to the disjoint bilinear terms in the constraints as $x_iy_j$, we derive the convex hull in the space of the $x$ variables only, while previous studies obtained the convex hull in the space of both $x$ and $y$ variables; and (iii) in our case, the $y$ variables are binary, whereas in the prior works, the $y$ variables can be continuous.
These differences necessitate the design of a new aggregation method with components that differ from those proposed in prior works.

\smallskip
This extension enables us to apply our convexification framework to mixing sets with a knapsack constraint, through a novel reformulation of the problem as a bilinear set over a simplex in a lifted space. 
%Our framework proceeds in two steps. First, we reformulate the underlying set as an extended bilinear model aligned with the structure studied in this paper. Second, we apply the developed convexification method to this bilinear set, deriving a rich family of valid inequalities that not only subsumes several known families in the literature but also introduces new inequalities for the convex hull of chance-constrained sets.
The advantages of this convexification framework are two-fold.
First, it provides a step-by-step aggregation procedure for deriving valid inequalities for the mixing set with a knapsack constraint (see Theorem~\ref{thm:chance_lift_generic}). Unlike existing approaches, this yields a constructive recipe for generating multiple classes of valid inequalities in different forms.
Second, it leads to the closed-form characterization of a broad family of valid inequalities that not only encompasses many existing families from the literature but also introduces new ones (see Theorem~\ref{thm: chance_liftVIAgeneric_permutation}). Together, these inequalities collectively capture a significantly larger portion of the convex hull, as confirmed by our computational experiments on benchmark instances.

\smallskip
In summary, the contributions of this work are as follows.
First, we develop a step-by-step aggregation technique to derive facet-defining inequalities for the convex hull of bilinear sets containing disjoint bilinear terms of the form $x_i y_j$ in the $x$-variable space, where the $y$ variables are binary and constrained to lie in a simplex. We believe this aggregation-based lift-and-project technique may be of independent interest beyond the class of CCPs studied in this paper.
Second, we propose a novel reformulation of the mixing set with a knapsack constraint as a bilinear set over a simplex in a lifted space, allowing for the application of our convexification procedure to obtain valid inequalities. 
Third, our convexification procedure provides a systematic framework to characterize important families of valid inequalities that not only subsume several well-known families in the literature but also unify them under a common structure.
Fourth, we identify new classes of valid inequalities that capture large portions of the convex hull not attainable by existing approaches.
Fifth, we design an efficient exact separation method to identify the most-violated inequalities within a subclass of the derived families. 
Sixth, preliminary computational experiments on \blue{two benchmark sets of} instances show that our new inequalities describe over 90\% of the facet-defining inequalities of the convex hulls, representing a substantial improvement over prior results.
Collectively, these contributions advance our understanding of the polyhedral structure of CCPs.

\smallskip
The remainder of this paper is organized as follows. Section~\ref{sec:background} provides background on the MIP reformulation of CCPs as a mixing set with a knapsack constraint, along with a review of explicit families of valid inequalities derived in the literature. Section~\ref{sec:BLP} introduces our aggregation-based convexification procedure for a general bilinear set containing disjoint bilinear terms in the constraints, which may also be of independent interest beyond CCPs. In Section~\ref{sec:chance_bilinear}, we present a reformulation of the mixing set with a knapsack constraint as a bilinear set in a lifted space, derive families of valid inequalities for this set via our convexification procedure, and develop an efficient exact separation algorithm for these families.
Section~\ref{sec:numerical} reports preliminary computational experiments on benchmark instances from the literature, demonstrating the extent to which our families of inequalities capture the convex hull compared to existing approaches. We end with conclusions and directions of future research 
in Section~\ref{sec:discussion}. 

\medskip
\textbf{Notation.} 
We denote vectors in boldface.
%For any $n \in \N$, we refer to the ordered index set $\{1,\dotsc,n\}$ by $[n]$.
For a set $S \subseteq \{(\vc{x}; \vc{y}) \in \Re^n \times \Re^m\}$, we refer to the convex hull and projection of $S$ on to the $\vc{x}$-space by $\conv(S)$ and $\proj_{\vc{x}}(S)$, respectively. 
%Given $n \in \N$, we denote by $\vc{e}^i$ for $i \in \{1, \dotsc, n\}$ the $i$-th unit vector in $\Re^n$.
%We refer to the origin in $\Re^n$ by $\vc{e}^0$.
%{\color{red} $\vc{1}$ denotes the column vector of all 1's.} 
Given $x \in \Re$, we define $(x)^+ = \max\{x, 0\}$.
%Given $n \in \N$, we use $[n]$ to denote the index set $\{1,\dotsc,n\}$. 
Further, we define the indicator function $\mathbb{I}(x) = 0$ if $x = 0$, and $\mathbb{I}(x) = 1$ otherwise.

\section{Background on Chance-Constrained Programming} \label{sec:background}

%In this section, we investigate the polyhedral structure of chance-constrained sets, which arise in various modeling approaches for stochastic optimization problems across diverse domains, including healthcare and energy. Due to the mixed-integer nature of these sets, significant research has been devoted to understanding their polyhedral properties. An important direction in this area involves deriving explicit valid inequalities for the convex hull of these sets, which can be directly incorporated into mathematical models to tighten relaxations and enhance both solution time and quality.

\smallskip
%Although the literature contains several families of such inequalities, computational evidence indicates that they capture only a limited portion of the convex hull, resulting in relatively weak relaxations. This highlights the need for continued efforts to discover richer classes of valid inequalities that can improve gap closure during optimization.

\smallskip
%To address this challenge, we propose a novel convexification approach through a different lens: reformulating the original chance-constrained set as a lifted bilinear set. This reformulation aligns with the structure studied in Section ??, thereby enabling the application of the \BLP procedure to systematically derive a broader family of explicit valid inequalities for its convex hull. Our approach unifies and extends existing results in the literature, and offers a new avenue for strengthening formulations of chance-constrained problems. 

We consider a probabilistically-constrained linear program with a random right-hand side:
\begin{subequations}
\label{form:chance}
\begin{align}
   \min_{\vc{u}} \quad  &   \vc{c}\tr \vc{u} \\
\text{s.t.} \quad & \mathbb{P}(W \vc{u} \geq \vc{\xi}) \geq 1 - \varepsilon, \\
&  \vc{u} \in U,
\end{align} 
\end{subequations}
where $U \subseteq \Re_+^{d_u}$ is a polyhedron, $W \in \Re^{d_\xi \times d_u}$ is a matrix, $\vc{\xi} \in \Re^{d_\xi}$ is a random vector, $\varepsilon \in [0,1]$ is a risk parameter, and $\vc{c} \in \Re^{d_u}$. 
Assume that $\vc{\xi}$ has a finite support on $m$ scenarios, i.e., $\vc{\xi} \in \{\vc{\xi}_1, \dotsc, \vc{\xi}_m\}$ with $\mathbb{P}(\vc{\xi} = \vc{\xi}_i) = \pi_i > 0$ for each $i \in M:=\{1,\dotsc,m\}$ and $\sum_{i=1}^m \pi_i = 1$. 
We assume without loss of generality that $\vc{\xi}_i \geq \vc{0}$ for each $i \in M$. Moreover, we impose that $\pi_i \leq \varepsilon$ for each $i$, since if $\pi_i > \varepsilon$, then $W \vc{u} \geq \vc{\xi}_i$ must hold for any feasible $\vc{u}$. In such cases, we can incorporate these inequalities directly into the definition of $U$ and exclude the corresponding scenario $i$ from further consideration. 
We also note that \eqref{form:chance} can represent a sampling-based approximation of a CCP, where the samples are given by $\{\vc{\xi}_1, \dotsc, \vc{\xi}_m\}$ and $\pi_i=1/m$ for all $i \in M$.

\subsection{MIP Reformulation of CCP} \label{subsec:MIP}

In this section, we present the widely-used formulation of the CCP, defined in \eqref{form:chance}, as a MIP. For each $i \in M$, we introduce a binary variable $x_i \in \{0, 1\}$, where $x_i=0$ enforces the constraint $W\vc{u} \ge \vc{\xi}_i$. Since $\vc{\xi}_i \ge \vc{0}$, this also implies that $W\vc{u} \ge \vc{0}$ for every feasible $\vc{u}$. Defining $\vc{v} = W\vc{u}$, we can then reformulate \eqref{form:chance} as 
\begin{subequations}
\label{form:chance_MIP}
\begin{align}
\min_{\vc{v}, \vc{u}, \vc{x}}  \quad & \vc{c}\tr \vc{u} \\
\text{s.t.}   \quad & \vc{u} \in U, \\
& \vc{v} - W\vc{u}=\vc{0}, \\
& x_i=0 \Rightarrow \vc{v} \ge \vc{\xi}_i, \quad  i  \in M, \\ 
%&  \sum_{i \in M}\pi_i x_i \leq \varepsilon, \\
&  \vc{\pi}\tr \vc{x}\leq \varepsilon, \label{eq:CCP5} \\
&  \vc{x} \in \{0,1\}^m, \; \vc{v} \ge \vc{0}. 
\end{align}
\end{subequations}

In this formulation, the inequality \eqref{eq:CCP5} is sometimes referred to as either a \textit{chance constraint} or a \textit{knapsack constraint}; in this paper, we adopt the latter terminology.
It is common in the literature to study the polyhedral structure of a subset of the feasible region of \eqref{form:chance_MIP} in the space of variables $(\vc{v};\vc{x})$, defined as 
\begin{equation*} 
\mac{G} = \Biggl\{ (\vc{v},\vc{x}) \in \Re_+^{d_\xi} \times \{0,1\}^m\, \Bigg| \,
\begin{array}{l}
x_i=0 \Rightarrow \vc{v}\ge \vc{\xi}_i, \;  i \in M\\
%\sum_{i \in M}\pi_i x_i \leq \varepsilon 
\vc{\pi}\tr \vc{x}\leq \varepsilon
\end{array}
\Biggr\}.   
\end{equation*}
We can write $\mac{G}$ as 
\begin{equation*}
\mac{G} = \bigcap_{j=1}^{d_\xi} \Bigl\{(\vc{v}, \vc{x}): (v_j, \vc{x}) \in \mac{F}_j\Bigr\},    
\end{equation*}
where, for $j\in \{1,\dotsc,d_\xi\}$, we define 
\begin{equation*} 
\mac{F}_j = \Biggl\{ (v_j,\vc{x}) \in \Re_+ \times \{0,1\}^m\, \Bigg| \,
\begin{array}{l}
x_i=0 \Rightarrow v_j \ge \xi_{ij}, \;  i \in M\\
%\sum_{i \in M}\pi_i x_i \leq \varepsilon 
\vc{\pi}\tr \vc{x}\leq \varepsilon
\end{array}
\Biggr\}.   
\end{equation*}
Therefore, a natural first step in developing a strengthened formulation for $\mac{G}$ is to develop a strengthened formulation for each $\mac{F}_j$. In particular, we have that any facet-defining inequality for
$\conv(\mac{F}_j)$ is also facet-defining for $\conv(\mac{G})$. This follows because the set of $m + 1$ affinely independent points in $\mac{F}_j$ that serves as the support of a facet-defining inequality of $\conv(\mac{F}_j)$ can be trivially extended to a set of $m+d_\xi$ affinely independent
supporting points of this inequality in $\mac{G} \subseteq \Re^{m+d_\xi}$ by appropriately assigning values to the $v_i$ variables for each  $i \neq j$; see \cite{luedtke2010integer,kuccukyavuz2012mixing} for a detailed account.
Thus, we consider a generic set 
\begin{equation} \label{eq:genprob_chance}
\mac{F}_c = \Biggl\{ (z,\vc{x}) \in \Re_+ \times \{0,1\}^m\, \Bigg| \,
\begin{array}{l}
x_i=0 \Rightarrow z \ge h_i, \;  i \in M\\
%\sum_{i \in M}\pi_i x_i \leq \varepsilon 
\vc{\pi}\tr \vc{x}\leq \varepsilon
\end{array}
\Biggr\},
\end{equation}
by dropping the index $j$, replacing $v_j$ with $z$, and substituting $\xi_{ij}$ with $h_j$ for each $i \in M$. 
In what follows, we refer to $\mac{F}_c$, \blue{or equivalently its linearized form $\mac{F}_c = \{ (z,\vc{x}) \in \Re_+ \times \{0,1\}^m\, | \,
z \ge h_i(1-x_i), \;  i \in M, \;  
\vc{\pi}\tr \vc{x}\leq \varepsilon\}$,} as the \textit{mixing set with a knapsack constraint}.
To perform our polyhedral analysis of this set, we begin by introducing notation and terminology. Without loss of generality, we assume that the right-hand-side values are ordered, i.e., $h_1\ge h_2 \ge \cdots \ge h_m$. Define $p=\max\{k: \sum_{i=1}^{k} \pi_i \le \varepsilon\}$. 
In addition, let $\{\left< 1 \right>,\dotsc, \left< m \right>\}$ be a permutation of $\{1,\dotsc,m\}$ such that $\pi_{\left< 1 \right>} \le \dotsc \le \pi_{\left< m \right>}$. Define $\vartheta=\max\{k: \sum_{i=1}^{k} \pi_{\left< i \right>} \le \varepsilon\}$. 
By definition, it is clear that $p,\vartheta \in M$. 

An important special case of $\mac{F}_c$ arises when the underlying probability distribution is uniform.
In this case, we have $\pi_i=1/m$ for all $i \in M$, which implies $p=\vartheta= \lfloor m \varepsilon \rfloor $.
Consequently, the knapsack constraint $\vc{\pi}\tr \vc{x}\leq \varepsilon$ reduces to the cardinality constraint $\vc{1}\tr \vc{x}\leq p$, where $\vc{1}$ is the vector of all ones with matching dimension.
Due to its numerous practical applications, this special case has received significant attention in the literature, \blue{see, e.g., \cite{song2014chance,wang2022solution,deng2016decomposition}.}
For ease of reference, we define it separately here and refer to it throughout the paper as the \textit{mixing set with a cardinality constraint}, formally given by
\begin{equation} \label{eq:cardinality}
\mac{F}_c^= = \Biggl\{ (z,\vc{x}) \in \Re_+ \times \{0,1\}^m\, \Bigg| \,
\begin{array}{l}
x_i=0 \Rightarrow z \ge h_i, \;  i \in M\\
%\sum_{i \in M}\pi_i x_i \leq \varepsilon 
\vc{1}\tr \vc{x}\leq p
\end{array}
\Biggr\}.
\end{equation}

Next, in Section \ref{sec:chance_existing}, we review existing results on the polyhedral structure of $\mac{F}_c$ that provide closed-form expression of valid inequalities for this set.

\subsection{Existing Families of Valid Inequalities in Closed Form}
\label{sec:chance_existing}

Set $\mac{F}_c$, as defined in \eqref{eq:genprob_chance},  has been extensively studied in the literature. When $\varepsilon=1$ (i.e., the chance constraint is trivially satisfied), $\mac{F}_c$ reduces to a \textit{mixing} set defined as
\begin{equation} \label{eq:mixing}
\mac{F}_c^1 = \Bigl\{ (z,\vc{x}) \in \Re_+ \times \{0,1\}^m\, \Big \vert 
\begin{array}{l}
x_i=0 \Rightarrow z\ge h_i, \;  i \in M\\ 
\end{array}
\Bigr\},    
\end{equation}
and is studied in \cite{gunluk2001mixing,atamturk2000mixed,guan2007sequential,miller2003tight}. 
The following result has served as the foundation for deriving valid inequalities for the CCP over the past several decades.

\begin{proposition}{\cite{gunluk2001mixing,atamturk2000mixed,guan2007sequential}}
\label{prop:mixing}
Consider $\{t_{1}, \cdots, t_{l}\} \subseteq M$ with $t_1 < \cdots < t_l$. Then, the {\it star inequalities}  
\begin{equation} \label{eq:star_ineq}
    z + \sum_{\iota=1}^{l} (h_{t_{\iota}} - h_{t_{\iota+1}}) x_{t_{\iota}} \ge h_{t_{1}}, 
\end{equation}
where $h_{t_{l+1}}:=0$, are valid for the mixing set \eqref{eq:mixing}. 
\end{proposition}

%The star inequalities \eqref{eq:star_ineq} are facet-defining when $t_1=1$ and sufficient to define $\conv(\mac{F}_c^1)$.  
Later, \cite{luedtke2010integer} extends Proposition~\ref{prop:mixing} to the case where $\varepsilon \leq 1$, as defined in $\mac{F}_c$ given by \eqref{eq:genprob_chance}.
\begin{proposition}
    {\cite[Theorem~2]{luedtke2010integer}}
\label{prop:chance}
Consider $\{t_{1}, \cdots, t_{l}\} \subseteq \{1,\dotsc,p\}$ with $t_1 < \cdots < t_l$. Then, the {\it strengthened star inequalities} 
\begin{equation} \label{eq:chance_ineq}
    z + \sum_{\iota=1}^{l} (h_{t_{\iota}} - h_{t_{\iota+1}}) x_{t_{\iota}} \ge h_{t_{1}}, 
\end{equation}
where $h_{t_{l+1}}:=h_{p+1}$, are valid for $\mac{F}_c$.
\end{proposition}

%Inequalities \eqref{eq:chance_ineq} are facet-defining for $\conv(\mac{F}_c)$ if and only if $t_1=1$, although they are not sufficient to describe $\conv(\mac{F}_c)$ \cite{luedtke2010integer}. 
The following result extends the strengthened star inequalities for the special case of a uniform probability distribution by adding new elements that allow variables with negative coefficients to be included in the inequality.

%consider the special case with a uniform probability distribution for the chance constraint, i.e., $\pi_i=1/m$ for $i \in M$ in the definition of $\mac{F}_c$. In this case, $p=\lfloor m \varepsilon \rfloor$. 

\begin{proposition}{\cite[Theorem~4]{luedtke2010integer}}
    \label{prop:chance_lift}
%Let $\bar{\bar{\mac{F}}}_c$ be the special case of $\mac{F}_c$, where the knapsack constraint $\vc{\pi}\tr \vc{x} \leq \varepsilon$ in its description is replaced with the cardinality constraint $\vc{1}\tr \vc{x} \le p$ for some $p \in M$.    
%Suppose $\pi_i=1/m$ for all $i \in M$. 
Consider the mixing set with a cardinality constraint $\mac{F}_c^=$.
Let $r \in \{1,\dotsc,p\}$, $\{t_1,\dotsc, t_{l}\} \subseteq \{1,\dotsc, r\}$, and $\{q_1,\dotsc, q_{p-r}\} \subseteq \{p+1, \dotsc, m\}$, where  $t_1 < \cdots < t_l$ and $q_1 < \cdots <q_{p-r}$. 
Define $\phi_{q_1}=h_{r+1}-h_{r+2}$ and 
\begin{equation*}
    \phi_{q_\iota}=\max\Bigl\{\phi_{q_{\iota-1}}, h_{r+1}- h_{r+\iota+1} - \sum_{k=1}^{\iota-1} \phi_{q_{k}}\Bigr\}, \quad \iota=2, \dotsc, p-r. 
\end{equation*}
Then, the inequalities 
\begin{equation} \label{eq:chance_lift}
    z + \sum_{\iota=1}^{l} (h_{t_{\iota}} - h_{t_{\iota+1}}) x_{t_{\iota}} + \sum_{\iota=1}^{p-r} \phi_{q_{\iota}} (1-x_{q_{\iota}}) \ge h_{t_{1}}, 
\end{equation}
where $h_{t_{l+1}}:=h_{r+1}$, are valid for $\mac{F}_c^=$. 
\end{proposition}

%The inequalities \eqref{eq:chance_lift} are facet-defining for $\conv(\mac{F}_c)$ when $t_1=1$ \cite{luedtke2010integer}. 

% \begin{example}
% \label{ex:2}
% Let's consider the same setup as in Example \ref{ex:1}, with $\pi_i=1/10$ for $i \in M$. 
% Let $r=2$, $T=\{1,2\}$, and $Q=\{5,6\}$. Then, $\phi_{5}=3$ and $\phi_{6}=\max\{3,8-3\}=5$. So, \eqref{eq:chance_lift} yields 
% \begin{equation*}
%     z+ 2 x_1 + 4 x_2 + 3(1-x_5)+ 5(1-x_{6})\ge 20. 
% \end{equation*}
% \end{example}

The following result extends the family of inequalities \eqref{eq:chance_lift} by allowing the permutation of variables with negative coefficients and expanding their index range. 

\begin{proposition}{\cite[Theorem~3]{kuccukyavuz2012mixing}}
    \label{prop:chance_lift_simge}
%Let $\bar{\bar{\mac{F}}}_c$ be the special case of $\mac{F}_c$, where the knapsack constraint $\vc{\pi}\tr \vc{x} \leq \varepsilon$ in its description is replaced with the cardinality constraint $\vc{1}\tr \vc{x} \le p$ for some $p \in M$.    
%Suppose $\pi_i=1/m$ for all $i \in M$. 
Consider the mixing set with a cardinality constraint $\mac{F}_c^=$.
Let $r \in \{1,\dotsc,p\}$ and $\{t_1, \dotsc, t_{l}\} \subseteq \{1,\dotsc, r\}$, where $t_1 < \cdots < t_l$. Select $\{q_1, \dotsc, q_{p-r}\}$ such that $q_\iota \ge r+\iota+1$ for $\iota \in \{1,\dotsc,p-r\}$.  
Define $\phi_{q_1}=h_{r+1}-h_{r+2}$ and 
\begin{equation*}
    \phi_{q_\iota}=\max\Bigl\{\phi_{q_{\iota-1}}, h_{r+1}- h_{r+\iota+1} - \sum_{\substack{k=1\\q_k \ge r+\iota+1}}^{\iota-1} \phi_{q_{k}}\Bigr\}, \quad \iota=2, \dotsc, p-r. 
\end{equation*}
Then, the inequalities 
\begin{equation} \label{eq:chance_lift_simge}
    z + \sum_{\iota=1}^{l} (h_{t_{\iota}} - h_{t_{\iota+1}}) x_{t_{\iota}} + \sum_{\iota=1}^{p-r} \phi_{q_{\iota}} (1-x_{q_{\iota}}) \ge h_{t_{1}}, 
\end{equation}
where $h_{t_{l+1}}:=h_{r+1}$, are valid for $\mac{F}_c^=$. 
\end{proposition}

%{\color{blue} IN CASE WE GENERALIZE \cite{kuccukyavuz2012mixing} IN  THEOREM 3: }

Subsequently, \cite{abdi2016mixing,zhao2017polyhedral} extended the result in Proposition~\ref{prop:chance_lift_simge} to the general-probability setting for $\mac{F}_c$. The next result presents the family of inequalities introduced in \cite[Theorem~1]{zhao2017polyhedral} under Assumption~1 of that paper, which subsumes those derived in \cite[Theorem~14]{abdi2016mixing}. 

\begin{proposition}{\cite[Theorem~1]{zhao2017polyhedral}}
    \label{prop:chance_lift_permutation}
    Consider $r \in \{1,\dotsc,p\}$, $l \in \{1,\dotsc,r\}$, and $v \in \{1,\dotsc, \vartheta-r\}$. %, and define $L:=\{1,\dotsc, l\}$, $N :=\{1,\dotsc, v\}$. 
    Consider a sequence of integers $\{s_0, s_1, \dotsc, s_{v+1}\}$ such that $1\le s_1 \le \dotsc \le s_v \le s_{v+1} =p-r+1$.
    Let $P:=\{t_1, \dotsc, t_{l}\} \subseteq \{1,\dotsc, r\}$, where $t_1 < \cdots < t_l$. Moreover, let $Q: = \{q_1, \dotsc, q_{v}\} \subseteq \{r+s_1+1, \dotsc, m\}$, where $q_\iota \ge r+\min\{1+s_\iota, s_{\iota+1}\}$ and 
    \begin{equation}
        \sum_{k=1}^{r+s_\iota} \pi_k + \sum_{k=\iota}^{v} \pi_{w_k}  > \varepsilon, \quad \sum_{k=1}^{r+s_\iota-1} \pi_k + \sum_{k=\iota}^{v} \pi_{w_k} \le \varepsilon, \label{s} 
    \end{equation} 
    for $\iota \in \{1,\dotsc,v\}$, where $\{w_1, \dotsc, w_v\}$ is a permutation of $Q$ such that $\pi_{w_1} \ge \dotsc \ge \pi_{w_v}$. 
    Define  $\phi_{q_{1}} = h_{r+s_1}-h_{r+s_2}$ and 
%\begin{subequations}
    \begin{equation*}
     %& \phi_{q_{1}} = h_{r+s_1}-h_{r+s_2},\\
     \phi_{q_{\iota}} = \max\Biggl\{\phi_{q_{\iota-1}}, h_{r+s_1}- h_{r+s_\iota+1} - \sum_{\substack{k=1 \\ q_k \ge r + \min\{1+s_\iota, s_{\iota+1}\}}}^{\iota-1} \phi_{q_{k}} \Biggr\}, \quad \iota=2, \dotsc, v.
    \end{equation*}
%\end{subequations}
%with a tie broken according to the order $\{1,\dotsc,p-r\}$.
Then, the inequalities 
\begin{equation} \label{eq:chance_lift_permutation}
    z + \sum_{\iota=1}^{l} (h_{t_{\iota}} - h_{t_{\iota+1}}) x_{t_{\iota}} + \sum_{\iota=1}^{v} \phi_{q_{\iota}} (1-x_{q_{\iota}}) \ge h_{t_{1}}, 
\end{equation}
with $h_{t_{l+1}}:=h_{r+s_1}$,   
are valid for $\mac{F}_c$.

\end{proposition}

%The inequalities in \eqref{eq:chance_lift_permutation} are facet-defining for $\conv(\mac{F}_c)$ if $t_1=1$, $\pi_{q_1} \ge \dotsc \ge \pi_{q_v}$, and 
%\begin{equation*}
%        \sum_{k=1}^{v} \pi_{q_k} + \pi_{j} \le \varepsilon,   
%\end{equation*} 
%for $j \in M \setminus (P \cup Q)$. 
%Moreover, the inequalities in \eqref{eq:chance_lift_permutation} are facet-defining for $\conv(\mac{F}_c)$ only if $t_1=1$ and 
%\begin{equation*}
%        \sum_{k=1}^{r+s_\iota-1} \pi_k + \sum_{k=\iota}^{v} \pi_{w_k} \le \varepsilon, 
%\end{equation*} 
%for $\iota \in N$ with $s_\iota \ge 1$ \cite{zhao2017polyhedral}. 

\smallskip
%{\color{red} Observe that for the special case of a uniform probability distribution, with $p=\tau$, Proposition~\ref{prop:chance_lift_permutation} requires that $s_\iota=p-r-v+\iota$ for $\iota \in [v]$, enforced by \eqref{s}. In particular, for the special case $v=p-r$, Proposition~\ref{prop:chance_lift_permutation} reduces to Proposition~\ref{prop:chance_lift_simge}.} 

The results presented in Propositions~\ref{prop:mixing}--\ref{prop:chance_lift_permutation} illustrate the progression of closed-form valid inequalities for $\mac{F}_c$ over the years, with each subsequent result subsuming the preceding ones as special cases. As seen from these results, the fundamental structure of the inequalities has remained largely unchanged, originating from the star inequality~\eqref{eq:star_ineq}, with gradual extensions achieved by adding new elements to broaden the family.
Although this process has yielded an extensive family of inequalities---whose most general form is presented in Proposition~\ref{prop:chance_lift_permutation}---their structural foundation remains constrained by the core form of the star inequality. Consequently, several important classes of facet-defining inequalities for $\mac{F}_c$ still cannot be represented within the family of \eqref{eq:chance_lift_permutation}, as we illustrate next.

%{\color{red} Maybe we can add the example of \cite{luedtke2010integer} here that shows the shortcoming of their result.}

\smallskip
\begin{example}
\label{ex:1}
Consider an instance of $\mac{F}_c^=$ with $m=10$, $\{h_1, \dotsc, h_m\}=\{20,18,14,11,6,5,4,3,2,1\}$, and $p= \vartheta = 4$. 
Then, it is easy to verify that the constraint 
\begin{equation} \label{eq:Ex1}
    z+ 6 x_1 + 2 x_4 + 3(1-x_5)+ 3(1-x_{6})\ge 20 
\end{equation}
is facet-defining for $\conv(\mac{F}_c^=)$. However, this inequality does not belong to the family of inequalities described in \eqref{eq:chance_lift_permutation}, as outlined next.
Based on the structure of the above inequality, the setup in Proposition~\ref{prop:chance_lift_permutation} would require setting $r = 4$, $P=\{1,4\}$, and $Q=\{5,6\}$.
However, since $p=r=\vartheta$, the conditions in that proposition enforce $Q=\emptyset$, which conflicts with the requirement $Q=\{5,6\}$ for this inequality. 
Moreover, the coefficients of variables in $P$, namely $x_1$ and $x_4$, do not match the values prescribed in Proposition~\ref{prop:chance_lift_permutation}. In particular, the coefficient of $x_1$, which is $6$, does not match $(h_1 - h_4)=9$,  and the coefficient of $x_4$, which is $2$, does not match $(h_4 - h_5)=5$ as $s_1=1$.
As a result, the above inequality exhibits a fundamentally different structure from those in \eqref{eq:chance_lift_permutation}, both in the segment corresponding to variables with positive coefficients and in that corresponding to variables with negative coefficients.
\hfill	$\blacksquare$
\end{example}

\smallskip
Motivated by Example~\ref{ex:1}, in Section~\ref{sec:chance_bilinear}, we develop a unifying convexification framework that not only subsumes many existing families of valid inequalities for $\mac{F}_c$ from the literature but also generates a rich collection of new families of valid inequalities that could not be captured by previous results, such as the inequality in Example~\ref{ex:1}. These new inequalities can constitute a substantial portion of the convex hull, as demonstrated by our preliminary computational experiments in Section~\ref{sec:numerical}.

\section{A Convexification Procedure for Bilinear Sets over a Simplex} \label{sec:BLP}

In this section, we develop a step-by-step aggregation procedure to derive facet-defining inequalities for the convex hull of a bilinear set over a simplex, which contains disjoint bilinear terms in its constraints. In Section~\ref{sec:chance_bilinear}, we then show how $\mac{F}_c$ in \eqref{eq:genprob_chance} can be reformulated as a bilinear set of similar structure, making it amenable to our convexification procedure for generating valid inequalities for its convex hull.

\smallskip
For $N := \{1,\dotsc,n\}$, $M := \{1,\dotsc,m\}$, $K := \{1,\dotsc,\kappa\}$, and $T := \{1, \dotsc, \tau\}$, we consider
\begin{equation}
\mac{S} = \left\{ (\vc{x};\vc{y}) \in \Xi \times \Delta_m \, \middle|\,
\vc{y}\tr A^k \vc{x} + (\vc{b}^k)\tr \vc{x} + (\vc{c}^k)\tr \vc{y} \geq d_k, \, \, \, \forall k \in K
\right\},
\label{eq:generalset}
\end{equation}
where $\Xi = \left\{ \vc{x} \in \Re^n_+ \, \middle| \, E\vc{x} \geq \vc{f} \right\}$, and $\Delta_m = \left\{ \vc{y} \in \Z_+^{m} \, \middle| \, \vc{1}\tr \vc{y} \leq 1 \right\}$ is an integral simplex. 
In this definition, $A^k \in \Re^{m \times n}$, $\vc{b}^k \in \Re^n$, $\vc{c}^k \in \Re^m$, $d_k \in \Re$, $E \in \Re^{\tau \times n}$, $\vc{f} \in \Re^{\tau}$ are parameters of appropriate dimension, $K$ is the index set of bilinear constraints in $\mac{S}$, and $T$ is the index set of linear constraints besides the non-negativity constraints in $\Xi$.

\smallskip
Define $\mac{S}(\bar{\vc{y}})$ to be the restriction of $\mac{S}$ at point $\vc{y} = \bar{\vc{y}}$, i.e.,

\begin{equation*}
	\mac{S}(\bar{\vc{y}}) =
	\left\{
(\vc{x};\vc{y}) \in \Re^{n+m} \, \middle|
	\begin{array}{l}
		\vc{y} = \bar{\vc{y}} \\
		\bar{\vc{y}}\tr \! A^k \vc{x} + (\vc{b}^k)\tr \vc{x} \geq d_k -  (\vc{c}^k)\tr \bar{\vc{y}},  \, \, \, \forall k \in K\\
		E\vc{x} \geq \vc{f}\\
		\vc{x} \geq \vc{0}
	\end{array}
	\right\}.
\end{equation*}

Since $\vc{y}$ is integral in $\Delta_m$, we can write $\mac{S} = \bigcup_{j=0}^m \mac{S}(\vc{e}^j)$, where $\vc{e}^j$ for $j \in M$ is the $j$-th unit vector in $\Re^m$, and $\vc{e}^0$ is the origin in $\Re^m$.
We impose the following assumption for the remainder of this section.

\begin{assumption} \label{asm:cone}
	We have that
	\begin{itemize}
		\item[(i)] $\mac{S}(\vc{e}^j) \neq \emptyset$ for \blue{some} $j \in M \cup \{0\}$,
		\item[(ii)] All sets $\mac{S}(\vc{e}^j)$, for $j \in M \cup \{0\}$, share the same recession cone, which is obtained by setting the right-hand sides of all constraints to zero.
	\end{itemize}
\end{assumption}

Since $\mac{S}$ is represented as a disjunctive union of $\mac{S}(\vc{e}^j)$ for $j \in M \cup \{0\}$, Assumption~\ref{asm:cone} implies that $\conv(\mac{S})$ is a polyhedron; see Section 4.9 in \cite{conforti:co:za:2014} for a detailed discussion of the convex hull of a finite union of polyhedra.
\blue{We further note that the binary nature of the $\vc{y}$ variables guarantees that, under the above assumptions, the convex hull of the bilinear set is a polyhedron. In contrast, this polyhedrality guarantee does not, in general, hold when the $\vc{y}$ variables are continuous.}

\blue{From an application perspective, it is sometimes necessary, as illustrated in Section~4 for the mixing set, to impose the simplex structure in the definition of $\mac{S}$ by introducing additional variables $\vc{y}$, thereby obtaining an extended bilinear formulation. Consequently, our goal is to characterize the convex hull of this lifted set in the space of the original variables $\vc{x}$.}
We refer to the procedure for deriving the convex hull of $\mac{S}$ in the space of variables $\vc{x}$ as \textit{bilinear lift-and-project} (BL\&P).
%For any $l \in K$ and $j \in M \cup \{0\}$, we define a \textit{class-$(l,j)$ \BLP inequality} as an inequality obtained through a weighted aggregation of the constraints in the description of $\mac{S}$ through the following procedure.
We aim to obtain valid inequalities for $\proj_{\vc{x}}\conv(\mac{S})$ through a weighted aggregation of the constraints in the description of $\mac{S}$ according to the following procedure.

\medskip
\begin{enumerate}
	\item[\textit{(A1)}] 
	\textit{We select $\hat{k} \in K$ to be the index of a bilinear constraint in $K$ used in the aggregation.
	We also select $\hat{j} \in M \cup \{0\}$ to indicate that the previously selected constraint $\hat{k}$ has a weight of $y_{\hat{j}}$ if $\hat{j} \in M$, or $1 - \vc{1} \tr \vc{y}$ if $\hat{j} = 0$ in the aggregation.
	In the resulting weighted inequality, we replace $y_i^2$ with $y_i$ for each $i \in M$, and we replace $y_i y_l$ with $0$ for each $i, l \in M$ and $i \neq l$.
	We refer to this weighted constraint as the \textit{base} constraint.}
	\item[\textit{(A2)}] 
	\textit{We select $\mac{K}^j$ for $j \in M \cup \{0\}$ to be subsets of $K$. 
	For each $j \in M$ (resp. $j = 0$) and for each $k \in \mac{K}^j$ such that $(k,j) \neq (\hat{k}, \hat{j})$, we multiply the bilinear constraint $\vc{y}\tr A^k \vc{x} + (\vc{b}^k)\tr \vc{x} + (\vc{c}^k)\tr \vc{y} \geq d_k$ with $\alpha^j_ky_j$ where $\alpha^j_k \geq 0$ (resp. with $\alpha^0_k(1-\vc{1} \tr \vc{y})$ where $\alpha^0_k \geq 0$).
	In the resulting weighted inequality, we replace $y_i^2$ with $y_i$ for each $i \in M$, and we replace $y_i y_l$ with $0$ for each $i, l \in M$ and $i \neq l$.}
	\item[\textit{(A3)}] 
	\textit{We select $\mac{T}^j$ for $j \in M \cup \{0\}$ to be subsets of $T$ whose intersection is empty. 
	For each $j \in M$ (resp. $j = 0$) and for each $t \in \mac{T}^j$, we multiply the constraint $E_{t.}\vc{x} \geq f_t$ with $\beta^j_ty_j$ where $\beta^j_t \geq 0$ (resp. with $\beta^0_t(1-\vc{1} \tr \vc{y})$ where $\beta^0_t \geq 0$).}
\end{enumerate}

\smallskip
The above sets are compactly represented as $\big[\mac{K}^0, \mac{K}^1,\dotsc,\mac{K}^m \big| \mac{T}^0, \mac{T}^1,\dotsc,\mac{T}^m\big]$, which is called a \textit{\BLP assignment}.
Each \BLP assignment is identified by the pair $(\hat{k},\hat{j})$ which represents the base constraint selected in step (A1) of the procedure.
We next aggregate all aforementioned weighted constraints.
%During the aggregation, we require that weights $\vc{\alpha}$ and $\vc{\beta}$ be chosen in such a way that:
%\begin{itemize}
%	\item[\textit{(C1)}] 
%	at least a total of $\sum_{j=0}^m|\mac{K}^j|+|\mac{T}^j|$ bilinear terms together with $y$-variables cancel, i.e., their coefficient becomes zero, and
%	\item[\textit{(C2)}]
%	if $\bigcup_{j =0}^m \mac{K}^j \cup \mac{T}^j \neq \emptyset$, for each constraint used in the aggregation (including the base constraint), at least one bilinear term or $y$-variable among all those created after multiplying that constraint with its corresponding weight is canceled.
%\end{itemize}
In the resulting aggregated inequality, the remaining bilinear terms and $y$-variables are replaced as follows:
\medskip
\begin{itemize}
	\item[\textit{(R1)}] 
	%For each $i \in N$, among all remaining terms $x_i y_j$ for $j \in M$, relax the bilinear term with the most positive coefficient (i.e., the largest coefficient that is also positive) into $x_i$ and relax the rest of these bilinear terms into $0$.
	\textit{For each $i \in N$, let $p_i$ be the most positive coefficient (i.e., the largest coefficient that is also positive) among all remaining terms $x_i y_j$ for $j \in M$, and let $q_i$ be the number of these terms with that coefficient. Set $p_i = q_i = 0$ if such coefficient does not exist. Then, replace the entire remaining terms $x_i y_j$ for $j \in M$ with a single term $p_ix_i$.}
	\item[\textit{(R2)}]
	%Among all remaining variables $y_j$ for $j \in M$, relax the one with the most positive coefficient (i.e., the largest coefficient that is also positive) into $1$ and relax the rest of these variables into $0$.
	\textit{Let $p_0$ be the most positive coefficient (i.e., the largest coefficient that is also positive) among all remaining variables $y_j$ for $j \in M$, and let $q_0$ be the number of these variables with that coefficient. Set $p_0 = q_0 = 0$ if such coefficient does not exist. Then, replace the entire remaining variables $y_j$ for $j \in M$ with the single constant $p_0$.}
	\end{itemize}
    \smallskip
During this procedure, we require that weights $\vc{\alpha}$ and $\vc{\beta}$ be chosen such that:
\medskip
\begin{itemize}
	\item[\textit{(C1)}] 
	\textit{the coefficients of at least $\sum_{j=0}^m \left(|\mac{K}^j|+|\mac{T}^j|\right) + \sum_{i=0}^n \left(\mathbb{I}(q_i) - q_i \right)$ bilinear terms, together with those of the $y$-variables, become zero during the aggregation---excluding the replacements made in (R1) and (R2).}
\end{itemize}
\smallskip
The resulting linear inequality after performing the above steps is referred to as a \textit{class-$(\hat{k},\hat{j})$ \BLP inequality}.

\smallskip
\begin{claim} \label{claim:BLP}
    The \BLP inequalities are valid for $\proj_{\vc{x}} \conv(\mac{S})$, and the collection of all \BLP inequalities is sufficient to describe $\proj_{\vc{x}} \conv(\mac{S})$.
\end{claim}
Our goal in this section is to prove this claim, which is provided in Theorem~\ref{thm:BLP}.
We begin with finding the convex hull description of $\mac{S}$ using disjunctive programming \cite{balas:1979}.

\begin{proposition} \label{prop:disjunctive}
	Define $\mac{Q} = \left\{ (\vc{x}, \vc{y}, \vc{u}^0, \vc{u}^1, \dotsc, \vc{u}^m) \middle| \eqref{eq:disjunctive_1} - \eqref{eq:disjunctive_8} \right\}$, where
	\begin{subequations}
		\begin{align}
			& (\vc{b}^k)\tr \left(\vc{x} - \sum_{j=1}^m \vc{u}^j\right) \geq d_k(1 - \vc{1}\tr \vc{y})	&\forall k \in K \label{eq:disjunctive_1}\\
			&\left(A^k_{j.} + (\vc{b}^k)\tr \right) \vc{u}^j \geq \left(d_k -  \vc{c}^k_j\right)y_j 	&\forall k \in K, \, j \in  M \label{eq:disjunctive_2}\\			
			&E\left(\vc{x} - \sum_{j=1}^m \vc{u}^j\right) \geq \vc{f}(1 - \vc{1}\tr \vc{y}) &	 \label{eq:disjunctive_3}\\
			&E\vc{u}^j \geq \vc{f}y_j		&\forall j \in M \label{eq:disjunctive_4}\\
			&\vc{x} - \sum_{j=1}^m \vc{u}^j  \geq \vc{0}	& \label{eq:disjunctive_5}\\			
			&\vc{u}^j \geq \vc{0}	&\forall j \in M \label{eq:disjunctive_6}\\
			&1 - \vc{1}\tr \vc{y} \geq 0 & \label{eq:disjunctive_7}\\
			&\vc{y} \geq \vc{0} & \label{eq:disjunctive_8}
		\end{align}
	\end{subequations}
	Then, $\proj_{\vc{x}} \conv(\mac{S}) = \proj_{\vc{x}}(\mac{Q})$.
\end{proposition}	

\proof{Proof}
The disjunctive programming formulation of $\mac{S} = \bigcup_{j=0}^m \mac{S}(\vc{e}^j)$ using $\vc{u}^j \in \Re^n$ and $\vc{v}^j \in \Re^m$ as replicas of $\vc{x}$ and $\vc{y}$ in disjunct $j \in M \cup \{0\}$ with convex multiplier $\xi_j$ is as follows:
	\begin{subequations}
		\begin{align}
		&\vc{v}^j = \vc{e}^j \xi_j	&\forall j \in M \cup \{0\} \label{eq:disjunct_1}\\
		&(\vc{e}^j)\tr \! A^k \vc{u}^j + (\vc{b}^k)\tr \vc{u}^j \geq \left(d_k -  (\vc{c}^k)\tr \vc{e}^j\right)\xi_j 	&\forall k \in K, \, j \in  M \cup \{0\} \label{eq:disjunct_2}\\
		&E\vc{u}^j \geq \vc{f}\xi_j		&\forall j \in M \cup \{0\} \label{eq:disjunct_3}\\
		&\vc{u}^j \geq \vc{0}	&\forall j \in M \cup \{0\} \label{eq:disjunct_4}\\
		&\vc{x} = \sum_{j=0}^m \vc{u}^j & \label{eq:disjunct_5}\\
		&\vc{y} = \sum_{j=0}^m \vc{v}^j & \label{eq:disjunct_6}\\
		&\sum_{j=0}^m \xi_j = 1 & \label{eq:disjunct_7}\\
		&\vc{\xi} \geq \vc{0}.  \label{eq:disjunct_8}
	\end{align}
	\end{subequations}
	Let $\mac{P}$ be the feasible region described by the above constraints.
	Under Assumption~\ref{asm:cone}, we have $\conv(\mac{S}) = \proj_{(\vc{x}, \vc{y})} (\mac{P})$; see \cite{conforti:co:za:2014}.
	Using equations \eqref{eq:disjunct_1}, \eqref{eq:disjunct_5} -- \eqref{eq:disjunct_7}, we may substitute $\xi_j$ with $y_j$ for $j \in M$, substitute $\xi_0$ with $1 - \vc{1}\tr \vc{y}$, and replace $\vc{u}^0$ with $\vc{x} - \sum_{j=1}^m \vc{u}^j$ in the rest of the equations to obtain formulation \eqref{eq:disjunctive_1} -- \eqref{eq:disjunctive_8}.	
	\hfill \Halmos
\endproof

\smallskip
Following the interpretation of the new variables added in the disjunctive programming formulation, we may view variables $u^j_i$ in \eqref{eq:disjunctive_1}--\eqref{eq:disjunctive_8} as $x_i y_j$ for each $i \in N$ and $j \in M$. 
To obtain the projection of \eqref{eq:disjunctive_1}--\eqref{eq:disjunctive_8} in the space of $\vc{x}$ variables, we define the vector of dual multipliers $\vc{\pi} = (\vc{\alpha}^0, \dotsc, \vc{\alpha}^m; \vc{\beta}^0, \dotsc, \vc{\beta}^m; \vc{\gamma}^0, \dotsc, \vc{\gamma}^m; \theta_0, \dotsc, \theta_m)$, where $\vc{\alpha}^0 \in \Re_+^{\kappa}$ (resp. $\vc{\alpha}^j \in \Re_+^{\kappa}$) is the vector of dual multipliers associated with constraints \eqref{eq:disjunctive_1} (resp. \eqref{eq:disjunctive_2}), $\vc{\beta}^0 \in \Re_+^{\tau}$ (resp. $\vc{\beta}^j \in \Re_+^{\tau}$) includes the dual multipliers associated with \eqref{eq:disjunctive_3} (resp. \eqref{eq:disjunctive_4}), $\vc{\gamma}^0 \in \Re_+^{n}$ (resp. $\vc{\gamma}^j \in \Re_+^{n}$) contains the dual multipliers corresponding to \eqref{eq:disjunctive_5} (resp. \eqref{eq:disjunctive_6}), and $\theta_0 \in \Re_+$ (resp. $\theta_j \in \Re_+$) is the dual multiplier for \eqref{eq:disjunctive_7} (resp. \eqref{eq:disjunctive_8}). 

\begin{proposition} \label{prop:projection}
	Define $\mac{C} = \left\{ \vc{\pi} \in \Re^{(m+1)(\kappa + \tau + n + 1)}_+ \, \middle| \eqref{eq:cone_1}, \eqref{eq:cone_2} \right\}$, where
	\begin{subequations}
		\begin{align}
			&\sum_{k=1}^{\kappa}  \left((A^k_{j,i} + b^k_i) \alpha^j_k - b^k_i \alpha^0_k \right) + \sum_{t=1}^{\tau} \left( E_{t, i} \beta^j_t - E_{t, i}\beta^0_t \right)  + \gamma^j_i - \gamma^0_i = 0  	& \forall i \in N, \, j \in M  \label{eq:cone_1}\\
			&\sum_{k=1}^{\kappa} \left((c^j_k - d_k) \alpha^j_k + d_k \alpha^0_k \right) + \sum_{t=1}^{\tau} \left( - f_t \beta^j_t + f_t \beta^0_t  \right) + \theta_j - \theta_0 = 0.  & \forall j \in M	\label{eq:cone_2}
		\end{align}
	\end{subequations}
	For any point $\bar{\vc{\pi}} \in \mac{C}$, the inequality
	\begin{equation}
		\left(\sum_{k=1}^{\kappa} \bar{\alpha}^0_k (\vc{b}^k) + E \tr \bar{\vc{\beta}}^0 + \bar{\vc{\gamma}}^0 \right)\tr \vc{x} \geq \sum_{k=1}^{\kappa} \bar{\alpha}^0_k d_k  + (\bar{\vc{\beta}}^0)\tr \vc{f} -\bar{\theta}_0  \label{eq:cone_3}
	\end{equation}
is valid for $\proj_{\vc{x}} \conv(\mac{S})$. Further, \eqref{eq:cone_3} is facet-defining for $\proj_{\vc{x}} \conv(\mac{S})$ only if $\bar{\vc{\pi}}$ is an extreme ray of $\mac{C}$.
\end{proposition}	
	
\proof{Proof}
	It follows from Proposition~\ref{prop:disjunctive} that $\proj_{\vc{x}} \conv(\mac{S}) = \proj_{\vc{x}}(\mac{Q})$, where $\mac{Q}$ is the set defined in that proposition.
	Therefore, we use polyhedral projection to project out variables $\vc{u}$ and $\vc{y}$ in the description of $\mac{Q}$.
	Considering the non-negative dual weights $\vc{\pi}$ as defined previously, we obtain the projection constraint \eqref{eq:cone_1} for each variable $u^j_i$ for $i \in N$ and $j \in M$.
	Similarly, for each variable $y_j$ for $j \in M$, the projection constraint is \eqref{eq:cone_2}.
	These constraints form the projection cone $\mac{C}$ defined in the proposition statement.
	As a result, the weighted combination of the constraints in $\mac{Q}$ with dual weights in $\mac{C}$ yields the valid inequality \eqref{eq:cone_3} for $\proj_{\vc{x}} \conv(\mac{S})$. 
	It is well-known that a facet-defining inequality of the form \eqref{eq:cone_3} corresponds to an extreme ray of the projection cone $\mac{C}$; see \cite{conforti:co:za:2014}.
	\hfill \Halmos
\endproof

\smallskip
We refer to valid inequalities of $\proj_{\vc{x}} \conv(\mac{S})$ that are implied by the constraints describing $\Xi$ as \textit{vertical} inequalities.
As a result, we are interested in obtaining non-vertical valid inequalities for $\proj_{\vc{x}} \conv(\mac{S})$.

\begin{proposition} \label{prop:vertical}
	Let $\bar{\vc{\pi}} \in \mac{C}$ be the dual weights that yield a non-vertical valid inequality of the form \eqref{eq:cone_3} for $\proj_{\vc{x}} \conv(\mac{S})$. 
	Then, $\bar{\alpha}^j_k > 0$ for some $j \in M \cup \{0\}$ and $k \in K$.
\end{proposition}	

\proof{Proof}
	Assume by contradiction that there exists a non-vertical valid inequality of the form \eqref{eq:cone_3} for $\proj_{\vc{x}} \conv(\mac{S})$ obtained from the dual weights $\bar{\vc{\pi}} \in \mac{C}$ such that $\bar{\alpha}^j_k = 0$ for all $j \in M \cup \{0\}$ and $k \in K$.
	Then, the resulting inequality reduces to $\left(E \tr \bar{\vc{\beta}}^0 + \bar{\vc{\gamma}}^0 \right)\tr \vc{x} \geq  (\bar{\vc{\beta}}^0)\tr \vc{f} $.
	This inequality can be obtained from combining inequalities $E \vc{x} \geq \vc{f}$ with weights $\bar{\vc{\beta}}^0$ and inequalities $\vc{x} \geq \vc{0}$ with weights $\bar{\vc{\gamma}}^0$.
	Since all these inequalities are in the description of $\Xi$, we conclude that $\left(E \tr \bar{\vc{\beta}}^0 + \bar{\vc{\gamma}}^0 \right)\tr \vc{x} \geq  (\bar{\vc{\beta}}^0)\tr \vc{f}$ is implied by the constraints in $\Xi$, making it a vertical inequality by definition.
	This is a contradiction to the initial assumption imposed above.
	\hfill \Halmos
\endproof

\smallskip
It follows from Propositions~\ref{prop:projection} and \ref{prop:vertical} that each non-vertical facet-defining inequality in $\proj_{\vc{x}} \conv(\mac{S})$ is of the form \eqref{eq:cone_3} for some extreme ray $\bar{\vc{\pi}} \in \mac{C}$, where $\bar{\alpha}^{\hat{j}}_{\hat{k}} > 0$ for some $\hat{j} \in M \cup \{0\}$ and $\hat{k} \in K$.
Since $\bar{\vc{\pi}}$ is a ray, we may scale it by fixing $\bar{\alpha}^{\hat{j}}_{\hat{k}} = 1$.
As a result, we can view $\bar{\vc{\pi}}$ as an extreme point of the set $\mac{C}_{(\hat{k}, \hat{j})} = \left\{ \vc{\pi} \in \Re^{(m+1)(\kappa + \tau + n + 1)}_+ \, \middle| \, \alpha^{\hat{j}}_{\hat{k}} = 1, \, \eqref{eq:cone_1}, \eqref{eq:cone_2} \right\}$.

Using the results of Propositions~\ref{prop:projection} and \ref{prop:vertical} together with the definition of $\mac{C}_{(\hat{k}, \hat{j})}$, we directly obtain the following result.

\begin{proposition} \label{prop:class}
	The collection of inequalities \eqref{eq:cone_3} obtained from the dual weights $\bar{\vc{\pi}}$ corresponding to the extreme points of $\mac{C}_{(\hat{k}, \hat{j})}$ for all $\hat{k} \in K$ and $\hat{j} \in M \cup \{0\}$ includes all non-vertical facet-defining inequalities in the description of $\proj_{\vc{x}} \conv(\mac{S})$.

    Conversely, any inequality \eqref{eq:cone_3} obtained from the dual weight $\bar{\vc{\pi}}$ corresponding to a feasible solution of $\mac{C}_{(\hat{k}, \hat{j})}$ for some $\hat{k} \in K$ and $\hat{j} \in M \cup \{0\}$ is valid for $\proj_{\vc{x}} \conv(\mac{S})$.
	%\Halmos
\end{proposition}	

\smallskip
\blue{Next, we identify structural properties of the extreme points of $\mac{C}_{(\hat{k}, \hat{j})}$ that can facilitate the derivation of \textit{useful} valid inequalities for $\proj_{\vc{x}} \conv(\mac{S})$.}

\begin{proposition} \label{prop:ext_point_1}
	Let $\bar{\vc{\pi}}$ be an extreme point of $\mac{C}_{(\hat{k}, \hat{j})}$ for some $\hat{k} \in K$ and $\hat{j} \in M \cup \{0\}$. Then there exist $\mac{K}^j \subseteq K$ and $\mac{T}^j \subseteq T$ for all $j \in M \cup \{0\}$, and $U_1 \subseteq N \times M$ and $U_2 \subseteq M$ such that:
	\begin{itemize}
		\item[(i)] for each $j \in M \cup \{0\}$, $\bar{\alpha}^j_k > 0$ for $k \in \mac{K}^j$, and $\bar{\alpha}^j_k = 0$ otherwise, except for $\bar{\alpha}^{\hat{j}}_{\hat{k}}$ which is 1.
		\item[(ii)] $\hat{k} \notin \mac{K}^{\hat{j}}$,		
		\item[(iii)] for each $j \in M \cup \{0\}$, $\bar{\beta}^j_t > 0$ for $t \in \mac{T}^j$, and $\bar{\beta}^j_t = 0$ otherwise,
		\item[(iv)] $\bigcap_{j=0}^m \mac{T}^j = \emptyset$,
		
		\item[(v)] $\Phi_{i,j} = 0$ for all $(i, j) \in U_1$, where
		\begin{align*}
			\Phi_{i,j} & =  \sum_{k \in K: k=\hat{k}, j=\hat{j}} (A^k_{j,i} + b^k_i) - \sum_{k \in K: k=\hat{k}, \hat{j} = 0} b^k_i + \sum_{k \in \mac{K}^j}  (A^k_{j,i} + b^k_i) \bar{\alpha}^j_k - \sum_{k \in \mac{K}^0} b^k_i \bar{\alpha}^0_k \\ 
            & \qquad{} + \sum_{t \in \mac{T}^j} E_{t, i} \bar{\beta}^j_t - \sum_{t \in \mac{T}^0} E_{t, i}\bar{\beta}^0_t 
		\end{align*}
		%$\sum_{k=1}^{\kappa}  \left((A^k_{j,i} + b^k_i) \bar{\alpha}^j_k - b^k_i \bar{\alpha}^0_k \right) + \sum_{t=1}^{\tau} \left( E_{t, i} \bar{\beta}^j_t - E_{t, i}\bar{\beta}^0_t \right) = 0$ for all $(i, j) \in U_1$,
		
		\item[(vi)] $\Psi_j = 0$ for all $j \in U_2$, where
		\begin{equation*}
			\Psi_j = \sum_{k \in K: k=\hat{k}, j=\hat{j}} (c^j_k - d_k) + \sum_{k \in K: k=\hat{k}, \hat{j} = 0} d_k + \sum_{k \in \mac{K}^j} (c^j_k - d_k) \bar{\alpha}^j_k + \sum_{k \in \mac{K}^0} d_k \bar{\alpha}^0_k - \sum_{t \in \mac{T}^j} f_t \bar{\beta}^j_t + \sum_{t \in \mac{T}^0} f_t \bar{\beta}^0_t 
		\end{equation*}
		%$\sum_{k=1}^{\kappa} \left((c^j_k - d_k) \bar{\alpha}^j_k + d_k \bar{\alpha}^0_k \right) + \sum_{t=1}^{\tau} \left( - f_t \bar{\beta}^j_t + f_t \bar{\beta}^0_t \right) = 0$ for all $j \in U_2$,
		
		\item[(vii)] for each $i \in N$, $\bar{\gamma}^0_i = \max_{j \in M} \left\{ (\Phi_{i,j})^+ \right\}$,
		%$\bar{\gamma}^0_i = \max_{j \in M} \left\{ \left( \sum_{k=1}^{\kappa}  (A^k_{j,i} + b^k_i) \bar{\alpha}^j_k - b^k_i \bar{\alpha}^0_k + \sum_{t=1}^{\tau} E_{t, i} \bar{\beta}^j_t - E_{t, i}\bar{\beta}^0_t \right)^+ \right\}$,
		
		\item[(viii)] for each $i \in N$ and $j \in M$, $\bar{\gamma}^j_i = \bar{\gamma}^0_i - \Phi_{i,j}$,
		%$\bar{\gamma}^j_i = \bar{\gamma}^0_i - \sum_{k=1}^{\kappa}  \left((A^k_{j,i} + b^k_i) \bar{\alpha}^j_k - b^k_i \bar{\alpha}^0_k \right) - \sum_{t=1}^{\tau} \left( E_{t, i} \bar{\beta}^j_t - E_{t, i}\bar{\beta}^0_t \right)$,
		
		\item[(ix)] $\bar{\theta}_0 = \max_{j \in M} \left\{ (\Psi_j)^+ \right\}$,
		%$\bar{\theta}_0 = \max_{j \in M} \left\{  \left( \sum_{k=1}^{\kappa} d_k \bar{\alpha}^0_k + (c^j_k - d_k) \bar{\alpha}^j_k + \sum_{t=1}^{\tau} f_t \bar{\beta}^0_t - f_t \bar{\beta}^j_t  \right)^+ \right\}$,
		
		\item[(x)] for each $j \in M$, $\bar{\theta}_j = \bar{\theta}_0 - \Psi_j$,
		%$\bar{\theta}_j = \bar{\theta}_0 - \sum_{k=1}^{\kappa} \left(d_k \bar{\alpha}^0_k + (c^j_k - d_k) \bar{\alpha}^j_k \right) - \sum_{t=1}^{\tau} \left( f_t \bar{\beta}^0_t - f_t \bar{\beta}^j_t \right)$,
		
		\item[(xi)] the valid inequality of $\proj_{\vc{x}} \conv(\mac{S})$ produced by $\bar{\vc{\pi}}$ is of the following form
		\begin{equation*}
			\sum_{i=1}^n \left( \sum_{k \in K: k=\hat{k}, \hat{j} = 0} b^k_i + \sum_{k \in \mac{K}^0} \bar{\alpha}^0_{k} b^k_i + \sum_{t \in \mac{T}^0} \bar{\beta}^0_{t} E_{t,i} + \bar{\gamma}^0_i \right) x_i \geq \sum_{k \in K: k=\hat{k}, \hat{j} = 0} d_k + \sum_{k \in \mac{K}^0} \bar{\alpha}^0_{k} d_k + \sum_{t \in \mac{T}^0} \bar{\beta}^0_{t} f_t - \bar{\theta}_0
		\end{equation*}
		
		\item[(xii)] $|U_1| + |U_2| \geq \sum_{j=0}^m|\mac{K}^j|+|\mac{T}^j| + \sum_{i=1}^n \mathbb{I}(\bar{\gamma}^0_i) - \sum_{i=1}^n\sum_{j=1}^m \mathbb{I}(\bar{\gamma}^0_i)(1 - \mathbb{I}(\bar{\gamma}^j_i)) + \mathbb{I}(\bar{\theta}_0) - \sum_{j=1}^m \mathbb{I}(\bar{\theta}_0)(1 - \mathbb{I}(\bar{\theta}_j))$		
	\end{itemize}

    Conversely, assume that $\bar{\vc{\pi}} \in \Re^{(m+1)(\kappa + \tau + n + 1)}$ satisfies conditions (i)--(xii) above for some $\hat{k} \in K$, $\hat{j} \in M \cup \{0\}$, $\mac{K}^j \subseteq K$, $\mac{T}^j \subseteq T$ for all $j \in M \cup \{0\}$, $U_1 \subseteq N \times M$ and $U_2 \subseteq M$.
    Then, $\bar{\vc{\pi}} \in \mac{C}_{(\hat{k},\hat{j})}$.
\end{proposition}		

\proof{Proof}
The coefficient matrix for the system of equations describing $\mac{C}_{(\hat{k}, \hat{j})}$ corresponding to the basic solution $\bar{\vc{\pi}}$ after appropriate rearrangement of rows and columns would be as follows:
	
\begin{gather} \label{eq:BLP matrix}
	\left[
	\def\arraystretch{2.5}
	\begin{array}{c||c|c|c|c||c|c||c}
		1 \, & \, 0 \, & \, 0 \, & \, 0 \, & \, 0 \, & \, 0 \, & \, 0 \, & \, 0 \, \\
		\hline
		\hline
		X_{(2, 1)} \, & \, X_{(2, 2)} \, & \, X_{(2, 3)} \, & \, X_{(2, 4)} \, & \, 0 \, & \, 0 \, & \, 0 \, & \, X_{(2, 8)} \, \\
		\hline
		X_{(3, 1)} \, & \, X_{(3, 2)} \, & \, X_{(3, 3)} \, & \, 0 \, & \, X_{(3, 5)} \, & \, 0 \, & \, 0 \, & \, X_{(3, 8)} \, \\
		\hline
		\hline
		X_{(4, 1)} \, & \, X_{(4, 2)} \, & \, X_{(4, 3)} \, & \, X_{(4, 4)} \, & \, 0 \, & \quad I \quad & \, 0 \, & \, X_{(4, 8)} \, \\
		\hline
		X_{(5, 1)} \, & \, X_{(5, 2)} \, & \, X_{(5, 3)} \, & \, 0 \, & \, X_{(5, 5)} \, & \, 0 \, & \quad I \quad & \, X_{(5, 8)} \, \\
		\hline
		\hline
		0 \, & \, 0 \, & \, 0 \, & \, 0 \, & \, 0 \, & \, 0 \, & \, 0 \, & \, X_{(6, 8)} \, \\
	\end{array}
	\right], 
\end{gather}
where $X_{(p,q)}$ is the submatrix of appropriate dimension in row block $p$ and column block $q$, and $0$ and $I$ are the zero and identity submatrices of appropriate dimension. 

\smallskip
In the matrix \eqref{eq:BLP matrix}, the first column corresponds to variable $\bar{\alpha}^{\hat{j}}_{\hat{k}}$.
The second and third column blocks are respectively associated with variables $\bar{\alpha}^j_k$ for $(k,j) \in K \times M \cup \{0\} \setminus \left\{(\hat{k},\hat{j})\right\}$ and $\bar{\beta}^j_t$ for $(t,j) \in T \times M \cup \{0\}$ that have positive values in the basic solution.
Similarly, the fourth and fifth column blocks are respectively associated with variables $\bar{\gamma}^0_i$ for $i \in N$ and $\bar{\theta}_0$ that have positive values in the basic solution.
The next two column blocks respectively correspond to variables $\bar{\gamma}^j_i$ for $(i,j) \in N \times M$ and $\bar{\theta}_j$ for $j \in M$ with positive values in the basic solution. 	
The last column block represents the remaining basic variables that appear with value zero in the basic solution.

\smallskip
In the basis matrix \eqref{eq:BLP matrix}, the first row shows the coefficients of the first constraint in the description of $\mac{C}_{(\hat{k}, \hat{j})}$, i.e., $\bar{\alpha}^{\hat{j}}_{\hat{k}} = 1$.
The second row block corresponds to constraints \eqref{eq:cone_1} for all $(i, j) \in N \times M$ that have a non-zero coefficient for at least one positive variable among $\bar{\alpha}^0_k$, $\bar{\alpha}^j_k$ for $k \in K$, $\bar{\beta}^0_t$, $\bar{\beta}^j_t$ for $t \in T$, and $\bar{\gamma}^0_i$.
This definition implies that for each row in the second row block, there is a non-zero element in the same row of at least one of the submatrices $X_{(2,2)}$, $X_{(2,3)}$, and $X_{(2,4)}$. 
Similarly, the third row block corresponds to constraints \eqref{eq:cone_2} for all $j \in M$ that have a non-zero coefficient for at least one positive variable among $\bar{\alpha}^0_k$, $\bar{\alpha}^j_k$ for $k \in K$, $\bar{\beta}^0_t$, $\bar{\beta}^j_t$ for $t \in T$, and $\bar{\theta}_0$.
This definition implies that for each row in the third row block, there is a non-zero element in the same row of at least one of the submatrices $X_{(3,2)}$, $X_{(3,3)}$, and $X_{(3,5)}$.
The fourth row block contains the coefficients of constraints \eqref{eq:cone_1} for all $(i, j) \in N \times M$ such that $\bar{\gamma}^j_i > 0$.
As a result, the submatrix representing the coefficients of $\bar{\gamma}^j_i$ in these constraints forms an identity matrix as shown in \eqref{eq:BLP matrix}. 
Similarly, the fifth row block contains the coefficients of constraints \eqref{eq:cone_2} for all $j \in M$ such that $\bar{\theta}_j > 0$.
As a result, the submatrix representing the coefficients of $\bar{\theta}_j$ in these constraints forms an identity matrix as shown in \eqref{eq:BLP matrix}. 
The last row block corresponds to all the remaining constraints in \eqref{eq:cone_1} and \eqref{eq:cone_2} that are not included in the previous row blocks.

\smallskip
To prove part (i), for each $j \in M \cup \{0\}$, define $\mac{K}^j$ to include constraints $k \in K$ whose aggregation weights $\bar{\alpha}^j_k$ correspond to a column in the second column block. 
By definition, these variables are positive, and the remaining variables that are not associated with columns of the second column block are zero.
Because of the structure of the matrix \eqref{eq:BLP matrix}, the column corresponding to $\bar{\alpha}^{\hat{j}}_{\hat{k}}$ is not among those in the second column block, because it is represented in the first column.
As a result, $\hat{k} \notin \mac{K}^{\hat{j}}$, which verifies part (ii).
To prove part (iii), for each $j \in M \cup \{0\}$, define $\mac{T}^j$ to include constraints $t \in T$ whose aggregation weights $\bar{\beta}^j_t$ correspond to a column in the third column block. 
By definition, these variables are positive, and the remaining variables that are not associated with columns of the third column block are zero.
To show part (iv), assume by contradiction that there exists $t^* \in \bigcap_{j=0}^m \mac{T}^j$. 
Therefore, the columns corresponding to the coefficients of variables $\bar{\beta}^j_{t^*}$ for all $j \in M \cup \{0\}$ are included in the third column block of \eqref{eq:BLP matrix}.
It follows from the description of the constraints \eqref{eq:cone_1} and \eqref{eq:cone_2} in $\mac{C}_{(\hat{k}, \hat{j})}$ that these columns linearly dependent as their summation yields the zero vector.
This is a contradiction to the fact that \eqref{eq:BLP matrix} is a basis matrix.

\smallskip
For part (v), we define $U_1$ to be the set of $(i,j) \in N \times M$ corresponding to constraints \eqref{eq:cone_1} in rows of the second row block of \eqref{eq:BLP matrix} such that $\bar{\gamma}^0_i = 0$.
As a result, \eqref{eq:cone_1} implies that $\sum_{k=1}^{\kappa}  \left((A^k_{j,i} + b^k_i) \bar{\alpha}^j_k - b^k_i \bar{\alpha}^0_k \right) + \sum_{t=1}^{\tau} \left( E_{t, i} \bar{\beta}^j_t - E_{t, i}\bar{\beta}^0_t \right) = 0$ for all $(i, j) \in U_1$.
Using the results of parts (i)--(iii), we may rewrite the left-hand-side of this equation as $\sum_{k \in K: k=\hat{k}, j=\hat{j}} (A^k_{j,i} + b^k_i) - \sum_{k \in K: k=\hat{k}, \hat{j} = 0} b^k_i + \sum_{k \in \mac{K}^j}  (A^k_{j,i} + b^k_i) \bar{\alpha}^j_k - \sum_{k \in \mac{K}^0} b^k_i \bar{\alpha}^0_k + \sum_{t \in \mac{T}^j} E_{t, i} \bar{\beta}^j_t - \sum_{t \in \mac{T}^0} E_{t, i}\bar{\beta}^0_t$, which is equal to $\Phi_{i,j}$ defined in part (v).
Similarly for part (vi), we define $U_2$ to be the set of $j \in M$ corresponding to constraints \eqref{eq:cone_2} in rows of the third row block of \eqref{eq:BLP matrix} if $\bar{\theta}_0 = 0$, otherwise $U_2 = \emptyset$.
As a result, \eqref{eq:cone_2} implies that $\sum_{k=1}^{\kappa} \left((c^j_k - d_k) \bar{\alpha}^j_k + d_k \bar{\alpha}^0_k \right) + \sum_{t=1}^{\tau} \left( - f_t \bar{\beta}^j_t + f_t \bar{\beta}^0_t \right) = 0$ for all $j \in U_2$.
Using the results of parts (i)--(iii), we may rewrite the left-hand-side of this equation as $\sum_{k \in K: k=\hat{k}, j=\hat{j}} (c^j_k - d_k) + \sum_{k \in K: k=\hat{k}, \hat{j} = 0} d_k + \sum_{k \in \mac{K}^j} (c^j_k - d_k) \bar{\alpha}^j_k + \sum_{k \in \mac{K}^0} d_k \bar{\alpha}^0_k - \sum_{t \in \mac{T}^j} f_t \bar{\beta}^j_t + \sum_{t \in \mac{T}^0} f_t \bar{\beta}^0_t$, which is equal to $\Psi_j$ defined in part (vi).

\smallskip
For part (vii), consider $i \in N$. 
On the one hand, it follows from \eqref{eq:cone_1} that the vectors containing the coefficient of $\gamma^j_i$ for all $j \in M \cup \{0\}$ are linearly dependent as their summation yields the zero vector.
As a result, all these vectors cannot be included in the basis matrix \eqref{eq:BLP matrix}.
Therefore, $\bar{\gamma}^j_i = 0$ for some $j \in M \cup \{0\}$.
On the other hand, the constraints \eqref{eq:cone_1} for fixed $i \in N$ imply that $\bar{\gamma}^0_i \geq \sum_{k=1}^{\kappa}  (A^k_{j,i} + b^k_i) \bar{\alpha}^j_k - b^k_i \bar{\alpha}^0_k + \sum_{t=1}^{\tau} E_{t, i} \bar{\beta}^j_t - E_{t, i}\bar{\beta}^0_t$ for each $j \in M$ because $\bar{\gamma}^j_i \geq 0$ by definition of $\mac{C}_{(\hat{k}, \hat{j})}$.
Using the results of parts (i)--(iii), we may rewrite the right-hand-side of this inequality as $\Phi_{i,j}$ given previously.
Also, since $\bar{\vc{\pi}}$ is a feasible solution of $\mac{C}_{(\hat{k}, \hat{j})}$, we have that $\bar{\gamma}^0_i \geq 0$.
Thus, we can combine the above inequalities as $\bar{\gamma}^0_i \geq \max_{j \in M} \left\{ (\Phi_{i,j})^+ \right\}$.
Because of the previous argument, this inequality must be satisfied at equality since otherwise $\bar{\gamma}^j_i$ would be strictly positive for all $j \in M \cup \{0\}$. 
Once $\bar{\gamma}^0_i$ is calculated, we may use the equations in \eqref{eq:cone_1} to calculate $\bar{\gamma}^j_i = \bar{\gamma}^0_i - \sum_{k=1}^{\kappa}  \left((A^k_{j,i} + b^k_i) \bar{\alpha}^j_k - b^k_i \bar{\alpha}^0_k \right) - \sum_{t=1}^{\tau} \left( E_{t, i} \bar{\beta}^j_t - E_{t, i}\bar{\beta}^0_t \right) = \bar{\gamma}^0_i - \Phi_{i,j}$ for all $j \in M$, verifying the equation in part (viii).
The proof of parts (ix) and (x) \blue{follows} from similar arguments to those of parts (vii) and (viii) when applied to \eqref{eq:cone_2} and variables $\theta_j$ for $j \in M \cup \{0\}$.
The inequality form in part (xi) is obtained by using the results of parts (i)--(iii) to properly separate the terms in the inequality \eqref{eq:cone_3}.

\smallskip
For part (xii), since the submatrices in position $(4,6)$ and $(5, 7)$ of matrix \eqref{eq:BLP matrix} are identity matrices, we can use basic column operations to make all elements of the submatrices $X_{(4, 2)}$, $X_{(4, 3)}$, $X_{(4, 4)}$, $X_{(5, 2)}$, $X_{(5, 3)}$, and $X_{(5, 5)}$ zero.
As a result, the nonzero elements of column blocks 2--5 are in rows included in the row blocks 2 and 3, forming the following submatrix:   

\begin{gather} \label{eq:BLP submatrix}
	\left[
	\def\arraystretch{2.5}
	\begin{array}{c|c|c|c}
		X_{(2, 2)} \, & \, X_{(2, 3)} \, & \, X_{(2, 4)} \, & \, 0 \, \\
		\hline
		X_{(3, 2)} \, & \, X_{(3, 3)} \, & \, 0 \, & \, X_{(3, 5)} \, \\
	\end{array}
	\right].
\end{gather}

Since the column vectors in this submatrix must be linearly independent due to being included in the basis matrix, the number of its columns cannot exceed the number of its rows.
We calculate the number of columns as follows.
By definition of sets $\mac{K}^j$, the number of columns in the first column block of \eqref{eq:BLP submatrix} is equal to $\sum_{j=0}^m|\mac{K}^j|$.
Similarly, the definition of $\mac{T}^j$ implies that the number of columns in the second column block of \eqref{eq:BLP submatrix} is equal to $\sum_{j=0}^m|\mac{T}^j|$.
It follows from the definition of the columns included in the third column block that the number of these columns is equal to the number of variables $\bar{\gamma}^0_i$ for $i \in N$ that are strictly positive, which can be written as $\sum_{i=1}^n \mathbb{I}(\bar{\gamma}^0_i)$.
Using a similar argument, we can calculate the number of columns in the fourth column block as $\mathbb{I}(\bar{\theta}_0)$.
Next, we calculate the number of rows in \eqref{eq:BLP submatrix}.
It follows from the structure of \eqref{eq:BLP matrix} that the rows associated with pair $(i,j) \in N \times M$ in the second row block must have $\bar{\gamma}^j_i = 0$.
By definition of $U_1$ given above, these rows are partitioned into those with $(i,j) \in U_1$ where $\bar{\gamma}^0_i = 0$, and those with $(i,j) \notin U_1$ where $\bar{\gamma}^0_i > 0$.
The latter partition can be described as $\{(i,j) \in N \times M | \bar{\gamma}^j_i = 0, \, \bar{\gamma}^0_i > 0\}$.
We may calculate the cadinality of this set as $\sum_{i=1}^n\sum_{j=1}^m \mathbb{I}(\bar{\gamma}^0_i)(1 - \mathbb{I}(\bar{\gamma}^j_i))$.
Therefore, the number of rows in the first row block of \eqref{eq:BLP submatrix} is equal to $|U_1| + \sum_{i=1}^n\sum_{j=1}^m \mathbb{I}(\bar{\gamma}^0_i)(1 - \mathbb{I}(\bar{\gamma}^j_i))$.
Using a similar argument for the second row block of \eqref{eq:BLP submatrix}, we can calculate the number of its rows as $|U_2| + \sum_{j=1}^m \mathbb{I}(\bar{\theta}_0)(1 - \mathbb{I}(\bar{\theta}_j))$.
Consequently, using the argument above, the number of rows in \eqref{eq:BLP submatrix} must be no less than the number of its columns, i.e., $|U_1| + \sum_{i=1}^n\sum_{j=1}^m \mathbb{I}(\bar{\gamma}^0_i)(1 - \mathbb{I}(\bar{\gamma}^j_i)) + |U_2| + \sum_{j=1}^m \mathbb{I}(\bar{\theta}_0)(1 - \mathbb{I}(\bar{\theta}_j)) \geq \sum_{j=0}^m|\mac{K}^j|+|\mac{T}^j| + \sum_{i=1}^n \mathbb{I}(\bar{\gamma}^0_i) + \mathbb{I}(\bar{\theta}_0)$, yielding the result in part (xi).

%\medskip
Next, we prove the converse statement.
It follows from conditions (i)--(iv), and (vii)--(x) that all components of $\bar{\vc{\pi}}$ are non-negative and $\bar{\alpha}^{\hat{j}}_{\hat{k}} =  1$.
As established earlier in this proof, conditions (vii)--(x) together with the equations in (v) and (vi) imply that $\bar{\vc{\pi}}$ satisfies equations \eqref{eq:cone_1} and \eqref{eq:cone_2}. 
We conclude that $\bar{\vc{\pi}} \in \mac{C}_{(\hat{k}, \hat{j})}$.
\hfill \Halmos	
\endproof

\medskip
We are now ready to prove the main result of this section outlined in Claim~\ref{claim:BLP}.
\blue{In particular, we show that the structural properties of the extreme points of $\mac{C}_{(\hat{k}, \hat{j})}$ obtained in Proposition~\ref{prop:ext_point_1} translate into the steps of the \BLP procedure to derive the facet-defining inequalities for $\proj_{\vc{x}} \conv(\mac{S})$.}
We further show that any \BLP inequality is valid for $\proj_{\vc{x}} \conv(\mac{S})$.
The combination of these statements \blue{leads} to our main result, as presented next. 

\begin{comment}
\begin{propisition} \label{prop:BLP steps}
	Let $\bar{\vc{\pi}}$ be an extreme point of $\mac{C}_{(\hat{k}, \hat{j})}$ for some $\hat{k} \in K$ and $\hat{j} \in M \cup \{0\}$. 
	Let \eqref{eq:cone_3} be the valid inequality for $\proj_{\vc{x}} \conv(\mac{S})$ obtained from $\bar{\vc{\pi}}$.
	Then, this inequality can be derived as a class-$(\hat{k},\hat{j})$ \BLP inequality with the \BLP assignment $\big[\mac{K}^0, \mac{K}^1,\dotsc,\mac{K}^m \big| \mac{T}^0, \mac{T}^1,\dotsc,\mac{T}^m\big]$ where $\mac{K}^j$ and $\mac{T}^j$ are defined according to Proposition~\ref{prop:ext_point_1}.  
\end{propisition}	
\end{comment}

\begin{theorem} \label{thm:BLP}
The collection of all \BLP inequalities together with $\Xi$ describes $\proj_{\vc{x}} \conv(\mac{S})$.
    %Any non-vertical facet-defining inequality of $\proj_{\vc{x}} \conv(\mac{S})$ can be obtained as a \BLP inequality.
\end{theorem}

\proof{Proof}
To prove the result, we first show that any non-vertical facet-defining inequality of $\proj_{\vc{x}} \conv(\mac{S})$ can be obtained as a \BLP inequality.
	Let $g(\vc{x}) \geq g_0$ be a non-vertical facet-defining inequality of $\proj_{\vc{x}} \conv(\mac{S})$.
	It follows from Proposition~\ref{prop:class} that $g(\vc{x}) \geq g_0$ is of the form \eqref{eq:cone_3} for some extreme point $\bar{\vc{\pi}}$ of $\mac{C}_{(\hat{k}, \hat{j})}$ for some $\hat{k} \in K$ and $\hat{j} \in M \cup \{0\}$.
	Proposition~\ref{prop:ext_point_1} implies that there exist $\mac{K}^j \subseteq K$ and $\mac{T}^j \subseteq T$ for all $j \in M \cup \{0\}$, and $U_1 \subseteq N \times M$ and $U_2 \subseteq M$ that satisfy properties (i)--(xii).
	We use these sets and properties to define the proper parameters for the \BLP procedure that produces $g(\vc{x}) \geq g_0$.
	We begin with step (A1), where we choose $\hat{k}$ to be the index of the bilinear constraint and $\hat{j}$ to represent the weight of this constraint in the aggregation. 
	The pair $(\hat{k}, \hat{j})$ defines the class of the desired \BLP inequality.  
	According to this step, the weighted constraint will be of the form
	\begin{align*}
		&\sum_{i=1}^n b_i^{\hat{k}} x_i - \sum_{i=1}^n \sum_{j=1}^m b_i^{\hat{k}} x_i y_{j} \geq d_{\hat{k}} - \sum_{j=1}^m d_{\hat{k}} y_{j}, & \text{if }\hat{j} = 0		\\ %\label{eq:agg1}\\
		&\sum_{i=1}^n A^{\hat{k}}_{\hat{j},i} x_iy_{\hat{j}} + \sum_{i=1}^n b_i^{\hat{k}} x_i y_{\hat{j}} + c^{\hat{k}}_{\hat{j}} y_{\hat{j}} \geq d_{\hat{k}} y_{\hat{j}}, & \text{if }\hat{j} \in M.	%%\label{eq:agg2}
	\end{align*}

	For step (A2), we choose sets $\mac{K}^j$ for $j \in M \cup \{0\}$ to be those defined in Proposition~\ref{prop:ext_point_1} with aggregation weights $\bar{\alpha}^j_k$ for each $k \in \mac{K}^j$.
	These sets satisfy the conditions of (A2) because of properties (i) and (ii) of Proposition~\ref{prop:ext_point_1}.
	This yields the weighted constraints of the form
	\begin{align*}
		&\sum_{i=1}^n \alpha^0_kb_i^k x_i - \sum_{i=1}^n \sum_{j=1}^m \alpha^0_kb_i^k x_i y_j \geq \alpha^0_kd_k - \sum_{j=1}^m \alpha^0_kd_k y_j, & \forall k \in \mac{K}^0	%\label{eq:agg3}
        \\
		&\sum_{i=1}^n \alpha^j_k A^k_{j,i} x_iy_{j} + \sum_{i=1}^n \alpha^j_k b_i^k x_i y_j + \alpha^j_k c^k_j y_j \geq \alpha^j_k d_k y_j, & \forall k \in \mac{K}^j, \forall j \in M.	%\label{eq:agg4}
	\end{align*}

	For step (A3), we choose sets $\mac{T}^j$ for $j \in M \cup \{0\}$ to be those defined in Proposition~\ref{prop:ext_point_1} with aggregation weights $\bar{\beta}^j_t$ for each $t \in \mac{T}^j$.
	These sets satisfy the conditions of (A3) because of properties (iii) and (iv) of Proposition~\ref{prop:ext_point_1}.
	This yields the weighted constraints of the form
	\begin{align*}
		&\sum_{i=1}^n \beta^0_t E_{t,i} x_i - \sum_{i=1}^n \sum_{j=1}^m \beta^0_t E_{t,i} x_i y_j \geq \beta^0_t f_t - \sum_{j=1}^m \beta^0_t f_t y_j, & \forall t \in \mac{T}^0	%\label{eq:agg5}
        \\
		&\sum_{i=1}^n \beta^0_t E_{t,i} x_i y_j \geq \beta^0_t f_t y_j, & \forall t \in \mac{T}^j, \forall j \in M.	%\label{eq:agg6}
	\end{align*}
	
	According to the \BLP procedure, we aggregate the above inequalities to obtain
	\begin{align}
		& \sum_{i \in N} \sum_{j \in M} \Bigg( - \sum_{k \in K: k=\hat{k}, \hat{j} = 0} b^k_i + \sum_{k \in K: k=\hat{k}, j=\hat{j}} (A^k_{j,i} + b^k_i) - \sum_{k \in \mac{K}^0} \bar{\alpha}^0_k b^k_i  + \sum_{k \in \mac{K}^j}  \bar{\alpha}^j_k (A^k_{j,i} + b^k_i) \nonumber \\ 
        & \qquad \qquad {}  - \sum_{t \in \mac{T}^0} \bar{\beta}^0_t E_{t, i} + \sum_{t \in \mac{T}^j} \bar{\beta}^j_t E_{t, i} \Bigg) x_i y_j \nonumber\\
		& \quad{} + \sum_{j \in M} \Bigg( \sum_{k \in K: k=\hat{k}, \hat{j} = 0} d_k + \sum_{k \in K: k=\hat{k}, j=\hat{j}} (c^j_k - d_k) + \sum_{k \in \mac{K}^0} \bar{\alpha}^0_k d_k + \sum_{k \in \mac{K}^j} \bar{\alpha}^j_k (c^j_k - d_k) \nonumber \\
        & \qquad \qquad {}  + \sum_{t \in \mac{T}^0} \bar{\beta}^0_t f_t  - \sum_{t \in \mac{T}^j} \bar{\beta}^j_t f_t  \Bigg) y_j \nonumber\\
		& \quad{}  + \sum_{i \in N} \left( \sum_{k \in K: k=\hat{k}, \hat{j} = 0} b^k_i + \sum_{k \in \mac{K}^0} \bar{\alpha}^0_{k} b^k_i + \sum_{t \in \mac{T}^0} \bar{\beta}^0_{t} E_{t,i} \right) x_i \geq \sum_{k \in K: k=\hat{k}, \hat{j} = 0} d_k + \sum_{k \in \mac{K}^0} \bar{\alpha}^0_{k} d_k + \sum_{t \in \mac{T}^0} \bar{\beta}^0_{t} f_t. \label{eq:aggregated}
	\end{align}
	
	It follows from property (v) of Proposition~\ref{prop:ext_point_1} that the coefficient of $x_iy_j$ in the above aggregated inequality is equal to $\Phi_{i,j}$ defined therein.
	Among these terms, those with $(i,j) \in U_1$ will have coefficient zero according to the \BLP procedure.
	The remaining bilinear terms $x_i y_j$ with $(i,j) \notin U_1$ are replaced following the step (R1) in the \BLP procedure.
	In particular, for each $i \in N$, we calculate the most positive coefficient of the remaining terms $x_i y_j$ as $p_i = \max_{j \in M: (i,j) \notin U_1} \{ \Phi_{i,j} \}$.
	If $p_i < 0$, we set it equal to zero.
	Since $\Phi_{i,j} = 0$ for all $(i,j) \in U_1$, we may rewrite this equation as $p_i = \max_{j \in M} \{ (\Phi_{i,j})^+ \}$.
	This value is equal to $\bar{\gamma}^0_i$ according to the property (vii) of Proposition~\ref{prop:ext_point_1}.
	Following the step (R1) of the \BLP procedure, all these remaining bilinear terms are replaced with $p_i x_i = \bar{\gamma}^0_i$. 
	
	\smallskip
	Similarly, the coefficient of $y_{j}$ in the aggregated inequality \eqref{eq:aggregated} is equal to $\Psi_{j}$ according to the property (vi) of Proposition~\ref{prop:ext_point_1}.
	The \blue{coefficients} of variables with indices $j \in U_2$ become zero.
	The remaining variables $y_j$ with $j \notin U_2$ are replaced following the step (R2) in the \BLP procedure.
	In particular, we calculate the most positive coefficient of the remaining variables $y_j$ as $p_0 = \max_{j \notin U_2} \{ \Psi_j \}$.
	If $p_0 < 0$, we set it equal to zero.
	Since $\Psi_j = 0$ for all $j \in U_2$, we may rewrite this equation as $p_0 = \max_{j \in M} \{ (\Psi_j)^+ \}$.
	This value is equal to $\bar{\theta}_0$ according to the property (viii) of Proposition~\ref{prop:ext_point_1}.
	Following the step (R2) of the \BLP procedure, all these remaining variables are replaced with $p_0 = \bar{\theta}_0$. 
	
	\smallskip
	Considering the above simplifications, the aggregated inequality \eqref{eq:aggregated} reduces to 
	\begin{equation*}
		\sum_{i \in N} \left( \sum_{k \in K: k=\hat{k}, \hat{j} = 0} b^k_i + \sum_{k \in \mac{K}^0} \bar{\alpha}^0_{k} b^k_i + \sum_{t \in \mac{T}^0} \bar{\beta}^0_{t} E_{t,i} + \bar{\gamma}^0_i \right) x_i \geq \!\!\!\!\!\!\!\sum_{k \in K: k=\hat{k}, \hat{j} = 0} d_k + \sum_{k \in \mac{K}^0} \bar{\alpha}^0_{k} d_k + \sum_{t \in \mac{T}^0} \bar{\beta}^0_{t} f_t - \bar{\theta}_0,
	\end{equation*}
which is the inequality $g(\vc{x}) \geq g_0$ produced by the extreme point $\bar{\vc{\pi}}$ of $\mac{C}_{(\hat{k}, \hat{j})}$ according to the property (xi) of Proposition~\ref{prop:ext_point_1}.

\smallskip
We next show that the condition (C1) of the \BLP procedure is satisfied for the derived inequality.
It follows from the previous arguments that the coefficients of bilinear terms $x_iy_j$ for $(i,j) \in U_1$ and the variables $y_j$ for $j \in U_2$ become zero during the aggregation process.
According to the condition (C1), $|U_1| + |U_2|$ must be no less than $\sum_{j=0}^m \left(|\mac{K}^j|+|\mac{T}^j|\right) + \sum_{i=0}^n \left(\mathbb{I}(q_i) - q_i \right)$, where $q_i$ for $i \in N$ (resp. $q_0$) is the number of the remaining bilinear terms $x_i y_j$ (resp. variables $y_j$) with coefficient $p_i$ (resp. $p_0$) in the aggregated inequality \eqref{eq:aggregated}.
On the other hand, according to the property (xii) of Proposition~\ref{prop:ext_point_1}, $|U_1| + |U_2| \geq \sum_{j=0}^m|\mac{K}^j|+|\mac{T}^j| + \sum_{i=1}^n \mathbb{I}(\bar{\gamma}^0_i) - \sum_{i=1}^n\sum_{j=1}^m \mathbb{I}(\bar{\gamma}^0_i)(1 - \mathbb{I}(\bar{\gamma}^j_i)) + \mathbb{I}(\bar{\theta}_0) - \sum_{j=1}^m \mathbb{I}(\bar{\theta}_0)(1 - \mathbb{I}(\bar{\theta}_j))$.
Therefore, it suffices to show that (I) $\mathbb{I}(\bar{\gamma}^0_i) = \mathbb{I}(q_i)$ for $i \in N$, (II) $\mathbb{I}(\bar{\theta}_0) = \mathbb{I}(q_0)$, (III) $\sum_{j=1}^m \mathbb{I}(\bar{\gamma}^0_i)(1 - \mathbb{I}(\bar{\gamma}^j_i)) = q_i$ for $i \in N$, and (IV) $\sum_{j=1}^m \mathbb{I}(\bar{\theta}_0)(1 - \mathbb{I}(\bar{\theta}_j)) = q_0$.
To prove (I), we have $q_i > 0$ if and only if $p_i > 0$ by definition.
It was shown in the previous sections of this proof that $p_i = \bar{\gamma}^0_i$, which implies that $\mathbb{I}(\bar{\gamma}^0_i) = \mathbb{I}(q_i)$.
The proof of (II) follows from a similar argument.
To prove (III), for $i \in N$, it is easy to verify that $\sum_{j=1}^m \mathbb{I}(\bar{\gamma}^0_i)(1 - \mathbb{I}(\bar{\gamma}^j_i))$ counts the number of $j \in M$ such that $\mathbb{I}(\bar{\gamma}^0_i) = 1$ and $\mathbb{I}(\bar{\gamma}^j_i) = 0$.
Equivalently, we count the number of $j \in M$ such that $\bar{\gamma}^0_i > 0$ and $\bar{\gamma}^j_i = 0$.
It follows from the property (vii) of Proposition~\ref{prop:ext_point_1} that $\bar{\gamma}^j_i = \bar{\gamma}^0_i - \Phi_{i,j}$.
Therefore, $\bar{\gamma}^j_i = 0$ implies that $\bar{\gamma}^0_i = \Phi_{i,j}$.
We previously established that $\Phi_{i,j}$ is the coefficient of the bilinear term $x_i y_j$ in the aggregated inequality \eqref{eq:aggregated}.
As a result, $\bar{\gamma}^j_i = 0$ implies that the coefficient of the bilinear term $x_i y_j$ in the aggregated inequality \eqref{eq:aggregated} is equal to $\bar{\gamma}^0_i$, which is equal to $p_i$ according to the above discussion.
In summary, $\sum_{j=1}^m \mathbb{I}(\bar{\gamma}^0_i)(1 - \mathbb{I}(\bar{\gamma}^j_i))$ is equal to the number of the remaining bilinear term $x_i y_j$ in the aggregated inequality \eqref{eq:aggregated} with coefficients equal to $p_i$, which is $q_i$ according to the condition (C1) of the \BLP procedure.
The proof of (IV) follows from a similar argument.

\smallskip
To complete the proof, it remains to show that any \BLP inequality is valid for $\proj_{\vc{x}} \conv(\mac{S})$.
The above arguments in this proof have established that the steps of the \BLP procedure correspond to the conditions (i)--(xii) of Proposition~\ref{prop:ext_point_1} for a point $\bar{\vc{\pi}}$.
It follows from the converse statement in Proposition~\ref{prop:ext_point_1} that any point $\bar{\vc{\pi}}$ that satisfies conditions (i)--(xii) is feasible to $\mac{C}_{(\hat{k}, \hat{j})}$ for some $\hat{k} \in K$ and $\hat{j} \in M \cup \{0\}$.
The converse statement in Proposition~\ref{prop:class} implies that any inequality obtained from a feasible solution to $\mac{C}_{(\hat{k}, \hat{j})}$ is valid for $\proj_{\vc{x}} \conv(\mac{S})$, proving the result.
\hfill \Halmos
\endproof

\medskip

\begin{comment}
{\color{red} The following remark is true, but did not seem to be useful for the paper direction, so it was removed because it could be distracting from the main flow}
The following remark discusses the difference between the selection of bilinear constraints in $\mac{S}$ and linear constraints in $\Xi$ through the \BLP procedure.

\begin{remark} \label{remark:intersection}
	In step (A2), the intersection of the sets $\mac{K}^j$ for $j \in M \cup \{0\}$ does not need to be empty. 
	This is because any aggregation of a bilinear constraint $k \in  \bigcap_{j=0}^m \mac{K}^j$ with weights $\alpha^j_k$ for all $j \in M \cup \{0\}$ would result in a bilinear inequality that participates in the aggregation process. For example, choosing the same aggregation weight $\alpha \in \Re_+$ for all $j \in M \cup \{0\}$ for bilinear constraint $k \in  \bigcap_{j=0}^m \mac{K}^j$ would reduce to using that constraint with a constant $\alpha$ during the aggregation process. This is in contrast to the intersection rule imposed in step (A3) when choosing constraints for sets $\mac{T}^j$ corresponding to the linear inequalities in $\Xi$.
\end{remark}	
\end{comment}

\smallskip
In view of Theorem~\ref{thm:BLP}, it follows from its proof and the connections to Proposition~\ref{prop:class} that condition (C1) in the \BLP procedure is a necessary requirement for the resulting inequality to be facet-defining for $\conv(\mac{S})$. However, this condition is not required for the inequality to be valid. In other words, a linear inequality produced by the \BLP procedure that does not satisfy (C1) remains valid for $\conv(\mac{S})$, but it will not be facet-defining.

\smallskip
\blue{
As shown in the proof of Theorem~\ref{thm:BLP}, the correctness of the \BLP method is established using disjunctive programming arguments. However, unlike the traditional lift-and-project framework, it enables the systematic derivation of closed-form, parametric valid inequalities directly in the original variable space without explicit lifting and projection. Once the sets and weights are chosen according to steps (A1)--(A3), (R1), and (R2), the resulting inequality is guaranteed to be valid for the convex hull of the bilinear set, while parametric weight choices generate entire families of valid inequalities. This capability distinguishes \BLP from classical disjunctive programming, where deriving parametric families of inequalities in the original space is generally intractable because it requires computing the extreme rays of a projection cone, which in turn involves solving systems of equations. As a result, \BLP not only provides explicit parametric descriptions of the convex hull that offer deeper polyhedral insight but also facilitates the development of efficient separation procedures without the need to solve separation optimization problems. In Section~\ref{sec:chance_bilinear}, we exploit these advantages of the \BLP procedure to derive new families of valid inequalities for the mixing set \eqref{eq:genprob_chance}. 
}

\smallskip
We conclude this section by providing several remarks that help adjust and \blue{simplify} the steps of the \BLP procedure for special structures that commonly appear in various application domains.
The first remark shows that for sparse structures of the bilinear set $\mac{S}$, the \BLP procedure can be substantially simplified.
These structures appear in numerous applications, including in CCP (see Section \ref{sec:chance_bilinear}).

\begin{remark} \label{remark:simplification}
	Consider a bilinear constraint $k^* \in K$ in the description of $\mac{S}$ that contains a single $y$ variable, and is of the form $(\vc{a}\tr \vc{x})y_{j^*} + c_{j^*} y_{j^*} \geq 0$ for some $j^* \in M$.
	Following step (A1) of the \BLP procedure, if this inequality is used for the base constraint, i.e., $\hat{k} = k^*$, the only choice for $\hat{j}$ will be $\hat{j} = j^*$.
	This leads to the weighted inequality $(\vc{a}\tr \vc{x})y_{j^*} + c_{j^*} y_{j^*} \geq 0$, which means using the same bilinear constraint with constant aggregation weight $1$ as the base constraint.
	This is due to the fact that using $\hat{j} \neq j^*$ would lead to zero terms on the left-hand-side of the aggregated inequality.
	Similarly, following step (A2) of the \BLP procedure, if this inequality is used in the aggregation, the only choice will be to select it in the set $\mac{K}^{j^*}$.
	This yields the weighted inequality $\alpha^{j^*}_{k^*}(\vc{a}\tr \vc{x})y_{j^*} + \alpha^{j^*}_{k^*}c_{j^*} y_{j^*} \geq 0$, which is using the same bilinear constraint with a constant aggregation weight $\alpha^{j^*}_{k^*}$.
    We can use similar arguments for the case where a bilinear constraint $k^* \in K$ in the description of $\mac{S}$ is of the form $(\vc{a}\tr \vc{x})(\vc{1} - \vc{1} \tr \vc{y}) + c_0 (\vc{1} - \vc{1} \tr \vc{y}) \geq 0$.
    Following step (A1) of the \BLP procedure, if this inequality is selected for the base constraint, i.e., $\hat{k} = k^*$, the only choice for $\hat{j}$ will be $\hat{j} = 0$.
    This leads to using this inequality in the aggregation with a constant weight of $1$ as the base constraint.
    Similarly, following step (A2) of the \BLP procedure, if this inequality is used in the aggregation, the only choice will be to select it in the set $\mac{K}^{0}$.
	This yields using this inequality with a constant aggregation weight $\alpha^{0}_{k^*}$.
	These simplifications can significantly reduce the number of possible options for \BLP assignments, which could in turn lead to more efficient cut-generation procedures.
\end{remark}	

\smallskip
In many applications, some of the variables $\vc{x}$ in the set $\Xi$ may have an upper bound. 
Although this bound can be treated as a regular constraint in step (A3) of the BL\&P framework, the procedure can be further streamlined to reduce the size of the \BLP assignment, as detailed in the following remark.

\begin{remark} \label{remark:upper bound}
	Assume that $\Xi$ contains the constraints of the form $-x_{i^*} \geq -1$, representing scaled upper bounds for certain variables $x_{i^*}$ for $i^* \in I^*$, where $I^* \subseteq N$.
	In step (A3) of the \BLP procedure, these constraints can be excluded from the assignment sets $\mac{T}^j$ for $j \in M$.
	Instead, we add the following step (R0) before the steps (R1) and (R2).
	\begin{itemize}
		\item[\textit{(R0)}] \textit{For each $i^* \in I^*$, select a subset $\mac{U}^{i^*}$ of the remaining bilinear terms $u_{i^*j}x_{i^*}y_j$ with positive coefficient $u_{i^*j} > 0$ for $j \in M$ in the aggregated inequality, and then replace them with $u_{i^*j}y_j$.
		Similarly, select a subset $\mac{V}^{i^*}$ of the remaining terms $v_jy_j$ with negative coefficient $v_j < 0$ for $j \in M$ in the aggregated inequality, and then replace them with $v_{j}x_{i^*}y_j$.}
		%If these replacements result in the coefficients of additional bilinear terms or $y$-variables to become zero in the aggregated inequality, record the total number of reductions as $\nu_{i^*}$.	
	\end{itemize}

	%After completing the substitution steps, the number of terms whose coefficients have become zero in condition (C1) can be updated as $\sum_{j=0}^m \left(|\mac{K}^j|+|\mac{T}^j|\right) + \sum_{i=0}^n \left(\mathbb{I}(q_i) - q_i \right) - \sum_{i \in I^*}\nu_{i}$.
\end{remark}

\smallskip
The next remarks discuss a simplifying step in the \BLP process when $\mac{S}$ contains bilinear complementarity constraints.

\begin{remark} \label{remark: complementarity 1}
    Assume that $\mac{S}$ contains the bilinear constraint $-x_{i^*} y_{j^*} \geq 0$ for some $i^* \in N$ and $j^* \in M$.
    This constraint is imposed when there is a complementarity relationship between $x_{i^*}$ and $y_{j^*}$, since it will imply that $x_{i^*} y_{j^*} = 0$ as $x_{i^*} \geq 0$ and $y_{j^*} \geq 0$.
    Within the \BLP procedure, this bilinear constraint can be omitted from $\mac{S}$ and excluded from the aggregation process.
    Instead, in the aggregated inequality, we can substitute the remaining term of the form $u x_{i^*} y_{j^*}$ for any $u \in \Re$ with $0$ prior to applying step (R1). 
\end{remark}

\begin{remark} \label{remark: complementarity 2}
    Assume that $\Xi$ contains the constraints of the form $-x_{i^*} \geq -1$, for some $i^* \in N$. In addition,  assume that $\mac{S}$ contains the bilinear constraint $-(1-x_{i^*}) y_{j^*} \geq 0$ for some $i^* \in N$ and $j^* \in M$.
    This constraint is imposed when there is a complementarity relationship between the complement of $x_{i^*} \in [0,1]$ and $y_{j^*}$, since it will imply that $(1-x_{i^*}) y_{j^*} = 0$ as $x_{i^*} \leq 1$ and $y_{j^*} \geq 0$.
    In the \BLP procedure, we can omit this bilinear constraint from $\mac{S}$, and thereby exclude it from the aggregation process.
    Instead, in the aggregated inequality, we can perform one of the following two substitution options \textit{after} step (R0) discussed in Remark~\ref{remark:upper bound} and \textit{before} step (R1) in the \BLP procedure: 
    %{\color{red} The following case is the most general one that leads to the strongest relaxation. However, since it is more contrived, for the simplicity of presentation, we only present the weaker version here where $u_{i^*,j^*} < 0$...} (i) we may replace any remaining term of the form $u_{i^*,j^*} x_{i^*} y_{j^*}$, where $u_{i^*,j^*} < p_{i^*}$, with $(u_{i^*,j^*}-p_{i^*}) y_{j^*}$, where $p_{i^*} = \max_{j \in M}\{u_{i^*,j}, 0\}$ and $u_{i^*,j}$ denotes the coefficient of the remaining bilinear terms $x_{i^*}y_j$ for all $j \in M$;
    (i) we may replace any remaining term of the form $u_{i^*,j^*} x_{i^*} y_{j^*}$ with $u_{i^*,j^*} y_{j^*}$ when $u_{i^*,j^*} < 0$; (ii) we may replace any remaining term of the form $u_{i^*,j^*} y_{j^*}$ with $u_{i^*,j^*}x_{i^*} y_{j^*}$ when $u_{i^*,j^*} > 0$.
    %If these replacements result in the coefficient of $y_{j^*}$ in option (i) or that of the bilinear term $x_{i^*} y_{j^*}$ in option (ii) to become zero in the aggregated inequality, the required number of such terms in condition (C1) is subtracted by one.
\end{remark}

%%%%%%%%%%%%%%%%%%%%%%%%%%%%%%%%%%%%%%%%%%%%%%%%%%
%%%%%%%%%%%%%%%%%%%%%%%%%%%%%%%%%%%%%%%%%%%%%%%%
%%%%%%%%%%%%%%%%%%%%%%%%%%%%%%%%%%%%%%%%%%%%%%%%

\section{Convexification of CCP through Bilinear Extended Reformulation}
\label{sec:chance_bilinear}

%Consider a general form of a chance constraint set, given by 
%\begin{equation} \label{eq:genprob_chance}
%\mac{F}_c = \left\{ (z,\vc{x}) \in \Re_+ \times \{0,1\}^m\, \middle \vert 
%\begin{array}{l}
%\sum\limits_{i \in M} \pi_i x_i \le \varepsilon,
%\; x_i=0 \Rightarrow z\ge h_i, \quad  i \in M\\ 
%\end{array}
%\right\},    
%\end{equation}
%where $M := \{1,\dotsc,m\}$, $\varepsilon \in [0,1)$ and we assume without loss of generality that $h_1\ge h_2 \ge \cdots \ge h_m$. Moreover,  $\sum_{i \in M} \pi_i=1$ and $\pi_i \ge 0$ for $i \in M$.
%{\color{red} Let $p=\max\{k: \sum_{i=1}^{k} \pi_i \le \varepsilon\}$ with $p <m$ by definition.} 
%We also define $F_j=\sum_{i=1}^{j} \pi_j$, $j \in M$. 

In this section, we introduce a novel technique for deriving valid inequalities for the convex hull of the mixing set with a knapsack constraint, $\mac{F}_c$, as defined in \eqref{eq:genprob_chance}. Our approach begins by reformulating $\mac{F}_c$ as a bilinear set over a simplex in a higher-dimensional space. We then apply the \BLP procedure to this bilinear set to generate families of valid inequalities for $\conv(\mac{F}_c)$, covering both the general case with an arbitrary probability distribution and the special case with a uniform probability distribution.

%%%%%%%%%%%%%%%%%%%%%%%%%%%%%%%%%%%%%%%%%%%%%%%%%%%%%
\subsection{Bilinear Reformulation} \label{subsec:reformulation}

Consider the bilinear set 
\begin{equation}
\mac{S}_c = \bigl\{ (z,\vc{x};\vc{y}) \in \Xi_c \times \Delta_m \, \big \vert \eqref{eq:general_set_1} - \eqref{eq:general_set_5} \bigr\}, \label{eq:bilinear-reformulation}
\end{equation}
where
\begin{subequations}
    \begin{align}
        (1-x_i) (1- \vc{1}\tr \vc{y}) &\ge 0, &i \in M \label{eq:general_set_1}\\ 
-(1-x_i) (1- \vc{1}\tr \vc{y}) &\ge 0, &i \in M \label{eq:general_set_2}\\ 
z y_i - h_i y_i &\ge 0, &i \in M \label{eq:general_set_3}\\
- x_i y_i &\ge 0, &i \in M \label{eq:general_set_4}\\ 
-(1-x_i) y_j &\ge 0, &i,j \in M:  i<j, \label{eq:general_set_5}
    \end{align}
\end{subequations}
and where $M=\{1,\dotsc,m\}$, $\Xi_c = \left\{ (z,\vc{x}) \in \Re_+ \times \Re^m_+ \, \middle| \, - \vc{x} \geq \vc{-1}, \; -\vc{\pi}\tr \vc{x} \ge -\varepsilon \right\}$, and $\Delta_m = \left\{ \vc{y} \in \Z_+^{m} \, \middle| \, \vc{1}\tr \vc{y} \leq 1 \right\}$.

\smallskip
\blue{
The next result shows that the bilinear set $\mac{S}_c$ provides an extended formulation of the set $\mac{F}_c$ when $\vc{x}$ variables are binary. The motivation is that the auxiliary binary variable $y_j$ indicates whether $x_j$ is the first variable, starting from $x_1$, to become zero, that is, whether the constraint corresponding to scenario $j$ is active. These relationships are captured by the bilinear constraints \eqref{eq:general_set_1}--\eqref{eq:general_set_5}. 
}

\blue{
\begin{proposition} \label{prop:extended formulation}
    It holds that $\proj_{(z, \vc{x})} \mac{S}_c \cap (\Re \times \Z^m \times \Re^m) = \mac{F}_c$.
\end{proposition}
}

\proof{Proof}
\blue{
For the direct inclusion, we prove $\proj_{(z, \vc{x})} \mac{S}_c \cap (\Re \times \Z^m \times \Re^m) \subseteq \mac{F}_c$. Consider $(\bar{z}, \bar{\vc{x}}) \in \proj_{(z, \vc{x})} \mac{S}_c \cap (\Re \times \Z^m \times \Re^m)$.
%We show that $(\bar{z}, \bar{\vc{x}})$ satisfies all constraints in $\mac{F}_c$.
Since $(\bar{z}, \bar{\vc{x}}) \in \Xi_c \cap (\Re \times \Z^m)$, it suffices to show that $\bar{z} \geq h_i$ whenever $\bar{x}_i = 0$, for $i \in M$. 
By the definition of projection, there exists $\bar{\vc{y}} \in \Delta_m$ such that $(\bar{z}, \bar{\vc{x}}; \bar{\vc{y}}) \in \mac{S}_c \cap (\Re \times \Z^m \times \Re^m)$. There are two cases.
First, assume that $\bar{y}_i = 0$ for all $i \in M$. Then, constraints \eqref{eq:general_set_1} and \eqref{eq:general_set_2} imply that $\bar{x}_i = 1$ for all $i \in M$. Consequently, constraints \eqref{eq:general_set_3} reduce to $\bar{z}$ is $\bar{z} \geq 0$, which is precisely the corresponding requirement in $\mac{F}_c$.
Second, assume that $\bar{y}_{i^*} = 1$ for some $i^* \in M$.
Then, constraint \eqref{eq:general_set_4} implies that $\bar{x}_{i^*} = 0$, while constraint \eqref{eq:general_set_5} implies that $\bar{x}_i = 1$ for all $i < i^*$.
Then, it follows from \eqref{eq:general_set_3} that $\bar{z} \geq h_{i^*}$, which is exactly as required by the definition of $\mac{F}_c$.
}

\blue{
For the reverse inclusion, we prove $\proj_{(z, \vc{x})} \mac{S}_c \cap (\Re \times \Z^m \times \Re^m) \supseteq \mac{F}_c$. Consider $(\bar{z}, \bar{\vc{x}}) \in \mac{F}_c$. 
Since $(\bar{z},\bar{\vc{x}}) \in \Xi_c \cap (\Re \times \Z^m)$, it remains to construct $\bar{\vc{y}} \in \Delta_m$ satisfying constraints \eqref{eq:general_set_1}--\eqref{eq:general_set_5}.
There are two cases. 
First, assume that $\bar{x}_i = 1$ for all $i \in M$, which can occur only when $\varepsilon = 1$.
Then, $\mac{F}_c$ requires that $\bar{z} \geq 0$.
Setting $\bar{y}_i = 0$ for all $i \in M$ immediately satisfies \eqref{eq:general_set_1}--\eqref{eq:general_set_5} together with $\Delta_m$.
Second, assume that $\bar{x}_i = 1$ for some $i \in M$.
Define $i^* = \min \{i \in M: \bar{x}_i = 0\}$. 
By the definition of $\mac{F}_c$, we have $\bar{z} \geq h_{i^*}$.
Setting $\bar{y}_{i^*} = 1$ and $\bar{y}_i = 0$ for all $i \in M \setminus \{i^*\}$ satisfies \eqref{eq:general_set_1}--\eqref{eq:general_set_5} together with $\Delta_m$.
}
\hfill \Halmos
\endproof

Since $\mac{S}_c$ is an extended formulation obtained by introducing the auxiliary variables $\vc{y}$, we are interested in characterizing the convex hull of its projection onto the space of the original variables $(z, \vc{x})$.
The bilinear set $\mac{S}_c$ conforms to the structure of the general set $\mac{S}$ introduced in \eqref{eq:generalset}, where the variables $(z, \vc{x})$ in $\mac{S}_c$ correspond to the $\vc{x}$ variables in $\mac{S}$. 
As a result, we can apply the \BLP procedure to obtain valid inequalities for the projection of the convex hull of $\mac{S}_c$ onto the $(z, \vc{x})$-space.
To this end, we \blue{first} need to show that Assumption~\ref{asm:cone} is satisfied.

\begin{lemma} \label{lem:chance}
    Assumption~\ref{asm:cone} holds for $\mac{S}_c$. 
\end{lemma}

\proof{Proof}
Let $\mac{S}_c(\vc{e}^j)$ be the restriction of $\mac{S}_c$ at point $\vc{y} = \vc{e}^j$ for $j \in M \cup \{0\}$.
\blue{For part (i) of Assumption~\ref{asm:cone}, we show that $\mac{S}_c(\vc{e}^j) \neq 0$ for some $j \in M \cup \{0\}$.
Consider $j = 1$. 
Then, the point $(\bar{z},\bar{\vc{x}};\bar{\vc{y}}) = (h_1, \vc{0}; \vc{e}^1)$ is feasible to $\mac{S}_c(\vc{e}^1)$.}
For part (ii) of Assumption~\ref{asm:cone}, it is easy to verify that all sets $\mac{S}_c(\vc{e}^j)$ for $j \in M \cup \{0\}$ share the same recession cone described as $\{(z,\vc{x};\vc{y}) \, | \, z \geq 0, \vc{x} = \vc{0}, \vc{y} = \vc{0} \}$.
\hfill \Halmos
\endproof

\smallskip
The next proposition shows the relationship between the convex hulls of $\mac{F}_c$ and $\mac{S}_c$.

\begin{proposition} \label{prop:genprob_chance_reform}
It holds that $\mac{F}_c \subseteq \proj_{(z,\vc{x})} \mac{S}_c$.     
\end{proposition}

\proof{Proof}
%Consider a point $(\bar{z},\bar{\vc{x}}) \in \mac{F}_c$. We show that there exists $\bar{\vc{y}}$ such that $(\bar{z},\bar{\vc{x}};\bar{\vc{y}}) \in \mac{S}_c$. There are two cases. For the first case, assume that $\bar{x}_i = 1$ for all $i \in M$. As a result, \eqref{eq:genprob_chance} implies that $\bar{z} \geq 0$. Let $\bar{y}_j = 0$ for all $j \in M$. It is easy to verify that the point $(\bar{z},\bar{\vc{x}};\bar{\vc{y}})$ satisfies all constraints in $\mac{S}_c$, and thus $(\bar{z},\bar{\vc{x}};\bar{\vc{y}}) \in \mac{S}_c$. For the second case, assume that $\bar{x}_i = 0$ for some $i \in M$. Let $i^* \in M$ be the smallest index of the variables with $\bar{x}_i=0$. As a result, \eqref{eq:genprob_chance} implies that $\bar{z} \geq h_{i^*}$. Let $\bar{y}_{i^*}=1$ and $\bar{y}_j=0$ for $j \in M \setminus \{i^*\}$. It is easy to verify that the point $(\bar{z},\bar{\vc{x}};\bar{\vc{y}})$ satisfies all constraints in $\mac{S}_c$, and thus $(\bar{z},\bar{\vc{x}};\bar{\vc{y}}) \in \mac{S}$. 
The result follows from Proposition~\ref{prop:extended formulation} as $\mac{F}_c = \proj_{(z, \vc{x})} \mac{S}_c \cap (\Re \times \Z^m \times \Re^m) \subseteq \proj_{(z,\vc{x})} \mac{S}_c$.
\hfill \Halmos 
\endproof

\smallskip
Lemma~\ref{lem:chance} and Proposition~\ref{prop:genprob_chance_reform} imply that the \BLP method can be applied to the bilinear set $\mac{S}_c$ to obtain valid inequalities for the convex hull of the original set 
$\mac{F}_c$.
Although the results presented in this section apply to the general form of $\mac{F}_c$---where the knapsack constraint involves an arbitrary probability distribution $\vc{\pi}$---many existing studies on deriving explicit valid inequalities focus on the special case of a uniform distribution, as defined by $\mac{F}_c^=$ in \eqref{eq:cardinality}; see, for instance, \cite{kuccukyavuz2012mixing,luedtke2010integer}.
In this special case, where the chance constraint reduces to a cardinality constraint, the relationship between $\mac{F}_c$ and $\mac{S}_c$, as described in Proposition~\ref{prop:genprob_chance_reform}, becomes stronger.
In fact, we can show that under this setting, both sets share the same convex hull in the $(z, \vc{x})$-space.
%This result, presented next, demonstrates that applying the \BLP procedure to $\bar{\bar{\mac{S}}}_c$ yields the convex hull of $\bar{\bar{\mac{F}}}_c$. 
%This key property underscores the generality and strength of our approach, offering broader applicability than those achieved by existing methods in the literature.

\begin{proposition} \label{prop:uniform_chance_reform}
%Let $\mac{F}_c^=$ and $\mac{S}_c^=$ be the special cases of $\mac{F}_c$ and $\mac{S}_c$, respectively, where the chance constraint $-\vc{\pi}\tr \vc{x} \ge -\varepsilon$ in their description is replaced with the cardinality constraint $-\vc{1}\tr \vc{x} \ge -p$ for some $p \in M$. 
It holds that 
%Then, 
$\conv(\mac{F}_c^=) = \proj_{(z,\vc{x})} \conv(\mac{S}_c^=)$.  
\end{proposition}

\proof{Proof}
Since $\mac{F}_c^=$ and $\mac{S}_c^=$ are special case $\mac{F}_c$ and $\mac{S}_c$, where $\pi_i = 1/m$ for all $i \in M$ and $p= \lfloor m \varepsilon \rfloor$.
Therefore, it follows from Proposition~\ref{prop:genprob_chance_reform} that $\mac{F}_c^= \subseteq \proj_{(z,\vc{x})} \mac{S}_c^=$.
Taking the convex hull from both sides and using the commutative property of the convex hull and projection operators, we obtain $\conv(\mac{F}_c^=) \subseteq \proj_{(z,\vc{x})} \conv(\mac{S}_c^=)$.

Next, we show the reverse inclusion $\conv(\mac{F}_c^=) \supseteq \proj_{(z,\vc{x})} \conv(\mac{S}_c^=)$.
Consider a point $(\bar{z},\bar{\vc{x}}) \in \proj_{(z,\vc{x})} \conv(\mac{S}_c^=)$. 
There exists $\bar{\vc{y}} \in \Re^m$ such that $(\bar{z},\bar{\vc{x}}; \bar{\vc{y}}) \in \conv(\mac{S}_c^=)$.
As a result, there exists a collection of points $(\hat{z}^k,\hat{\vc{x}}^k; \hat{\vc{y}}^k) \in \mac{S}_c^=$ for some $k \in \{1, \dotsc, \kappa\}$ such that $(\bar{z},\bar{\vc{x}}; \bar{\vc{y}}) = \sum_{k=1}^{\kappa} \lambda_k (\hat{z}^k,\hat{\vc{x}}^k; \hat{\vc{y}}^k)$ with $\sum_{k=1}^{\kappa} \lambda_k = 1$ and $\lambda_k \geq 0$ for all $k \in \{1,\dotsc,\kappa\}$.
We next show that $(\hat{z}^k,\hat{\vc{x}}^k) \in \conv(\mac{F}_c^=)$ for all $k \in \{1,\dotsc,\kappa\}$.
Consider $(\hat{z}^k,\hat{\vc{x}}^k; \hat{\vc{y}}^k)$ for some $k \in \{1,\dotsc,\kappa\}$.
There are two cases.
For the first case, assume that $\hat{y}_j^k = 0$ for all $j \in M$.
It follows from \eqref{eq:general_set_1} and \eqref{eq:general_set_2} that $\hat{x}_i^k = 1$ for all $i \in M$.
Further, \eqref{eq:general_set_3} implies that $\hat{z}^k \geq 0$.
Hence, $(\hat{z}^k,\hat{\vc{x}}^k) \in \mac{F}_c^= \subseteq \conv(\mac{F}_c^=)$.
For the second case, assume that $\hat{y}_{j^*}^k = 1$ for some $j^* \in M$ and $\hat{y}_{j}^k = 0$ all $j \in M \setminus \{j^*\}$.
It follows from \eqref{eq:general_set_3} that $\hat{z}^k \geq h_{j^*}$.
Furthermore, \eqref{eq:general_set_4} implies that $\hat{x}_{j^*}^k = 0$, while \eqref{eq:general_set_5} implies that $\hat{x}_{i}^k = 1$ for all $i < j^*$.
Since the set $\Xi_c$ in the description of $\mac{S}_c^=$ is a unit cube intersected with a cardinality constraint, all of its extreme points are integral.
Therefore, there exists a collection of integral points $\tilde{\vc{x}}^l$ for $l \in \{1, \dotsc, \ell\}$ of $\Xi_c$ such that $\hat{\vc{x}}^k = \sum_{l=1}^{\ell} \mu_l \tilde{\vc{x}}^l$ with $\sum_{l=1}^{\ell} \mu_l = 1$ and $\mu_l > 0$ for all $l \in \{1,\dotsc,\ell\}$. 
According to the previous argument, since $\hat{x}_{i}^k \in \{0,1\}$ for all $i \leq j^*$, we must have $\tilde{x}^l_i = \hat{x}_{i}^k$ for all $i \leq j^*$ and all $l \in \{1,\dotsc,\ell\}$.
In particular, for each $l \in \{1,\dotsc,\ell\}$, we have $\tilde{x}^l_{j^*} = 0$ and $\tilde{x}^l_{i} = 1$ for $i < j^*$, which implies that the point $(\hat{z}^k,\tilde{\vc{x}}^l)$ is feasible to $\mac{F}_c^=$.
Hence, we can write that $(\hat{z}^k,\hat{\vc{x}}^k) = \sum_{l=1}^{\ell} \mu_l (\hat{z}^k,\tilde{\vc{x}}^l)$, which shows that $(\hat{z}^k,\hat{\vc{x}}^k) \in \conv(\mac{F}_c^=)$.
In sum, we have shown that $(\bar{z},\bar{\vc{x}})$ is represented as a convex combination of points $(\hat{z}^k,\hat{\vc{x}}^k) \in \conv(\mac{F}_c^=)$ for all $k \in \{1,\dotsc,\kappa\}$, proving that $(\bar{z},\bar{\vc{x}}) \in \conv(\mac{F}_c^=)$.
\hfill \Halmos 
\endproof

\blue{
We emphasize that the bilinear formulation $\mac{S}_c$ is not intended to be used directly as a model within a solver or as a basis for generating cutting planes in computational studies. Rather, it serves as a device for establishing a key structural property that enables the application of the \BLP procedure to derive valid inequalities for the mixing set with a knapsack constraint. Specifically, this property is that the associated disjoint bilinear set is defined over a simplex.}

\blue{
Interpreted as a disjunctive formulation in which the disjuncts correspond to the extreme points of the simplex, this formulation has the important advantage of admitting a compact representation due to a linear number of disjuncts. For the cardinality-constrained case, Proposition~\ref{prop:uniform_chance_reform} shows that it recovers the convex hull, similar to the compact formulation presented in Section~4 of \cite{kuccukyavuz2012mixing}. For the general knapsack case, Proposition~\ref{prop:genprob_chance_reform} shows that it provides a compact and strong relaxation whose valid inequalities capture various facets of the original set, as demonstrated in Section~\ref{subsec:BLP family-general}. In contrast, the standard disjunctive programming formulation for the mixing set with a general knapsack constraint requires exponentially many disjuncts, making it significantly more expensive to represent and exploit computationally.
}

%%%%%%%%%%%%%%%%%%%%%%%%%%%%%%%%%%%%%%%%%%%%%%%%%%%%%
\subsection{Valid Inequalities for the Mixing Set with a Knapsack Constraint} \label{subsec:BLP family-general}

In this section, we demonstrate how the \BLP procedure yields the explicit form of an important family of valid inequalities for $\conv(\mac{F}_c)$, which we refer to as \textit{BL\&P inequalities}. This family subsumes several known results from the literature while also introducing a broad class of new inequalities with distinct structural forms that lie beyond the scope of existing results.

\blue{In the following theorem, to provide a unified treatment of the trivial cases in which certain sets are empty, we adopt the convention that any conditions or expressions involving such sets in the statements and proofs are understood to be omitted. Specific details are provided at the beginning of the proof.}

\begin{theorem}\label{thm:chance_lift_generic}
%Consider $\mac{F}_c$ and $\mac{S}_c$ as defined in \eqref{eq:genprob_chance} and \eqref{eq:bilinear-reformulation}, respectively. 
Let $r \in \{1,\dotsc,p\}$, $l \in \{1,\dotsc,r\}$, and define $L:=\{1,\dotsc, l\}$.  
\blue{Select  $\delta_{t_\iota} \in \Re$ for  $\iota \in L$ such that $\delta_{t_\iota} \ge h_{t_{\iota+1}} - h_{t_{\iota}}$, where  $t_{l+1}:=r+1$, with $\sum_{\iota \in L} \delta_{t_{\iota}} \le  h_{r+1}$ when $\varepsilon \in [0,1)$ and $\sum_{\iota \in L} \delta_{t_{\iota}} =h_{r+1}$ when $\varepsilon =1$}. 
Consider %\textcolor{red}{$v \in \{1,\dotsc, p-r+ \sum_{\iota \in L} \mathbb{I}(\delta_{t_\iota}) \}$}, 
\blue{$v \in \left\{0,\dotsc, \min\{p-r+ l,m-r\} \right\}$}, 
and define 
$V :=\{1,\dotsc, v\}$.
Let $P:=\{t_\iota: \iota \in L\} \subseteq \{1,\dotsc, r\}$ and $Q:=\{q_\iota: \iota \in V\} \subseteq \{r + 1, \dotsc, m\}$, where  $t_1 < \cdots < t_l$ and $q_1 < \cdots <q_{v}$, and define $a_j:=\bigl|\{\iota \in L: t_\iota < j  \}\bigr|$. %  and $b_j:=\bigl|\{\iota \in V: q_\iota < j  \}\bigr|$ for $j \in M$. 
%\textcolor{red}{I don't see any reason why $Q$ can't be a subset, with the proof written as is without counting the number of relaxed terms. Also, there is not much insight in the proof why $Q$ can't include $r+1$. This probably can't be the case given the number of relaxed terms..}
Moreover, select $A_j \subseteq V$ for $j \in  M$.
%Suppose that for some $(\phi_{q_1}, \dotsc, \phi_{p-r}) \in \Re^{p-r}_+$ and $\beta^j \in \Re_+$, $j \in M \setminus P$, we have 
%where $A_{j}^{+}:=\{\iota \in V: \phi_{q_{\iota}} - \beta^{j} \ge 0\}$ and $A_{j}^{-}= V \setminus A_{j}^{+}$, $j \in M \setminus P$.
%Then, the inequalities 
%\begin{equation} \label{eq:chance_lift_ineq_generic}
%    z + \sum_{\iota=1}^{l} (h_{t_{\iota}} - h_{t_{\iota+1}}) x_{t_{\iota}} + \sum_{\iota=1}^{p-r} \phi_{q_{\iota}} (1-x_{q_{\iota}}) \ge h_{t_{1}}, 
%\end{equation}
%with $h_{t_{l+1}}:=h_{r+1}$,   
%are valid for $\conv(\mac{F}_c)$, as defined in \eqref{eq:genprob_chance}. 

Suppose that for some $(\phi_{q_1}, \dotsc, \phi_{q_{v}}) \in \Re^{v}_+$ there exists $\beta_j \in \Re_+$ for each $j \in M$ that satisfies  
\begin{subequations}
\label{A}
\begin{align}
    & \max\Big\{ \frac{\phi_{q_\iota}}{m  \pi_{q_\iota}}: \iota \in A_j, \; q_{\iota} > j  \Big\} \le \beta_j  \le \min\Big\{\frac{\phi_{q_\iota}}{m \pi_{q_\iota}}: \iota \in V \setminus A_j, \;  q_{\iota} > j \Big\}, \label{A_1}\\ 
    %& \beta^{q_j}   \Big ( F_{q_j} - \pi_{q_j} - \varepsilon + \sum_{\iota \in V, q_i > q_j} \pi_{q_\iota}  - \sum_{\iota \in A_{q_j}^-,   q_{\iota}>q_j} \pi_{q_\iota} \Big) \ge h_{t_{l+1}} - h_{q_j} -   \phi_{q_j} - \sum\limits_{\iota \in A_{q_j}^-,  q_{\iota}>q_j} \phi_{q_{\iota}}, \quad j \in V, \label{A_2}\\ 
    & \beta_{j}   \Big ( F_{j}  - \pi_j - \varepsilon + \sum_{\substack{\iota \in V \\ q_{\iota} > j}} \pi_{q_\iota}  - \sum_{\substack{\iota \in A_{j} \\   q_{\iota}>j}} \pi_{q_\iota} \Big) \ge \Big(  h_{t_{a_j+1}} - h_{j} - \sum_{\substack{\iota \in L \\ t_\iota <j}} \delta_{t_\iota}  - \mathbb{I}(j \in Q)\phi_j  -\sum\limits_{\substack{\iota \in A_{j}\\  q_{\iota} > j}} \phi_{q_{\iota}}\Big)\big/m, \label{A_3}  
\end{align}   
\end{subequations}
where  
${ \blue {F_j:=\sum_{i=1}^{j} \pi_i} }$. 
Then, the \BLP inequality 
\begin{equation} \label{eq:chance_lift_ineq_generic}
    z + \blue{\sum_{\iota \in L}} (h_{t_{\iota}} - h_{t_{\iota+1}} + \delta_{t_{\iota}}) x_{t_{\iota}} + \blue{\sum_{\iota \in V}} \phi_{q_{\iota}} (1-x_{q_{\iota}}) \ge h_{t_{1}}, 
\end{equation}   
is valid for $\proj_{(z,\vc{x})} \conv(\mac{S}_c)$; hence, is valid for $\conv(\mac{F}_c)$.

\end{theorem}

\proof{Proof}
\blue{Before beginning the proof, we adopt the following convention to simplify the presentation. With a slight abuse of notation, we impose that whenever a set is empty, any conditions or expressions involving that set in the statements and proofs are understood to be omitted. For instance, when \(v=0\), the sets \(V\) and \(Q\) are empty. Consequently, the associated sets \(A_j\), \(j \in M\), and condition \eqref{A_1} are disregarded. Likewise, any terms involving these sets in condition \eqref{A_3} are omitted from the corresponding expressions.}

    %\blue{Note that if $\varepsilon=1$, $\mac{F}_c$ reduces to the mixing set \eqref{eq:mixing}, for which the star inequalities \eqref{eq:star_ineq} are valid. We thus prove the statement for $\varepsilon <1$.}
    To comply with the indexing rules of the \BLP procedure, we refer to the bilinear inequalities \eqref{eq:general_set_1}--\eqref{eq:general_set_5} as {\it type} $k = 1, \dotsc, 5$, respectively. 
    Since each bilinear inequality of type $k=1,\dotsc,4$ entails multiple constraints with indices $j \in M$, we distinguish each individual constraint by using the notation $j^k$ in the \BLP procedure.
    We also refer to the index of the knapsack constraint $-\vc{\pi}\tr \vc{x} \ge -\varepsilon$ in $\Xi$ with $c$. 

    \smallskip
    Suppose that for some $(\phi_{q_1}, \dotsc, \phi_{q_v}) \in \Re^{v}_+$ and $\beta_j \in \Re_+$, for each $j \in M$, \eqref{A} holds. 
    We show that \eqref{eq:chance_lift_ineq_generic} is valid for $\proj_{(z,\vc{x})} \conv(\mac{S}_c)$, which implies its validity for $\conv(\mac{F}_c)$ by Proposition~\ref{prop:genprob_chance_reform}.
    To show the former statement, according to Theorem~\ref{thm:BLP}, it suffices to prove that \eqref{eq:chance_lift_ineq_generic} can be obtained as a \BLP inequality.
    In particular, we show that this \BLP inequality is produced by a class-$(t_{1}^3,t_{1})$ \BLP assignment, where $t_1^3$ indicates the bilinear constraint of type 3 with constraint index $t_1$ as defined in set $P$.
    The choice of $t_1$ for the second component of the class follows from Remark~\ref{remark:simplification} since the bilinear constraints in $\mac{S}_c$ contain a single $y$ variable.
    The aggregation sets for this \BLP assignment are selected as follows.
{\blue{Define $\beta_0$ such that $\beta_0 m(1-\varepsilon)=h_{r+1} - \sum_{\iota \in L} \delta_{t_\iota}$. In other words, $\beta_0=\frac{h_{r+1} - \sum_{\iota \in L} \delta_{t_\iota}}{m(1-\varepsilon)}$ when $\varepsilon \in [0,1)$ and $\beta_0 \in \Re_+$ is chosen arbitrarily otherwise}}. 
It is clear that $\beta_0 \geq 0$ for any $\varepsilon \in [0, 1)$ as $h_{r+1} \geq \sum_{\iota \in L} \delta_{t_\iota}$ by assumption.
Let $\mac{K}^0=\{j^1: j \in Q\} \cup \{j^2: j \in M\}$,
    $\mac{K}^{j}= \{j^3\}$ for $j\in M\setminus \{t_1\}$, 
    $\mac{T}^0=\{c\}$, 
    and $\mac{T}^{j}=\{c\}$ for $j \in M$.
    %such that $\beta_j \ge 0$.
    All other sets $\mac{K}^j$ and $\mac{T}^j$ are empty. 

    \smallskip
    Next, we assign aggregation weights for the constraints in the \BLP assignment as described in steps (A1)--(A3) of the \BLP procedure. 
    Because of the special structure of the bilinear constraints in $\mac{S}_c$, determining the aggregation weights of these constraints can be substantially simplified as discussed in Remark~\ref{remark:simplification}.
    In particular, the base constraint has a constant weight of $1$.
    The bilinear constraint $t_{\iota}^2 \in \mac{K}^0$, for $\iota \in L$, has a constant weight of $(h_{t_{\iota}} - h_{t_{\iota+1}}  + \delta_{t_\iota} + m \pi_{t_{\iota}} \beta_0)$. 
    This weight is non-negative because $h_{t_{\iota}} - h_{t_{\iota+1}} + \delta_{t_\iota} \geq 0$ by assumption, and $\beta_0 \geq 0$ as discussed above.
    The constraints $q^1_{\iota}, q^2_{\iota} \in \mac{K}^0$, for $\iota \in V$, have constant weights of $(\phi_{q_\iota} - m\pi_{q_\iota} \beta_0)^{+}$ and $(m\pi_{q_\iota}\beta_0 -  \phi_{q_\iota})^{+}$, respectively.
    The bilinear constraint $j^2 \in \mac{K}^0$, for $j \in M\setminus (P \cup Q)$, has a constant weight of $m \pi_j\beta_0$. 
    The bilinear constraint $j^3 \in \mac{K}^{j}$, for $j\in M\setminus \{t_1\}$, has a constant weight of $1$.
    The knapsack constraint $c \in \mac{T}^0$ has a weight of $m \beta_0 (1 - \vc{1} \tr \vc{y})$.
    Furthermore, the knapsack constraint $c \in \mac{T}^j$, for $j \in M$, has a weight of $m \beta_j y_j$. 
    Note that all these weights are non-negative. All other weights are zero.
    
\smallskip
The next step in the \BLP procedure is to aggregate the above weighted constraints to obtain the aggregated inequality 
    \begin{align}
	%\begin{split}
    & \sum_{j \in M} z y_j + \sum_{\iota \in L} (h_{t_{\iota}} - h_{t_{\iota+1}} + \delta_{t_\iota}) x_{t_{\iota}} +  \sum_{\iota \in V} \phi_{q_{\iota}} (1-x_{q_{\iota}}) + \sum_{j \in M} \big(h_{t_1} - h_j - \sum_{\iota \in V} \phi_{q_\iota}+ m \varepsilon \beta_j\big) y_j \nonumber \\
       & \quad {} + \sum_{i \in M} \sum_{j \in M} \sigma_{ij} x_{i} y_j \ge h_{t_1}, \label{eq:chance_lift_aggregated1}
    %\end{split}
    \end{align}
    where $\sigma_{ij}$, for $i,j \in M$, is defined as 
    \begin{equation}
    \label{eq:sigma}
        \sigma_{ij}=
        \begin{cases}
                -h_{t_{\iota}} + h_{t_{\iota+1}} - \delta_{t_\iota} - m \pi_{t_{\iota}}\beta_{j}, & i= t_{\iota}, \; \iota \in L, \;  j \in M,\\
  \phi_{q_{\iota}} - m\pi_{q_{\iota}}\beta_j, & i= q_{\iota}, \; \iota \in V, \; j \in M, \\
                -m\pi_i\beta_{j}, & i \in M \setminus (P \cup Q), \; j \in M.
        \end{cases}
    \end{equation}

    As the next step in the \BLP procedure, we apply the rules (R1) and (R2) to substitute the bilinear terms and $y$ variables in  \eqref{eq:chance_lift_aggregated1}. 
    Note that $\sigma_{ij}\le 0$ for $i \in P$ (i.e., $i= t_{\iota}, \iota \in L$) and $j \in M$ because $\delta_{t_\iota} \geq -h_{t_{\iota}} + h_{t_{\iota+1}}$ and $\beta_j \geq 0$ by definition. 
    We also have $\sigma_{ij}\le 0$ for $i \in M\setminus (P \cup Q)$ and $j \in M$.
    However, for $i \in Q$ (i.e., $i= q_{\iota}, \iota \in V$) and $j \in M$, $\sigma_{ij}$ can be positive or non-positive. 
    %Since $\Xi$ involves upper bounds for variables $\vc{x}$, we can apply the substitution rule (R0) as described in Remark~\ref{remark:upper bound} to the terms $\sum_{j \in P} \sum_{\iota \in V} \sigma_{q_{\iota}j} x_{q_{\iota}} y_j$ in \eqref{eq:chance_lift_aggregated1} to replace it with $\sum_{j \in P}  y_j \sum_{\iota \in V} \phi_{q_{\iota}}$.
    Using the above observations about the sign of $\sigma_{ij}$, we can separate the terms in $\sum_{j \in M} \sum_{i \in M} \sigma_{ij} x_{i} y_j$ in \eqref{eq:chance_lift_aggregated1} as follows:
    \begin{align}
    \label{eq:chance_lift_aggregated_3}\sum_{j \in M} \sum_{\iota \in L} (-& h_{t_{\iota}} + h_{t_{\iota+1}} - \delta_{t_\iota} - m \pi_{t_{\iota}}\beta_{j}) x_{t_\iota} y_j +  \sum_{j \in M} \sum_{i \in M \setminus (P \cup Q)}  (-m \pi_i \beta_{j}) x_i y_j \nonumber \\ 
    & \quad {} +  \sum_{j \in  M} \sum_{\iota \in V} (\phi_{q_{\iota}} - m \pi_{q_{\iota}}\beta_{j}) x_{q_{\iota}} y_{j}.   
    \end{align}
    Since \eqref{eq:general_set_4} matches the constraint form described in Remark \ref{remark: complementarity 1}, we can apply the substitution rule in that remark to all $x_i y_j$ terms with $i=j$ to eliminate them.
    Similarly, since \eqref{eq:general_set_5} matches the constraint form described in Remark~\ref{remark: complementarity 2}, we can apply the substitution rule (i) in that remark to all $x_i y_j$ terms with $i<j$ that have non-positive coefficients.
    This includes terms $x_{t_\iota} y_j$ for $j \in M$ and $\iota \in L$ with $t_{\iota} < j$ since their coefficient $-h_{t_{\iota}} + h_{t_{\iota+1}} - \delta_{t_\iota} - m \pi_{t_{\iota}}\beta_{j}$ is non-positive as established above, as well as terms $x_i y_j$ for $j \in M$ and $i \in M \setminus (P \cup Q)$ with $i < j$ since their coefficient $- m\pi_{i}\beta_{j}$ is non-positive.
    For term $x_{q_\iota} y_j$ for $j \in M$ and $\iota \in V$ with $q_\iota < j$, there are two cases. If the coefficient $\phi_{q_{\iota}} - m \pi_{q_{\iota}} \beta_j$ of this term is non-negative, then we may apply rule (R0) as described in Remark~\ref{remark:upper bound} to replace it with $(\phi_{q_{\iota}} - m\pi_{q_{\iota}} \beta_j)y_j$. If this coefficient is non-positive, we may apply the substitution rule (i) in Remark~\ref{remark: complementarity 2} to replace it with $(\phi_{q_{\iota}} - m\pi_{q_{\iota}} \beta_j)y_j$; both cases leading to the same substitution term.    
    Using these substitutions while further separating the terms in \eqref{eq:chance_lift_aggregated_3} based on the precedence of $i$ and $j$ indices in the bilinear terms, we obtain 
    \begin{align}
        %\begin{split}
            & \sum_{j \in M} \sum_{\iota \in L, t_{\iota} >j } (-h_{t_{\iota}} + h_{t_{\iota+1}}  - \delta_{t_\iota} - m\pi_{t_{\iota}} \beta_{j}) x_{t_\iota} y_j +  \sum_{j \in M} y_j  \sum_{\iota \in L, t_{\iota} < j} (-h_{t_{\iota}} + h_{t_{\iota+1}}   - \delta_{t_\iota} - m\pi_{t_{\iota}}\beta_{j})  \nonumber \\
            & \quad {} +  \sum_{j \in M} \sum_{i \in M \setminus (P \cup Q), i > j}  (-m\pi_i \beta_{j}) x_i y_j - \sum_{j \in M} y_j \sum_{i \in M \setminus (P \cup Q), i < j} m \pi_i \beta_{j}   \nonumber \\ 
             & \quad {} +  \sum_{j \in M}   \sum_{\iota \in V\setminus A_{j}, q_{\iota} > j} (\phi_{q_{\iota}} - m\pi_{q_{\iota}}\beta_{j}) x_{q_{\iota}} y_j +  \sum_{j \in M}  y_{j} \sum_{\iota \in V\setminus A_{j}, q_{\iota} < j} (\phi_{q_{\iota}} - m\pi_{q_{\iota}}\beta_{j})\nonumber\\
             & \quad {} + \sum_{j \in M}  \sum_{\iota \in A_{j}, q_{\iota} > j} (\phi_{q_{\iota}} - m\pi_{q_{\iota}}\beta_{j})  x_{q_{\iota}} y_{j} + \sum_{j \in M}  y_{j} \sum_{\iota \in A_{j}, q_{\iota} < j} (\phi_{q_{\iota}} - m\pi_{q_{\iota}}\beta_{j}).
\label{eq:chance_lift_aggregated_4} 
    \end{align}
    Note that in \eqref{eq:chance_lift_aggregated_4},  the coefficient of $x_{q_\iota} y_j$ term, for $j \in M$ and $\iota \in V \setminus A_j$ with $q_\iota >j$, is non-negative by \eqref{A_1}. 
    Therefore, we may apply rule (R0) as described in Remark~\ref{remark:upper bound} to replace this term with $(\phi_{q_{\iota}} - m\pi_{q_{\iota}} \beta_j)y_j$.
    Moreover, the coefficient of $x_{q_\iota} y_j$ term, for $j \in M$ and $\iota \in A_j$ with $q_\iota >j$, is non-positive by \eqref{A_1}.
    All remaining product terms $x_i y_j$ have a non-positive coefficient as discussed earlier. 
    Thus, we can apply rule (R1) in the \BLP procedure to eliminate these terms with non-positive coefficients, which reduces \eqref{eq:chance_lift_aggregated_4} to
    \begin{align}
        & \sum_{j \in M} y_j  \sum_{\iota \in L, t_{\iota} < j} (-h_{t_{\iota}} + h_{t_{\iota+1}} - \delta_{t_\iota} - m \pi_{t_{\iota}}\beta_{j})  - \sum_{j \in M} y_j \sum_{i \in M \setminus (P \cup Q), i < j}  m \pi_i \beta_{j}   \nonumber \\ 
             & \quad {} +  \sum_{j \in M} y_j  \sum_{\iota \in V\setminus A_{j}, q_{\iota} > j} (\phi_{q_{\iota}} - m \pi_{q_{\iota}}\beta_{j})   +  \sum_{j \in M}  y_{j} \sum_{\iota \in V\setminus A_{j}, q_{\iota} < j} (\phi_{q_{\iota}} - m\pi_{q_{\iota}}\beta_{j})\nonumber\\
             & \quad {} + \sum_{j \in M}  y_{j} \sum_{\iota \in A_{j}, q_{\iota} < j} (\phi_{q_{\iota}} - m\pi_{q_{\iota}}\beta_{j}).
             \label{eq:chance_lift_aggregated_4prime} 
    \end{align}
    Combining all $y_j$ terms, for $j \in M$, in \eqref{eq:chance_lift_aggregated_4prime} with the $y_j$ terms on the left-hand-side of \eqref{eq:chance_lift_aggregated1} yields 
    \begin{align}
            & \sum_{j \in M} \big(h_{t_1} - h_j - \sum_{\iota \in V} \phi_{q_\iota}+ m\varepsilon \beta_j\big) y_j + \sum_{j \in M} y_j  \sum_{\iota \in L, t_{\iota} < j} (-h_{t_{\iota}} + h_{t_{\iota+1}}  - \delta_{t_\iota} - m\pi_{t_{\iota}}\beta_{j}) \nonumber \\ 
            & \quad {}  - \sum_{j \in M} y_j \sum_{i \in M \setminus (P \cup Q), i < j}  m\pi_i \beta_{j}  +  \sum_{j \in M} y_{j} \sum_{\iota \in V\setminus A_{j}, q_{\iota} \neq j} (\phi_{q_{\iota}} - m\pi_{q_{\iota}} \beta_{j}) \nonumber \\
            & \quad {} + \sum_{j \in M}  y_{j} \sum_{\iota \in A_j, q_{\iota} < j} (\phi_{q_{\iota}} - m\pi_{q_{\iota}}\beta_{j}) \nonumber \\
            = & \sum_{j \in M} \Bigg(h_{t_{a_j+1}} - h_j -\sum_{\iota \in L, t_{\iota} < j}\delta_{t_\iota} - \sum_{\iota \in V} \phi_{q_\iota}+ m\bigg(\varepsilon - \sum_{\iota \in L, t_{\iota} <j} \pi_{t_{\iota}} - \sum_{i \in M\setminus(P \cup Q),i<j} \pi_i\bigg) \beta_j\Bigg) y_j \nonumber \\
            & \quad {} +
            \sum_{j \in M}  y_{j} \Bigg( \sum_{\iota \in V\setminus A_{j},  q_{\iota} \neq j } (\phi_{q_{\iota}} - m\pi_{q_\iota} \beta_{j})   + \sum_{\iota \in A_j, q_{\iota} < j} (\phi_{q_{\iota}} - m\pi_{q_\iota}\beta_{j})  \Bigg) \nonumber \\
            = & \sum_{j \in  M}\Bigg(h_{t_{a_j+1}} - h_j -\sum_{\iota \in L, t_{\iota} < j}\delta_{t_\iota} - \sum_{\iota \in V} \phi_{q_\iota}+ m\bigg(\varepsilon - F_{j} + \pi_j   + \sum_{\iota \in V, q_{\iota} <j} \pi_{q_{\iota}}\bigg) \beta_j\Bigg) y_j \nonumber \\
            & \quad {} +
            \sum_{j \in M} y_{j} \Bigg( \sum_{\iota \in V\setminus A_{j},  q_{\iota} \neq j } (\phi_{q_{\iota}} - m\pi_{q_\iota} \beta_{j})  +  \sum_{\iota \in A_{j}, q_\iota \neq j} (\phi_{q_{\iota}} - m\pi_{q_\iota} \beta_{j}) - \sum_{\iota \in A_j, q_{\iota} > j} (\phi_{q_{\iota}} - m\pi_{q_\iota}\beta_{j})  \Bigg) \nonumber \\
            = & \sum_{j \in M} \Bigg(h_{t_{a_j+1}} - h_j -\sum_{\substack{\iota \in L\\ t_{\iota} < j}}\delta_{t_\iota} - \mathbb{I} (j \in Q) \phi_{j} + m\bigg(\varepsilon - F_{j} + \pi_j +  \sum_{\substack{\iota \in V\\ q_\iota < j}} \pi_{q_\iota} - \sum_{\iota \in V} \pi_{q_\iota} + \mathbb{I}(j \in Q)\pi_j \bigg) \beta_j  \Bigg) y_j \nonumber \\ 
            & \quad {} - \sum_{j \in M} \sum_{\substack{\iota \in A_{j}\\  q_{\iota}>j}} y_j (\phi_{q_{\iota}} - m\pi_{q_\iota}\beta_{j}), 
            \label{eq:chance_lift_aggregated_7_general2}             
    \end{align}
    where the second equality follows from the definition of $F_j$ and the fact that $\{\iota \in A_j:q_\iota \neq j\}=\{\iota \in A_j: q_\iota < j\} \cup \{\iota \in A_j: q_\iota > j\} $, while the third equality is due to $\{\iota \in V \setminus A_j:q_\iota \neq j\} \cup \{\iota \in A_j: q_\iota \neq j\} = \{\iota \in V: q_\iota \neq j\} $
    %$A_j \cup V \setminus A_j = V$, 
    for $j \in M$, and the fact that $\sum_{\iota \in V} \phi_{q_{\iota}} - \sum_{\iota \in V, q_{\iota} \neq j} \phi_{q_{\iota}} = \mathbb{I}(j \in Q) \phi_{j}$. 
    For each $j \in M$, the coefficient of $y_{j}$ in  \eqref{eq:chance_lift_aggregated_7_general2} can be simplified to 
    \begin{equation}
            m\beta_{j} \Big ( \varepsilon - F_{j} + \pi_j - \sum_{\substack{\iota \in V\\ q_\iota > j}} \pi_{q_\iota} + \sum_{\substack{\iota \in A_{j}\\   q_{\iota} > j}} \pi_{q_\iota} \Big) -h_{j} +  h_{t_{a_j+1}} -\sum_{\substack{\iota \in L\\ t_{\iota} < j}}\delta_{t_\iota} -  \mathbb{I}(j \in Q) \phi_{j} - \sum_{\substack{\iota \in A_{j}\\  q_{\iota} > j}} \phi_{q_{\iota}}, \label{eq:chance_lift_aggregated_8}
    \end{equation}
    which is enforced to be non-positive by \eqref{A_3}. 
    This allows us to apply rule (R2) in the \BLP procedure to eliminate these terms from the left-hand-side of \eqref{eq:chance_lift_aggregated1}.
    Finally, we apply rule (R1) to $\sum_{j \in M} zy_j$ on the left-hand-side of \eqref{eq:chance_lift_aggregated1} to replace it with $z$.
    As a result, this inequality reduces to \eqref{eq:chance_lift_ineq_generic}. 
    The remark following Theorem~\ref{thm:BLP} implies that \eqref{eq:chance_lift_ineq_generic} is valid for $\conv(\mac{S}_c)$.    
    \hfill \Halmos    
\endproof

\smallskip
The next result provides a necessary condition for the \BLP inequalities derived in Theorem~\ref{thm:chance_lift_generic} to be facet-defining for the convex hull of the underlying bilinear set.

\begin{proposition} \label{prop:facet}
    The \BLP inequality \eqref{eq:chance_lift_ineq_generic} is facet-defining for $\proj_{(z,\vc{x})} \conv(\mac{S}_c)$ only if
at least $2m + 1$ of the inequalities in \eqref{A_1} and \eqref{A_3} for $j \in M$, and $\pi_i \beta_j \geq 0$ for $j \in M, \ i \in M \setminus (P \cup Q), \ i > j$, hold at equality.
\end{proposition}

\proof{Proof}

\begin{comment}
{\color{red}
Furthermore, the \BLP inequality \eqref{eq:chance_lift_ineq_generic} is facet-defining for $\proj_{(z,x)} \conv(\mac{S}_c)$ only if
at least $2m + 1$ of the following constraints are satisfied.
\begin{subequation} \label{eq:necessary}
\begin{align} 
    & h_{t_{\iota}} + h_{t_{\iota+1}} - \delta_{t_\iota} - m \pi_{t_{\iota}}\beta_{j} = 0   &j \in M, \ \iota \in L \\
    &  m \pi_i \beta_{j} = 0 & j \in M, \ i \in M \setminus (P \cup Q) \\ 
    &  \phi_{q_{\iota}} - m \pi_{q_{\iota}}\beta_{j} = 0 & j \in  M, \ \iota \in V   \\
    & m\beta_{j} \Big ( \varepsilon - F_{j} + \pi_j - \sum_{\substack{\iota \in V\\ q_\iota > j}} \pi_{q_\iota} + \sum_{\substack{\iota \in A_{j}\\   q_{\iota} > j}} \pi_{q_\iota} \Big) -h_{j} +  h_{t_{a_j+1}} -\sum_{\substack{\iota \in L\\ t_{\iota} < j}}\delta_{t_\iota} -  \mathbb{I}(j \in Q) \phi_{j} - \sum_{\substack{\iota \in A_{j}\\  q_{\iota} > j}} \phi_{q_{\iota}} = 0 & j \in M.
\end{align}
\end{subequation}
}
\end{comment}

The remark following Theorem~\ref{thm:BLP} implies that, if the \BLP inequality \eqref{eq:chance_lift_ineq_generic} is facet-defining for $\proj_{(z,\vc{x})} \conv(\mac{S}_c)$, then it satisfies condition (C1). 
    In the \BLP procedure outlined in the proof of Theorem~\ref{thm:chance_lift_generic}, the total number of constraints in the sets $\mac{K}^j$ and $\mac{T}^j$ for $j \in M$ that participate in the aggregation with nonzero weights is at most $3m$, since by definition, at most one of the constraints $q_{\iota}^1$ and $q_{\iota}^2$ in $\mac{K}^0$ can have a nonzero weight for each $\iota \in V$. According to the aggregation procedure described in the proof of Theorem~\ref{thm:chance_lift_generic}, rule (R1) is applied to the terms with positive coefficients, represented by $\sum_{j \in M} zy_j$, where the maximum coefficient equals one and is achieved by $m$ terms.
    Thus, it follows from the condition (C1) of the \BLP procedure that the number of bilinear and $y$ terms whose coefficients become zero must be at least $3m + (1-m) = 2m + 1$.
    % THIS IS THE PROOF FOR THE PART THAT IS COMMENTED OUT IN THE THEOREM STATEMENT: As a result, $2m + v + 1$ of the coefficients of the bilinear terms in \eqref{eq:chance_lift_aggregated_3} and the coefficients of the $y$ variables in \eqref{eq:chance_lift_aggregated_8} must become zero, which yields the system of equations in \eqref{eq:necessary}.
    In particular, $2m + 1$ of the coefficients of the bilinear terms in \eqref{eq:chance_lift_aggregated_4}, together with the coefficients of the $y$-variables in \eqref{eq:chance_lift_aggregated_8}, must become zero.
Note that for the coefficient of $x_{t_{\iota}}y_j$ with $j \in M$ and $\iota \in L$, where $t_{\iota}>j$, to vanish in \eqref{eq:chance_lift_aggregated_4}, it must hold that $m\pi_{t_{\iota}}\beta_j = -h_{t_{\iota}} + h_{t_{\iota+1}} -\delta_{t_{\iota}}$.
    Since the left-hand side of this equality is non-negative and the right-hand side is non-positive by assumption, we must have $\beta_j = 0$ and $\delta_{t_{\iota}} = -h_{t_{\iota}} + h_{t_{\iota+1}}$. This implies that the coefficient of $x_{t_{\iota}}$ becomes zero in \eqref{eq:chance_lift_ineq_generic}, meaning that index $t_{\iota}$ can be removed from set $P$, which in turn corresponds to a different parameter configuration for obtaining the inequality.
Consequently, we may exclude the coefficients of $x_{t_{\iota}}y_j$ for $j \in M$ and $\iota \in L$ with $t_{\iota}>j$ from the set of coefficients selected to be zero.
Thus, the remaining coefficients to be zero correspond to the conditions represented by constraints \eqref{A_1}, \eqref{A_3} for all $j \in M$, and $\pi_i \beta_j \geq 0$ for all $j \in M$ and $i \in M \setminus (P \cup Q)$ with $i > j$, among which at least $2m + 1$ must hold at equality.    
    \hfill \Halmos    
\endproof

\smallskip
The family of \BLP inequalities described in Theorem~\ref{thm:chance_lift_generic} subsumes several existing results in the literature, including those in Propositions~\ref{prop:mixing}–\ref{prop:chance_lift_simge}, as well as Proposition~\ref{prop:chance_lift_permutation} in the case of a uniform probability distribution, as will be demonstrated in Section~\ref{subsec:BLP family-special}.
At the same time, this family encompasses a substantially broader variety of valid inequalities for $\conv(\mac{F}_c)$ than previously available results, as evidenced by the computational experiments presented in Section~\ref{sec:numerical}.
To further illustrate, we next present an example of facet-defining \BLP inequalities derived from Theorem~\ref{thm:chance_lift_generic} that cannot be obtained using Proposition~\ref{prop:chance_lift_permutation}, which provides the most general form of known inequalities for $\mac{F}_c$ in the literature, as discussed in Section~\ref{sec:background}.

\begin{comment}
To illustrate this, we next present an example of facet-defining \BLP inequalities derived from Theorem \ref{thm:chance_lift_generic} that can be obtained using Proposition \ref{prop:chance_lift_permutation}.

\begin{example}
\label{ex:2}
Let $m=10$, with $h=\{20,18,14,11,6,5,4,3,2,1\}$. 
Let $p=4$, $r=2$, $P=\{1,2\}$, and $Q=\{5,6\}$.  %Consider a permutation $\{2,1\}$ of $\{1,2\}$, 
Let $\phi^5=5$ and $\phi^6=3$, and choose $\delta_1=\delta_2=0$. 
Define $\beta_0=2.33$. 
Moreover, choose $A_4=\{6\}$ and $A_j=\emptyset$ for $ j \neq 4$. %, $A_4=\{6\}$, and $A_6=A_6=A_7=A_8=A_9=A_{10}=\emptyset$.
Then, the constraints in \eqref{A} can be written as 
\begin{align*}
    & 0 \le  \beta_1 \le \min\{\phi^4,\phi^5\}, \\
    & \beta_1 \le 0,\\
    & 0 \le  \beta_2 \le \min\{\phi^4,\phi^5\}, \\
    & \beta_2 \le 0,\\
    & 0 \le  \beta_3 \le \min\{\phi^4,\phi^5\}, \\
    & \beta_3 \le 0, \\ 
    & \phi^5 \ge \beta_4 \ge  \phi^6, \\
    & \phi^6 \ge 3, \\ 
    & \beta_5 \le  \phi^6, \\
    & \beta_5 \ge 8 - \phi^5, \\ 
    & \beta_6 \ge 9 - \phi^6,\\
    & \beta_7 \ge 5, \; 
    \beta_8 \ge 3.67, \; 
    \beta_9 \ge 3, \;
    \beta_{10} \ge 2.6.
\end{align*}
Observe that 
$\beta_1=\beta_2=\beta_3=0$, 
$\beta_4 =3$, $\beta_{5}=3$, $\beta_6 =6$, $\beta_7=5$, $\beta_8=3.67$, $\beta_9=3$, and $\beta_{10}=2.6$ hold in the above set of constraints. 
Hence, constraint 
\begin{equation*}
    z+ 2 x_1 + 4 x_2 + 5(1-x_5)+ 3(1-x_{6})\ge 20 
\end{equation*}
is valid. 
\hfill	$\blacksquare$
\end{example}
\end{comment}

\begin{example}
\label{ex:3}
Consider the instance of $\mac{F}_c^=$ given in Example~\ref{ex:1}, where $m=10$, $\{h_1, \dotsc, h_m\}=\{20,18,14,11,6,5,4,3,2,1\}$, and $p=\vartheta=4$.
We established in Example~\ref{ex:1} that the inequality~\eqref{eq:Ex1} cannot be represented by the family of valid inequalities in Proposition~\ref{prop:chance_lift_permutation}.
In contrast, we show here that this inequality belongs to the family of \BLP inequalities~\eqref{eq:chance_lift_ineq_generic} derived in Theorem~\ref{thm:chance_lift_generic}.
For this problem instance, we have $r=4$, $P=\{1,4\}$, and $Q=\{5,6\}$.
Following the settings in that theorem, we may choose $\delta_1=\delta_4=-3$, and select $A_5=\{1\}$, $A_6=\{1,2\}$, and $A_j=\emptyset$ for all $j \in M \setminus \{5,6\}$.
%Define $\beta_0=2$.
Then, we can write the conditions in~\eqref{A} together with the non-negativity requirements for $\vc{\beta}$ as follows:
\begin{align*}
    & 0 \le \beta_1 \le \min\{\phi_5,\phi_6\}, \\
    & 0 \le  \beta_2 \le \min\{\phi_5,\phi_6, 4\}, \\
    & 0 \le  \beta_3 \le \min\{\phi_5,\phi_6\}, \\
    & 3 \le  \beta_4 \le \min\{\phi_5,\phi_6\}, \\
    & \max\{0,6 - \phi_5\} \le \beta_5 \le  \phi_6, \\
    & \beta_6 \ge \max\{0,7 -\phi_6\}, \\ 
    & \beta_7 \ge 4, \; 
    \beta_8 \ge 3, \; 
    \beta_9 \ge 2.5, \;
    \beta_{10} \ge 2.2.
\end{align*}
%\begin{align*}
%    & 0 \le \beta_1 \le \min\{\phi^5,\phi^6\}, \\
%    & 0 \le  \beta_2 \le \min\{\phi^5,\phi^6\}, \\
%    & \beta_2 \le 4,\\
%    & 0 \le  \beta_3 \le \min\{\phi^5,\phi^6\}, \\
%    & 0 \le  \beta_4 \le \min\{\phi^5,\phi^6\}, \\
%    & \beta_4 \ge 3,\\
%    & \beta_5 \le  \phi^6, \\
%    & \beta_5 \ge 6 - \phi^5,\\
%    & \beta_6 \ge 7 -\phi^6, \\ 
%    & \beta_7 \ge 4, \; 
%    \beta_8 \ge 3, \; 
%    \beta_9 \ge 2.5, \;
%    \beta_{10} \ge 2.2.
%\end{align*}
It is clear that, with the choice of $\phi_5 = \phi_6 = 3$ to attain the lower bound value of $\beta_4$, there exist values $\beta_1=\beta_2=\beta_3=0$, $\beta_4=\beta_5=3$, $\beta_6=\beta_7=4$, $\beta_8=3$, $\beta_9=2.5$, and $\beta_{10}=2.2$ that satisfy the above set of conditions.
Hence, the resulting \BLP inequality
\begin{equation*}
    z+ 6 x_1 + 2 x_4 + 3(1-x_5)+ 3(1-x_{6})\ge 20, 
\end{equation*}
which matches \eqref{eq:Ex1}, is valid for $\conv(\mac{F}_c^=)$.
\hfill	$\blacksquare$
\end{example}

%First, to the best of our knowledge, this is the first result to explicitly derive a closed-form family of valid inequalities for general chance constraints under an arbitrary probability distribution.
%In contrast, closed-form results in the literature are limited to the special case of uniform probability distributions.
%Second, relative to the most general known result in the literature as outlined in Proposition \ref{prop:chance_lift}, Theorem \ref{thm:chance_lift_generic} permits sets $Q \subseteq \{r+2, \dotsc, m\}$, thereby yielding an exponentially larger set of options than Proposition \ref{prop:chance_lift}, which restricts $Q$ to $\{p+1, \dotsc, m\}$; see Example \ref{ex8:ctd} below.
%Third, the coefficients $\phi_{q_{\iota}}$ for $\iota \in V$ in \eqref{eq:chance_lift_ineq_generic} may take any value combinations, provided that conditions \eqref{A_1} and \eqref{A_3} are satisfied. This flexibility yields a factorially larger family of inequalities compared to Proposition \ref{prop:chance_lift}, which enforces a nondecreasing order on the coefficients $\phi_{q_{\iota}}$; see Example ?? below.

%Recall Example \ref{ex:8}, where we studied the inequality \eqref{eq:setQ-example} that did not fall within the family described by Proposition \ref{prop:chance_lift} due to the restriction of $Q$. Here, we demonstrate that this inequality is, in fact, encompassed by the family of inequalities \eqref{eq:chance_lift_ineq_generic} derived in Theorem \ref{thm:chance_lift_generic}.

%\addtocounter{example}{-1} % step back by one

\smallskip
In light of the above example, we note that although the family of \BLP inequalities in Theorem~\ref{thm:chance_lift_generic} includes inequalities that do not belong to the family presented in Proposition~\ref{prop:chance_lift_permutation}, it does not fully subsume that family under the general probability setting, as illustrated in Example~\ref{ex:7}. Nevertheless, the computational results in Section~\ref{sec:numerical} demonstrate that the family of \BLP inequalities provides a substantially larger coverage of the convex hull for the studied instances compared to the inequalities in Proposition~\ref{prop:chance_lift_permutation}.

\begin{example}{\cite[Example~1]{zhao2017polyhedral}}
\label{ex:7}
    Consider an instance of $\mac{F}_c$, where $m=10$, $\varepsilon=0.5$, $\{h_1, \dotsc, h_{m}\}=\{40, 38, 34, 31, 26, 16, 8, 4, 2, 1\}$. Let $\pi_1=\dotsc=\pi_4=\varepsilon/4$ and $\pi_5=\dotsc=\pi_{10}=\varepsilon/6$. 
   For this problem instance, we have  $p=6$ and $\vartheta=4$. It is shown in 
   \cite[Example~1]{zhao2017polyhedral} that 
   \begin{equation*}
        z + 2 x_1  + 4(1-x_4) + 4(1-x_7) + 8(1-x_8) \ge 40
    \end{equation*}
    is facet-defining for $\conv(\mac{F}_c)$, where $r=1$, $P=\{1\}$, and $Q=\{4,7,8\}$. Following the setting in Theorem \ref{thm:chance_lift_generic}, we may choose $\delta_1=0$. However, it can be shown that for the choice $\phi_4 = \phi_7 = 4$ and $\phi_8 = 8$, which is required to produce this inequality form, there does not exist any $\beta_j \in \Re_{+}$ for some $j \in M$ that satisfies \eqref{A} for any choice of $A_j$. Thus, this inequality is not captured by the family of \BLP inequalities in Theorem \ref{thm:chance_lift_generic}. \hfill	$\blacksquare$
\end{example}

\smallskip
As noted earlier, beyond encompassing new classes of valid inequalities that cannot be captured by existing results (as illustrated in Example~\ref{ex:3}), Theorem~\ref{thm:chance_lift_generic} also subsumes many known results on valid inequalities for $\mac{F}_c$.
A summary of the relationship between the families of inequalities developed in the literature (Propositions~\ref{prop:mixing}--\ref{prop:chance_lift_permutation}) and the \BLP inequalities established in Theorems~\ref{thm:chance_lift_generic} in covering $\conv(\mac{F}_c)$ under a general probability distribution is illustrated in Figure~\ref{fig: general_prob}.
More specifically, for the widely studied special case where the underlying probability distribution is uniform, we show in the next section that a subfamily of the \BLP inequalities in \eqref{eq:chance_lift_ineq_generic} strictly subsumes all existing results in the literature, as presented in Section~\ref{sec:chance_existing}.

\begin{figure}[htbp]
    \centering
    \begin{tikzpicture}[line width=0.9pt, font=\small]

    % --- canvas scale ---
    \def\W{5}   % rectangle width
    \def\H{4}    % rectangle height

    % --- frame rectangle ---
    \draw (-\W/2-3.4,-\H/2) rectangle (\W/2,\H/2);
    \node[above right] at (-\W/2-3.4,-\H/2) {$\conv(\mac{F}_c)$};

    % --- parameters ---
    \def\cy{0}             % vertical center
    \def\anglestart{10}    % open gap start (degrees)
    \def\angleend{350}     % open gap end

    % --- nested, right-shifted ellipses (each contains the previous) ---
    % Format: center-x / x-radius / y-radius / label
    \foreach \cx/\a/\b/\c/\i in {%
      -3.5/0.6/0.4/Prop~\ref{prop:mixing}/1,   % smaller innermost ellipse
      -2.9/1.3/0.6/Prop~\ref{prop:chance}/2,
      -2.3/2/0.8/Prop~\ref{prop:chance_lift}/3,
      -1.7/2.7/1/Prop~\ref{prop:chance_lift_simge}/4
    %-.7/4.3/1.9/Prop~\ref{prop:chance_lift_permutation}/5%
    }{
        \draw (\cx,\cy) ellipse [x radius=\a, y radius=\b, start angle=\anglestart, end angle=\angleend];
        \node[anchor=east] at (\cx+\a, \cy) {\c};
    }

% --- outer ellipse that contains all and grows to the LEFT ---
\draw (-1.1,0) ellipse [x radius=3.4, y radius=1.6, start angle=\anglestart, end angle=\angleend];
    \node[anchor=east] at (2.3, \cy) {Prop~\ref{prop:chance_lift_permutation}};

% --- outer ellipse that contains all and grows to the LEFT ---
\draw (-2.3,0) ellipse [x radius=3.4, y radius=1.6, start angle=\anglestart, end angle=\angleend];
    \node[anchor=west] at (-5.7, \cy) {Thm~\ref{thm:chance_lift_generic}};

    \end{tikzpicture}
    \caption{Relationship between existing inequalities (Propositions~\ref{prop:mixing}--\ref{prop:chance_lift_permutation}) and \BLP inequalities (Theorems~\ref{thm:chance_lift_generic}) in characterizing $\conv(\mac{F}_c)$ under a general probability distribution.}
    \label{fig: general_prob}
\end{figure}
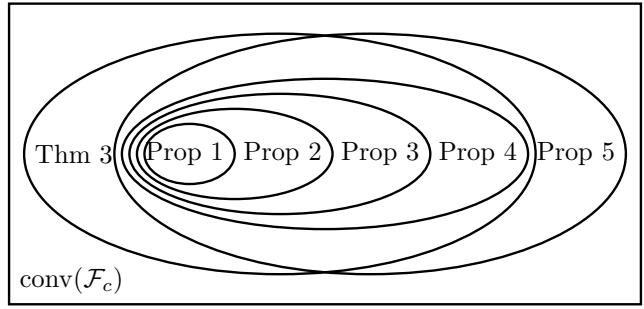

%%%%%%%%%%%%%%%%%%%%%%%%%%%%%%%%%%%%%%%%%%%%%%%%%%%%%
\subsection{Valid Inequalities for the Mixing Set with a Cardinality Constraint} \label{subsec:BLP family-special}

In this section, we present a special subfamily of the \BLP inequalities in \eqref{eq:chance_lift_ineq_generic} corresponding to the case of a uniform probability distribution.
In this setting, the explicit parameter values can be readily determined and directly incorporated into the \BLP inequality, eliminating the need to verify the existence of $\vc{\beta}$ satisfying the conditions in \eqref{A}.
This leads to an efficient approach for generating closed-form valid inequalities for CCPs.
%where $Q$ is restricted to $\{p+1, \dotsc, m\}$, $|Q|=p-r$, and explicit values are assigned to the coefficients $\phi_{i}$, $i \in Q$. 
We show in Corollary~\ref{cor:special_case} that this subfamily encompasses all families of valid inequalities previously derived in the literature, as presented in Propositions~\ref{prop:mixing}--\ref{prop:chance_lift_permutation}. Moreover, it also includes valid inequalities that cannot be captured by Proposition~\ref{prop:chance_lift_permutation} (see Example~\ref{ex:9}).
%Even this special sub-family yields a factorially larger collection of valid inequalities compared to what has been documented in the literature.

%Theorem \ref{thm:chance_lift_generic} introduces a broad family of valid inequalities, derived from the existence of feasible $\beta_j$ to \eqref{A_1}–\eqref{A_3} for certain choices of $A_j$, $j \in M \setminus P$. 
%However, it does not provide guidance on how to select the $A_j$’s for a given $(\phi_{q_1}, \dotsc, \phi_{q_{p-r}}) \in \Re^{p-r}_+$.  In contrast, Theorem \ref{thm: chance_liftVIAgeneric_permutation} establishes the validity of \eqref{eq:chance_lift_ineq_generic} for a specific structure of $(\phi_{q_1}, \dotsc, \phi_{q_{p-r}})$, by carefully constructing the corresponding $A_j$’s. As we shall see shortly, this structure generalizes the choices outlined in Proposition \ref{prop:chance_lift}. 
%In a series of propositions we show that for the inequalities in Propositions \ref{prop:mixing}--\ref{prop:chance_lift}, we construct $\beta_{j}\ge 0$ and $A_j $, $j \in M\setminus P$ that satisfy \eqref{A_1}--\eqref{A_3}.  

\blue{As in Theorem~\ref{thm:chance_lift_generic}, we adopt the convention that whenever a set is empty, any conditions or expressions involving that set in the statements and proofs are understood to be omitted.}

\begin{theorem}
    \label{thm: chance_liftVIAgeneric_permutation}
    Consider the mixing set with a cardinality constraint $\mac{F}_c^=$, as defined in \eqref{eq:cardinality}.
    Let $r \in \{1,\dotsc,p\}$, $l \in \{1,\dotsc,r\}$, and $v \in \{0,\dotsc, p-r\}$. 
    \blue{Define $V:=\{1,\dotsc,v\}$. If $v>0$, let $s_\iota=p-r-v+\iota$ for each $\iota\in V$; otherwise, set $s_1=p-r+1$.}
    Define $L:=\{1,\dotsc, l\}$ and let $P=\{t_\iota: \iota \in L\} \subseteq \{1,\dotsc, r\}$, where $t_1 < \cdots < t_l$. 
    Select  $\delta_{t_\iota} \in \Re$ for  $\iota \in L$ such that $\delta_{t_\iota} \ge h_{t_{\iota+1}} - h_{t_\iota}$, $\sum_{\iota=1}^{k-1} \delta_{t_\iota} \ge 0$ for $k=2, \dotsc, l$, where  $t_{l+1}:=r+s_1$, \blue{and $\sum_{\iota \in L} \delta_{t_{\iota}} \le  h_{r+s_1} -h_{r+s_2}$ for $v>0$, whereas $\sum_{\iota \in L} \delta_{t_{\iota}} \le  h_{r+s_1}$ for $v=0$. }
    Let $Q = \{q_1, \dotsc,q_{v}\}$, where $r+s_\iota+1 \leq q_\iota \leq m$ for $\iota \in V$. Define 
\begin{subequations}
\label{alpha}
    \begin{align}
     & \phi_{q_{1}} = h_{r+s_1}-h_{r+s_2} -\sum_{k \in L} \delta_{t_k}, \label{alpha_1}\\
     & \phi_{q_{\iota}} = \max\Bigg\{\phi_{q_{\iota-1}}, h_{r+s_1}- h_{r+s_\iota+1} - \sum_{k \in L} \delta_{t_k} - \sum_{\substack{k=1\\ q_{k} \ge r+s_\iota+1}}^{\iota-1} \phi_{q_{k}} \Bigg\}, \quad \iota=2, \dotsc, v. \label{alpha_2}
    \end{align}
\end{subequations}
Then, the \BLP inequality
\begin{equation} \label{eq:chance_lift_ineq_generic_permutation}
    z + \blue{\sum_{\iota \in L} } (h_{t_{\iota}} - h_{t_{\iota+1}}+ \delta_{t_\iota}) x_{t_{\iota}} + \blue{\sum_{\iota \in V}} \phi_{q_{\iota}} (1-x_{q_{\iota}}) \ge h_{t_{1}}, 
\end{equation}
is valid for $\proj_{(z,\vc{x})} \conv(\mac{S}_c^=)$; hence, is valid for $\conv(\mac{F}_c^=)$.

\end{theorem}

\proof{Proof}
By Theorem~\ref{thm:chance_lift_generic}, it suffices to show that for each $j \in M$, there exists $\beta_j \in \Re_{+}$ that satisfies \eqref{A} for some $A_j \subseteq V$. 
First, note that $\phi_{q_{\iota}} \geq 0$ for all $\iota \in V$ because of \eqref{alpha_1} and \eqref{alpha_2}, and the assumption that $\sum_{\iota \in L} \delta_{t_{\iota}} \le  h_{r+s_1} -h_{r+s_2}$. 
Following the notation in Theorem~\ref{thm:chance_lift_generic}, for this special case, we have $\varepsilon = p/m$, $\pi_j =  1/m$, and $F_j = j/m$ for $j \in M$.
We present the results for multiple cases of $j \in M$ as follows. 
%We break down the proof into two parts for $j \in M \setminus  (P \cup Q)$ and $ j \in Q$. 

\begin{itemize}
    \item[(I)] Assume that $j \in M \setminus  (P \cup Q)$ and $j > \max\big\{p+1, \max\{q_{\iota}: \iota \in V\}\big\}$.
We choose $A_j=\emptyset$.
In this case, constraint \eqref{A_1} does not exist, and constraint \eqref{A_3} can be written as 
\begin{equation*}
    \beta_{j}   \ge  (h_{r+s_1} - h_{j} - \sum_{\iota \in L} \delta_{t_\iota})/(j-p-1), \label{2c}
\end{equation*}
using the fact that $j >p+1$.
%Note that in this relation, $a_j = l$ by definition, which leads to $h_{t_{a_j + 1}} = h_{t_{l + 1}} = h_{r+s_1}$.
%As a result, $h_{r+1} - h_j \geq 0$.
Therefore, we can simply pick $\beta_j = \max\big\{0, (h_{r+s_1} - h_{j} - \sum_{\iota \in L} \delta_{t_\iota})/(j-p-1)\big\}$ to satisfy the conditions in Theorem~\ref{thm:chance_lift_generic}.

\item[(II)] Assume that $j \in M \setminus  (P \cup Q)$ and $j = r+s_1$.
We choose $A_j=\emptyset$.
In this case, $q_{\iota} > j$ for all $\iota \in V$ as $q_{\iota} \geq r+s_{\iota}+1 \geq r+s_1+1 > r+s_1 = j$, where the first inequality follows from the definition of $q_{\iota}$ and the second inequality is implied by the definition of $s_{\iota}$.
As a result, constraint \eqref{A_1} can be written as
\begin{equation*}
        \beta_j \le \min\{\phi_{q_\iota}: \iota \in V \}. \label{4a}
\end{equation*}   
Furthermore, constraint \eqref{A_3} is trivially satisfied for this case since the right-hand-side will reduce to $-\sum_{\iota \in L} \delta_{t_\iota}$, and the coefficient of $\beta_j$ on the left-hand-side is also zero as $j-1-p+v=0$ due to $j=r+s_1$ and $s_1=p-r-v+1$.  
Therefore, we can simply select $\beta_j = 0$ to satisfy the conditions in Theorem~\ref{thm:chance_lift_generic}.

\item[(III)] Assume that $j \in M \setminus  (P \cup Q)$ and $j < r+s_1$.
We choose $A_j=\emptyset$.
Similarly to the previous case, we have $q_{\iota} > j$ for all $\iota \in V$.
Therefore, constraints \eqref{A_1} \blue{reduce to} 
\begin{align}
    & \beta_j \le \min\{\phi_{q_\iota}: \iota \in V \}. \label{eq:IIIa}    
\end{align}   
Further, the coefficient of $\beta_j$ in \eqref{A_3} reduces to $j - 1 - p + v = j - r - s_1 < 0$, where the equality follows from the definition of $s_1 = p-r-v+1$, and the inequality follows the assumption of this case.
Thus, \eqref{A_3} can be written as 
\begin{align}
    & \beta_{j}    \le   (h_j - h_{t_{a_j+1}} + \sum_{\substack{\iota \in L\\t_\iota <j}} \delta_{t_\iota})/(  r + s_1 - j ), \label{eq:IIIb} 
\end{align}   
where $ h_j - h_{t_{a_j+1}} \geq 0$ and $\sum_{\substack{\iota \in L\\t_\iota <j}} \delta_{t_\iota}\ge 0 $ by definition. 
Because the right-hand-side values of both inequalities \eqref{eq:IIIa} and \eqref{eq:IIIb} are non-negative, we can pick $\beta_j = 0$ to satisfy the conditions in Theorem~\ref{thm:chance_lift_generic}.
%\begin{equation}
%\label{5ac}
%    0 \le \beta_j \le \min\{(h_j - h_{t_{a_j+1}})/( r+1-j), \min\{\phi_{q_\iota}: \iota \in V \}\}. 
%\end{equation}

\item[(IV)] 
Assume that $j \in M \setminus  (P \cup Q)$ and $\max\{p+1, \min\{q_{\iota}: \iota \in V\}\} < j < \max\{p+1,\max\{q_{\iota}: \iota \in V\}\}$. Clearly, if $\max\{p+1, \min\{q_{\iota}: \iota \in V\}\}=\max\{p+1, \max\{q_{\iota}: \iota \in V\}\}=p+1$, then this case does not occur. 
Otherwise, we choose $A_j=\{\iota \in V: q_{\iota} > j\}$. 
Hence, constraints \eqref{A_1} and \eqref{A_3} can be written as 
\begin{align*}
    &  \beta_{j} \ge \max\{\phi_{q_\iota}: \iota \in A_j\}, \\
    & \beta_{j}    \ge   \big(h_{r+s_1} - h_{j} - \sum_{\iota \in L} \delta_{t_\iota} - \sum_{\iota \in A_j} \phi_{q_\iota}\big)/(j-p-1), 
\end{align*}   
where the second inequality follows from the fact that $j > p+1 $ by assumption. The above inequalities provide lower bounds for $\beta_j$.
Therefore, we can pick
\begin{equation*}
    %\label{3ac}
    \beta_j = \max\Big\{(h_{r+s_1} - h_{j} - \sum_{\iota \in L} \delta_{t_\iota} -\sum\nolimits_{ \iota \in A_j}\phi_{q_\iota})/( j - p -1), \max\{\phi_{q_\iota}: \iota \in A_j\}\Big\} 
\end{equation*} 
to satisfy the conditions in Theorem~\ref{thm:chance_lift_generic} as $\phi_{q_\iota} \geq 0$ by definition.

\item[(V)] Assume that $j \in M \setminus  (P \cup Q)$ and $r+s_1+1 \le j \leq p+1$. 
We can write that $ j=r+s_\iota+1$ for some $ 1 \le \iota \le v$. Choose $A_{j}=\big\{ k \in \{1,\dotsc,\iota\} : q_{k} \ge j \big\}$. It is clear that $\iota \in A_j$ as $q_{\iota} \ge r+s_\iota +1 =j$ by assumption, where the equality across the chain is impossible since it would imply $j \in Q$. Thus, $\iota \in A_j$ and $q_\iota>j$. 
As a result, constraints 
\eqref{A_1} and \eqref{A_3} can be written as 
\begin{subequations}
\begin{align}
    &  \phi_{q_{\iota}}=\max\{\phi_{q_{k}}: k \in A_j, \; q_k >j \} \le \beta_{j} \le \min\{\phi_{q_{k}}: k \in V\setminus A_j, \; q_k >j\}=\phi_{q_{\iota+1}}, \label{Va}\\
    &  \phi_{q_{\iota}} \ge   h_{r+s_1} - h_{j} -\sum_{k \in L} \delta_{t_k} -\sum_{\substack{k=1\\q_k > j }}^{\iota-1} \phi_{q_{k}} = h_{r+s_1} - h_{r+s_\iota +1} -\sum_{k \in L} \delta_{t_k} -\sum_{\substack{k=1\\q_k \geq r+s_\iota +1}}^{\iota-1} \phi_{q_{k}}. \label{Vc}
\end{align}   
\end{subequations}
The left equality in \eqref{Va} follows from the definition of $\phi_{q_{k}}$ for $k \in V$ in \eqref{alpha}, and the facts that $\iota \in A_j$ and $q_{\iota} > j$.  
For the right equality in \eqref{Va}, we claim that $V \setminus A_j \cap \{k \in V: q_k >j\}=\{\iota+1, \dotsc,v\}$. 
Note that $V \setminus A_j =\{\iota+1, \dotsc, v\} \cup \big\{k \in \{1,\dotsc,\iota-1\}: q_k < j \big\}$. 
Therefore, we can write $V \setminus A_j \cap \{k \in V: q_k >j\}= \big\{k \in \{\iota + 1,\dotsc,v\}: q_k > j \big\}$.
Consider $k \in \{\iota+1, \dotsc, v\}$. It follows from the definition of $q_k$ that $q_k \ge r + s_k + 1 > r+s_\iota+1=j$. Thus, $k \in V \setminus A_j \cap \{k \in V: q_k >j\}$, proving our claim.
%Second, consider $k \in \{1,\dotsc,\iota-1\}$ such that $ q_k < j$. Thus, by the definition of $\iota$, we have $q_k <j$; $k \notin V \setminus A_j \cap \{k \in V: q_k >j\}$. 
Moreover, the inequality in \eqref{Vc} is due to the fact that the coefficient of $\beta_j$ in \eqref{A_3} is $j-1-p+\big|\{ k \in V \setminus A_j: q_k>j\}\big|=j-1-p+ v - \iota = 0 $, where the first equality follows from the earlier argument and the second equality holds because of the definitions $j=r+s_\iota+1$ and $s_\iota=p-r-v+\iota$.
The equality in \eqref{Vc} follows from the assumption $j=r+s_\iota+1$ and the fact that $q_k \neq j$ for all $k \in V$ by assumption of this case.
Therefore, we can pick $\beta_j = \phi_{q_{\iota}}$, which satisfies \eqref{Va} because of \eqref{alpha_2} as $\phi_{q_{\iota}} \leq \phi_{q_{\iota+1}}$.
Furthermore, \eqref{Vc} is directly implied by \eqref{alpha_2}.

\item[(VI)] Assume that $j \in M \setminus  (P \cup Q)$ and $p+1 < j < \max\big\{p+1, \min\{q_{\iota}: \iota \in V\}\big\}$. Clearly, for this case to occur, we must have $\max\big\{p+1, \min\{q_{\iota}: \iota \in V\}\big\} = \min\{q_{\iota}: \iota \in V\} \neq p+1$.
%must  be strictly larger than $p+1$ and equal to $q_{\iota}$ for some $\iota \in V$.
As a result, we have $j < q_{\iota}$ for all $\iota \in V$.
Choose $A_{j} = V$.
Therefore, constraints \eqref{A_1} and \eqref{A_3} can be written as 
\begin{subequations}
\begin{align}
    &  \beta_{j} \ge \max\{\phi_{q_\iota}: \iota \in V \} = \phi_{q_{v}}, \label{VIa} \\
    & \beta_{j}  \ge   (h_{r+s_1} - h_{j} -\sum_{k \in L} \delta_{t_k} - \sum_{k\in V} \phi_{q_{k}})/( j -p -1), \label{VIb}  
\end{align}  
\end{subequations}
where the inequality in \eqref{VIa} follows from the fact that $A_j \cap \{k \in V: q_k >j \}= V$ because $q_\iota >j$ for all $\iota \in V$, and the equality in this relation is due to the definition of $\phi_{q_{\iota}}$ for $\iota \in V$ in \eqref{alpha}. Moreover, the inequality in \eqref{VIb} holds because the coefficient of $\beta_j$ in \eqref{A_3} is $j -p -1$, which is positive by the assumption of this case. 
Since both \eqref{VIa} and \eqref{VIb} impose a lower bound on $\beta_j$, we can simply pick
\begin{equation*}
\beta_j = \max\Big\{ \phi_{q_{v}}, (h_{r+s_1} - h_{j} -\sum_{k \in L} \delta_{t_k}- \sum_{k\in V} \phi_{q_{k}})/( j -p -1) \Big\}
\end{equation*}
to satisfy these conditions.

%\item[(VII)] Assume that $j \in Q$, $v=1$, and $q_1 \ge r+2=p+1$.
%In this case, $j = q_1$. Choose $A_j = \emptyset$. Hence, constraint \eqref{A_1} does not exist, and constraint \eqref{A_3} can be written as $\beta_j (j-p-1) \ge h_{r+1}-h_{r+2} -\phi_{q_1} -  \sum_{k \in L} \delta_{t_k} $. If $j=p+1$, then the coefficient of $\beta_j$ is zero, and this inequality is readily satisfied because of \eqref{alpha_1}. Otherwise, if $j>p+1$, we can select $\beta_j=\max\{0, (h_{r+1}-h_{r+2} -\phi_{q_1} -  \sum_{k \in L} \delta_{t_k})/(j-p-1)\}$ to satisfy the inequality. 

\item[(VII)] Assume that $j \in Q$ and $r+s_1 +1 \le j \leq p+1$. On the one hand, since $j \in Q$, we can write $j=q_i$ for some $1\le i \le v$. 
On the other hand, since $r+s_1 +1 \le j \leq p+1$, we can write that $j=r+s_\iota +1 $ for some $1 \le \iota \le v$.
We claim that $i \le \iota$.
Assume by contradiction that $i > \iota$.
Then, we must have $j = q_i \ge r+s_i +1 > r+s_{\iota} +1 = j$, where the equalities follow from the previous arguments for this case, the first inequality is implied by the theorem requirement for $q_i$, and the second inequality is due to the fact that $s_i > \iota$ as $i > \iota$ by definition, yielding a contradiction.
Next, we consider two sub-cases.
For the first sub-case, assume that $i<\iota$.
Choose $A_{j}=\big\{ k \in \{1,\dotsc,\iota\} : q_{k} \geq j \big\}$. It is clear that $\iota \in A_j$ as $q_{\iota} \ge r+s_\iota +1 =j=q_i$, where the equality across the chain is impossible since it would imply $q_\iota=q_i$. Thus, $\iota \in A_j$ and $q_\iota>j$. As a result, constraints 
\eqref{A_1} and \eqref{A_3} can be written as 
\begin{subequations}
\begin{align}
    &  \phi_{q_\iota}=\max\{\phi_{q_{k}}: k \in A_j, \; q_k >j \} \le \beta_{j} \le \min\{\phi_{q_{k}}: k \in V\setminus A_j, \; q_k >j\}=\phi_{q_{\iota+1}}, \label{VIIIa_1}\\
    & \phi_{q_{\iota}} \ge h_{r+s_1} - h_{j}  -\sum_{k \in L} \delta_{t_k}-  \phi_{j} - \sum_{\substack{k=1\\q_k > j}}^{\iota-1} \phi_{q_{k}} = h_{r+s_1} - h_{r + s_\iota +1}  -\sum_{k \in L} \delta_{t_k}- \sum_{\substack{k=1\\q_k \ge r + s_\iota +1}}^{\iota-1} \phi_{q_{k}}. \label{VIIIc_1}
\end{align}   
\end{subequations}
The left equality in \eqref{VIIIa_1} follows from the definition of $\phi_{q_{k}}$ for $k \in V$ in \eqref{alpha}, and the facts that $\iota \in A_j$ and $q_{\iota} > j$.
For the right equality in \eqref{VIIIa_1}, we use the fact that $V \setminus A_j \cap \{k \in V: q_k >j\}=\{\iota+1, \dotsc,v\}$, which follows from an argument similar to that of case (V),  discussed above. 
%We note that $V \setminus A_j \subseteq \{\iota, \dotsc, p-r\} \cup \{k \in V: q_k < r + \iota +1\}$. It is clear by the definition of $\iota$ that for $k=\iota+1, \dotsc, p-r$, we have $q_k \ge r + k + 1 > r+\iota+1=q_i=j$; thus, $k \in V \setminus A_j \cap \{k \in V: q_k >j\}$. On the other hand, for $k=1, \dotsc, \iota-1$, suppose that we have $q_k < r + \iota + 1$. Thus, by the definition of $\iota$, we have $q_k <j$; $k \notin V \setminus A_j \cap \{k \in V: q_k >j\}$. 
Moreover, the inequality in \eqref{VIIIc_1} is due to the fact that the coefficient of $\beta_j$ in \eqref{A_3} is zero, which can be shown using a similar argument to that of case (V). %$j-1-p+|\{ k \in V: k \in V \setminus A_j, \; q_k>j\}|=j-1-p+ p-r - \iota  =0$.
The equality in this relation follows from the facts that $j = r+s_{\iota} + 1$, and $\big\{k \in \{1,\dotsc,\iota-1\}: q_k > j\big\} \cup \{j\} = \big\{k \in \{1,\dotsc,\iota-1\}: q_k \geq j\big\}$ since $j = q_i$ for some $i \in \{1,\dotsc,\iota - 1\}$ by the assumption of this case.
Therefore, we can pick $\beta_j = \phi_{q_{\iota}}$, which satisfies \eqref{VIIIa_1} because of \eqref{alpha_2} as $\phi_{q_{\iota}} \leq \phi_{q_{\iota+1}}$.
Furthermore, \eqref{VIIIc_1} is implied by \eqref{alpha_2}.

For the second sub-case, assume that $i = \iota$.
Choose 
$A_{j}=\big\{ k \in \{1,\dotsc,\iota\} : q_{k} \ge j \big\}$. It is clear that $\iota \in A_j$ as $q_{\iota}=j$. 
Hence, constraints \eqref{A_1} and \eqref{A_3} can be written as 
\begin{subequations}
\begin{align}
    &  \phi^*_1=\max\{\phi_{q_{k}}: k \in A_j, \; q_k >j \} \le \beta_{j} \le \min\{\phi_{q_{k}}: k \in V\setminus A_j, \; q_k >j\}=\phi_{q_{\iota+1}}, \label{VIIIa_2}\\
    &  \phi_{q_{\iota}} \ge h_{r+s_1} - h_{j}  -\sum_{k \in L} \delta_{t_k} - \sum_{\substack{k=1\\q_k > j}}^{\iota} \phi_{q_{k}} =  h_{r+s_1} - h_{r + s_\iota +1} -\sum_{k \in L} \delta_{t_k} -\sum_{\substack{k=1\\q_k \ge r + s_\iota +1}}^{\iota-1} \phi_{q_{k}}, \label{VIIIc_2}
\end{align}   
\end{subequations}
where $0\le \phi^*_1 \le \phi_{q_{\iota-1}}$ since $\phi_{q_k} \leq \phi_{q_{\iota-1}}$ for all $k \in \{1,\dotsc,\iota - 1\}$. 
The right equality in \eqref{VIIIa_2} follows from an argument similar to that of the first sub-case above.
Moreover, the inequality in \eqref{VIIIc_2} is due to the fact that the coefficient of $\beta_j$ in \eqref{A_3} is zero, which can be computed similarly to case (V), discussed previously.
The equality in \eqref{VIIIc_2} follows from the facts that $j = r + s_\iota +1$ and $\big\{k \in \{1,\dotsc,\iota\}: q_k > j\big\} = \big\{k \in \{1,\dotsc,\iota-1\}: q_k \geq j\big\}$ since $j = q_i = q_{\iota}$ by the assumption of this case.
Therefore, we can pick $\beta_j = \phi_{q_{\iota+1}}$, which satisfies \eqref{VIIIa_2} because of \eqref{alpha_2} as $\phi^*_1 \le \phi_{q_{\iota-1}} \leq \phi_{q_{\iota+1}}$.
Furthermore, \eqref{VIIIc_2} is directly implied by \eqref{alpha_2}.

\item[(VIII)] Assume that $j \in Q$ and $j>p+1$.
We can write $j = q_{\iota}$ for some $\iota = 1, \dotsc, v$. We choose $A_{j}=\{ k \in V : q_{k} > j \}$. 
Hence, constraints \eqref{A_1} and \eqref{A_3} for $j$ can be written as 
\begin{align*}
     &\beta_{j} \ge \max\{\phi_{q_{k}}: k \in A_j\}:=\phi^*_2,\\
     & \beta_{j} \ge (h_{r+s_1}- h_{j} -\sum_{k \in L} \delta_{t_k} - \phi_{q_\iota} -  \sum\nolimits_{k \in A_j} \phi_{q_k})/(j-p-1),
\end{align*}
where $\phi^*_2 \ge \phi_{q_{1}}\ge0$
and $j-p-1 > 0$ by assumption.
Therefore, we can select 
$\beta_{j} = \max \big\{ \phi^*_2, (h_{r+s_1}- h_{j} -\sum_{k \in L} \delta_{t_k} - \phi_{q_\iota} -  \sum\nolimits_{k \in A_j} \phi_{q_k})/(j-p-1) \big\}$ to satisfy the above inequalities. 

\item[(IX)] Assume that $j \in P$. We choose $A_j=\emptyset$. In this case, constraints \eqref{A_1} and \eqref{A_3} reduce to 
\begin{align*}
    & \beta_j \le \min\{\phi_{q_\iota}: \iota \in V \}, \\
    & \beta_{j} (j-1 - p + v) \ge  (h_{t_{a_j+1}} - h_j - \sum_{\substack{\iota \in L\\ t_\iota < j}} \delta_{t_\iota}).  
\end{align*} 
We can write that $j-1-p+v = j - 1 - r - (p-r-v) \le j-1-r<0$ where the first inequality follows from $p-r-v \ge 0$ and the second inequality holds by assumption. 
Using this result together with the fact that $h_j=h_{t_{a_j+1}}$ as $j \in P$, we can simplify the above inequalities as 
\begin{align*}
    & \beta_j \le \min\{\phi_{q_\iota}: \iota \in V \}, \\
    & \beta_{j} \le \sum_{\substack{\iota \in L\\ t_\iota < j}} \delta_{t_\iota}/(p-v + 1 -j).  
\end{align*} 
Because both right-hand-side values of these inequalities are non-negative, we can simply pick $\beta_j = 0$ to satisfy the conditions in Theorem~\ref{thm:chance_lift_generic}. 
This completes the proof. \hfill \Halmos

\end{itemize}
\endproof

Next, we show that the family of \BLP inequalities in Theorem~\ref{thm: chance_liftVIAgeneric_permutation} encompasses the existing results in the literature, represented by Proposition~\ref{prop:chance_lift_permutation}, as a special case under the uniform probability distribution.

\begin{corollary} \label{cor:special_case}
    Consider the setting of Theorem~\ref{thm: chance_liftVIAgeneric_permutation} for the mixing set with a cardinality constraint $\mac{F}_c^=$.
    By setting $\delta_{t_{\iota}} = 0$ for all $\iota \in L$, the \BLP inequality \eqref{eq:chance_lift_ineq_generic_permutation} reduces to the inequality \eqref{eq:chance_lift_permutation} in Proposition~\ref{prop:chance_lift_permutation}.
\end{corollary}

\proof{Proof}
We have previously established in Section~\ref{subsec:MIP} that, for $\mac{F}_c^=$, it holds that $p = \tau = \lfloor m\varepsilon \rfloor$.
Consequently, the requirement \eqref{s} in Proposition~\ref{prop:chance_lift_permutation} implies $s_\iota = p - r - v + \iota$ for all $\iota \in \{1, \dotsc, v\}$.
Thus, we obtain $r + \min\{1 + s_\iota, s_{\iota+1}\} = r + s_\iota + 1$.
This ensures that the parameter definitions and assumptions in Theorem~\ref{thm: chance_liftVIAgeneric_permutation} and Proposition~\ref{prop:chance_lift_permutation} coincide when $\delta_{t_{\iota}} = 0$ for all $\iota \in L$.
As a result, the definition of $\phi_{q_\iota}$ is consistent across Theorem~\ref{thm: chance_liftVIAgeneric_permutation} and Proposition~\ref{prop:chance_lift_permutation}, leading to the equivalence between the \BLP inequality \eqref{eq:chance_lift_ineq_generic_permutation} and the inequality \eqref{eq:chance_lift_permutation}.
\hfill \Halmos
\endproof

The following example illustrates the result of Corollary~\ref{cor:special_case} by presenting a facet-defining inequality for $\conv(\mac{F}_c^=)$ that can be derived from Theorem~\ref{thm: chance_liftVIAgeneric_permutation} but not from Proposition~\ref{prop:chance_lift_permutation}, thereby demonstrating that the family of \BLP inequalities \eqref{eq:chance_lift_ineq_generic_permutation} \textit{strictly} subsumes the existing results.

\begin{example}
\label{ex:9}
Consider an instance of $\mac{F}_c^=$, with $m=10$, $\{h_1, \dotsc, h_m\}=\{20,18,14,11,6,5,4,3,2,1\}$, and $p=4$.
We apply Theorem~\ref{thm: chance_liftVIAgeneric_permutation} by selecting $r=1$, $l = 1$, $v=3$, $P=\{1\}$, and $Q=\{6,8,7\}$, i.e., $q_1=6$, $q_2=8$, and $q_3=7$.
Moreover, we choose $\delta_1=1$. 
Following the equations \eqref{alpha_1} and \eqref{alpha_2}, we can calculate $\phi_6=3$, $\phi_7=5$, and $\phi_8=3$. 
As a result, \eqref{eq:chance_lift_ineq_generic_permutation} yields the \BLP inequality
\begin{equation*}
    z+ 3 x_1 + 3(1-x_6)+ 5(1-x_{7}) + 3(1-x_{8})\ge 20,
\end{equation*}
which can be shown to be facet-defining for $\conv(\mac{F}_c^=)$.
However, this inequality cannot be derived from Proposition~\ref{prop:chance_lift_permutation}, partly because the coefficient of $x_1$, which is $3$, does not match $(h_1 - h_2) = 2$ as prescribed therein.
\hfill	$\blacksquare$
\end{example}

\smallskip

%As demonstrated in Corollary~\ref{cor:special_case}, when the values of $\delta_{t_\iota}$ are fixed to zero, the structure of the \BLP inequality \eqref{eq:chance_lift_ineq_generic_permutation} coincides with that of Proposition~\ref{prop:chance_lift_permutation}, which in turn aligns with Proposition~\ref{prop:chance_lift_simge}; see the proof of Corollary~\ref{cor:special_case}. Consequently, for any fixed values of $\delta_{t_\iota}$ with $\iota \in L$, we can employ an argument similar to that in Section 3.1 of \cite{kuccukyavuz2012mixing} to develop an efficient separation algorithm that identifies the parameters in Theorem~\ref{thm: chance_liftVIAgeneric_permutation} yielding the most violated inequality at a given point.

\smallskip
Although Theorem~\ref{thm: chance_liftVIAgeneric_permutation} generalizes existing results in the literature for mixing sets with cardinality constraints, it still constitutes a subfamily of the \BLP inequalities described in Theorem~\ref{thm:chance_lift_generic}. This is because Theorem~\ref{thm: chance_liftVIAgeneric_permutation} imposes several restrictions on the choice of parameters in Theorem~\ref{thm:chance_lift_generic}, which, in turn, enables the derivation of explicit valid inequalities without the need to verify the existence of a vector $\vc{\beta}$ satisfying the conditions in \eqref{A}.
Some of these restrictions are as follows:
(i) In Theorem~\ref{thm: chance_liftVIAgeneric_permutation}, we have $v = |Q| \le p - r$, whereas in Theorem~\ref{thm:chance_lift_generic}, $v = |Q| \le p - r + |P|$;
(ii) In Theorem~\ref{thm: chance_liftVIAgeneric_permutation}, the range of $Q$ is restricted to $\{r + s_1 + 1, \dotsc, m\}$, with $s_1 = p - r - v + 1 \ge 1$, whereas in Theorem~\ref{thm:chance_lift_generic}, the range of $Q$ may be within $\{r + 1, \dotsc, m\}$; and
(iii) Theorem~\ref{thm: chance_liftVIAgeneric_permutation} imposes additional restrictions on $\delta_{t_{\iota}}$ compared to Theorem~\ref{thm:chance_lift_generic}.

To illustrate, the following example presents a facet-defining inequality that can be derived from Theorem~\ref{thm:chance_lift_generic} but not from Theorem~\ref{thm: chance_liftVIAgeneric_permutation}, indicating that the former theorem yields a richer family of valid inequalities---even under the special case of uniform probability distributions---as further demonstrated by the computational experiments in Section~\ref{sec:numerical}.

\begin{example}
\label{ex:10}
Consider the instance of $\mac{F}_c^=$, given in Example \ref{ex:1}, with $m=10$, $\{h_1, \dotsc, h_m\}=\{20,18,14,11,6,5,4,3,2,1\}$, and $p=4$. We established in Example~\ref{ex:3} that inequality~\eqref{eq:Ex1} belongs to the family of \BLP inequalities~\eqref{eq:chance_lift_ineq_generic} derived in Theorem~\ref{thm:chance_lift_generic}. In contrast, we show here that this inequality cannot be represented by the family of valid inequalities in Theorem~\ref{thm: chance_liftVIAgeneric_permutation}. 
Based on the structure of inequality~\eqref{eq:Ex1}, for this problem instance, we must have $r=4$, $P=\{1,4\}$, and $Q=\{5,6\}$. However, since $p=r$, the conditions in Theorem~\ref{thm: chance_liftVIAgeneric_permutation} enforce $Q=\emptyset$, which contradicts with the requirement $Q=\{5,6\}$ for this inequality. Moreover, we must choose $\delta_1=\delta_4=-3$ to ensure that the coefficient of variables in $P$, namely $x_1$ and $x_4$, are $(h_1-h_4+\delta_1)=6$ and $(h_4-h_5+\delta_4)={\blue{2}}$, respectively. While $\delta_1\ge h_4-h_1=-9$ and $\delta_4 \ge h_5-h_4={\blue{-5}}$, with $\delta_1 + \delta_2 \le h_5  - h_6 = 1$ as $s_1=1$ and $s_2=2$, the requirement that $\delta_1 \ge 0$ is not satisfied. 
As a result, inequality~\eqref{eq:Ex1} cannot be represented by the family of inequalities in Theorem~\ref{thm: chance_liftVIAgeneric_permutation}. 
\hfill	$\blacksquare$
\end{example}

\smallskip
Figure~\ref{fig: uniform_prob} summarizes the relationship between the families of inequalities developed in the literature (Propositions~\ref{prop:mixing}--\ref{prop:chance_lift_permutation}) and the \BLP inequalities established in this work (Theorems~\ref{thm:chance_lift_generic}--\ref{thm: chance_liftVIAgeneric_permutation}) in characterizing $\conv(\mac{F}_c^=)$ under a uniform probability distribution.

\begin{figure}[htbp]
    \centering
    \begin{tikzpicture}[line width=0.9pt, font=\small]

    % --- canvas scale ---
    \def\W{8.2}   % rectangle width
    \def\H{3.6}    % rectangle height

    % --- frame rectangle ---
    \draw (-\W/2-0.7,-\H/2) rectangle (\W/2,\H/2);
    \node[above right] at (-\W/2-0.7,-\H/2) {$\conv(\mac{F}_c^=)$};

    % --- parameters ---
    \def\cy{0}             % vertical center
    \def\anglestart{10}    % open gap start (degrees)
    \def\angleend{350}     % open gap end

    % --- nested, right-shifted ellipses (each contains the previous) ---
    % Format: center-x / x-radius / y-radius / label
    \foreach \cx/\a/\b/\c/\i in {%
      -3.3/0.7/0.4/Prop~\ref{prop:mixing}/1,   % smaller innermost ellipse
      -2.8/1.3/0.6/Prop~\ref{prop:chance}/2,
      -2.3/1.9/0.8/Prop~\ref{prop:chance_lift}/3,
      -1.8/2.5/1/Prop~\ref{prop:chance_lift_simge}/4,
      -1.3/3.1/1.2/Prop~\ref{prop:chance_lift_permutation}/5,
      -0.8/3.7/1.4/Thm~\ref{thm: chance_liftVIAgeneric_permutation}/6, -0.3/4.3/1.6/Thm~\ref{thm:chance_lift_generic}/7
    }{
        \draw (\cx,\cy) ellipse [x radius=\a, y radius=\b, start angle=\anglestart, end angle=\angleend];
        \node[anchor=east] at (\cx+\a, \cy) {\c};
    }

    \end{tikzpicture}
    \caption{Relationship between existing inequalities (Propositions~\ref{prop:mixing}--\ref{prop:chance_lift_permutation}) and \BLP inequalities (Theorems~\ref{thm:chance_lift_generic}--\ref{thm: chance_liftVIAgeneric_permutation}) in characterizing $\conv(\mac{F}_c^=)$ under a uniform probability distribution.}
    \label{fig: uniform_prob}
\end{figure}
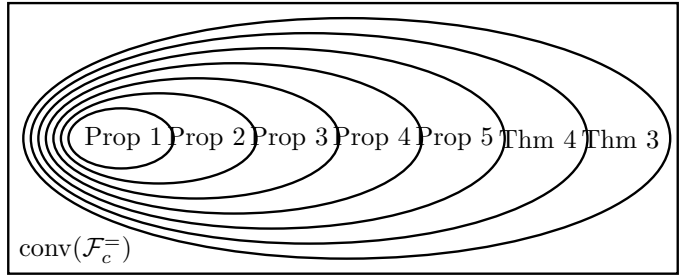

\blue{
The \BLP family \eqref{eq:chance_lift_ineq_generic_permutation} contains exponentially many inequalities, rendering their full implementation computationally demanding. As is common in the literature, efficient separation algorithms are often developed for particular subclasses of such families; see, for example, Section~3.1 of \cite{kuccukyavuz2012mixing}. Following a similar approach, we next present an efficient exact separation algorithm for an important subclass of the \BLP family \eqref{eq:chance_lift_ineq_generic_permutation} that contains exponentially many inequalities, as described next.
}

\blue{
\begin{definition} \label{def:subclass}
    Consider the setting of Theorem~\ref{thm: chance_liftVIAgeneric_permutation} for set $\mac{F}_c^=$. We say that a \BLP inequality of the form \eqref{eq:chance_lift_ineq_generic_permutation} is $Q$-\textit{symmetric} if set $Q$ used in its derivation can be written as $Q = Q_1 \cup Q_2$ where $Q_1 = \{p-v+2, \dotsc, p-v+1+\tau\}$ for some $\tau \in \{0, \dotsc, \max\{0,v-1\}\}$ and $Q_2 \subseteq \{p+1, \dotsc, m\}$. We refer to the collection of all $Q$-symmetric inequalities as the $Q$-\textit{symmetric subclass} of the \BLP family.
\end{definition}
}

\smallskip

%An important special case of valid inequalities for $\conv(\mac{F}_c^=)$ has received considerable attention in the literature because of its simplicity and the prevalence of this structure among the facet-defining inequalities of the set; see, for example, \cite[Corollary~2]{zhao2017polyhedral}. Specifically, these are the inequalities with a single variable having a negative coefficient. The following result shows that, in this special case, all \BLP inequalities are $Q$-symmetric.

\blue{The following result shows that every \BLP inequality with at most one variable having a negative coefficient is $Q$-symmetric.}

\blue{
\begin{corollary}\label{cor:Q-symmetric}
    Consider the setting of Theorem~\ref{thm: chance_liftVIAgeneric_permutation} for set $\mac{F}_c^=$. Then every \BLP inequality of the form \eqref{eq:chance_lift_ineq_generic_permutation} satisfying $|Q| \leq 1$ is $Q$-symmetric.
\end{corollary}
}

\proof{Proof}
\blue{
First, consider the case with $|Q|=0$, which implies that $v = 0$. Therefore, $Q$ can be written as $Q = Q_1 \cup Q_2$, where $Q_1 = \emptyset$ and $Q_2 = \emptyset \subseteq \{p+1, \dotsc, m\}$, satisfying the $Q$-symmetric conditions in Definition~\ref{def:subclass} for $\tau = 0$. 
Second, consider the case with $|Q| = 1$, which implies that $v = 1$. As a result, the requirement for $q_1$ in Theorem~\ref{thm: chance_liftVIAgeneric_permutation} reduces to $q_1 \geq r+s_1+1 = p-v+1 + 1 = p+1$. Therefore, $Q$ can be written as $Q = Q_1 \cup Q_2$, where $Q_1 = \emptyset$ and $Q_2 = \{q_1\} \subseteq \{p+1, \dotsc, m\}$, satisfying the $Q$-symmetric conditions in Definition~\ref{def:subclass} for $\tau = 0$. 
}    
\hfill \Halmos
\endproof

\smallskip
\blue{
While the above result shows that the entire \BLP family is captured by the $Q$-symmetric structure when $|Q| \leq 1$, this structure continues to represent a significant portion of the \BLP family in the general case with $|Q|>1$. 
For example, it is straightforward to verify that the inequality derived in Example~\ref{ex:9} is $Q$-symmetric. 
%More importantly, the $Q$-symmetric subclass captures a substantial portion of the \BLP family in practice. 
Moreover, the computational experiments presented in Section~\ref{sec:numerical} confirm that a substantial fraction of the \BLP inequalities appearing in the convex hull descriptions of the sets studied in that section belong to the $Q$-symmetric subclass.
}

%\blue{It is straightforward to verify that the inequality derived in Example~\ref{ex:9} is $Q$-symmetric. More importantly, the $Q$-symmetric subclass captures a substantial portion of the \BLP family in practice. The computational experiments presented in Section~\ref{sec:numerical} confirm that a significant fraction of the \BLP inequalities appearing in the convex hull descriptions of the sets studied in that section belong to the $Q$-symmetric subclass.}

\blue{
When the $\delta$ values in the derivation of \eqref{eq:chance_lift_ineq_generic_permutation} are fixed, the $Q$-symmetric subclass contains exponentially many inequalities due to the possible choices of the subsets $P$ and $Q_2$.
Nevertheless, the exact separation problem over this subclass can be solved efficiently, as presented in Algorithm~\ref{alg:separation_1}.
}

\blue{
Intuitively, the algorithm exploits a structural property of the $Q$-symmetric inequalities: the elements of the set $Q_2$ can be determined independently of their associated coefficients. This property makes it possible to identify the corresponding ordering by appropriately sorting the components of the solution to be separated.
This property also induces a symmetry among a subgroup of valid inequalities with the same parameters, which motivates the name of this subclass.
The algorithm then determines the elements of the set $P$ by solving a shortest-path problem on a network whose nodes and arc weights are constructed from the contributions of candidate elements and their interactions with subsequent elements in the sequence. Together, these techniques reduce the exponentially large search space of possible choices to a polynomial-time search for the most violated inequality.
Proposition~\ref{prop:separation} establishes the exactness of the proposed separation algorithm and shows that it runs in polynomial time.
}

\blue{In both Algorithm~\ref{alg:separation_1} and proof of Proposition~\ref{prop:separation} we adopt the previously established convention that whenever a set is empty, any conditions or expressions involving that set are understood to be omitted. 
}

\begin{algorithm}[!ht]
	\caption{Exact separation algorithm for the $Q$-symmetric subclass of \BLP inequalities}
	\label{alg:separation_1}
	\KwIn{An instance of $\mac{F}_c^=$; values of $\delta_j$ for $j \in M$; a solution $(\vc{x}^*,z^*) \in [0,1]^m \times \Re_+$.}
	\KwOut{Either the most violated $Q$-symmetric \BLP inequality of the form \eqref{eq:chance_lift_ineq_generic_permutation} or a statement that no such violated inequality exists.}

    Initialize the best violation value $\Delta^* \leftarrow 0$ and the parameter set $\left(r^*,l^*,v^*,P^*,Q^*\right) \leftarrow \emptyset$;\\
    Sort the values $\left\{x_j^*:j=p+1,\dotsc,m\right\}$ in nonincreasing order to obtain $\left\{x_{j_1}^*,\dotsc,x_{j_{m-p}}^*\right\}$;\\
	\For{$r = 1, \dotsc, p$}
	{
        Construct a directed acyclic network $G(N,A)$, where $N=\{0,1,\dotsc,r,r+1\}$ and $A$ contains an arc $(i,j)$ whenever $i<j$ and $(i,j)\neq(0,r+1)$;\\
        Assign weight $-h_j$ to each arc $(i,j)\in A$ with $i=0$;\\
        
		\For{$v=0, \dotsc, p-r$}
		{
            \For{$\tau=0, \dotsc, \max\{0,v-1\}$}
		      {              
                Set $Q_1 = \emptyset$ if $\tau = 0$, and otherwise set $Q_1 = \{q_1, \dotsc, q_{\tau}\}$, where $q_{\iota} = p-v+1+\iota$ for $\iota = 1, \dotsc, \tau$;\\
                Set $Q_2 = \emptyset$ if $v = 0$, and otherwise set $Q_2 = \{q_{\tau + 1}, \dotsc, q_v\}$, where $q_{\iota} = j_{v-\iota + 1}$ for $\iota = \tau + 1, \dotsc, v$;\\  
                Set $Q = Q_1 \cup Q_2$;\\
                Assign weight $(h_i - h_j + \delta_i)x^*_i - \delta_i\sum_{q \in Q}(1-x^*_{q})$ to each arc $(i,j) \in A$ with $i \neq 0$ and $j \neq r+1$;\\
                Assign weight $(h_i - h_{p-v+1} + \delta_i)x^*_i - \delta_i\sum_{q \in Q}(1-x^*_{q})$ to each arc $(i,j) \in A$ with $i \neq 0$ and $j = r+1$;\\

                Find a shortest path from node $0$ to node $r+1$ in $G$, and record the labels of its internal nodes as $P=\{t_1,\dotsc,t_l\}$ for some $l\in\{1,\dotsc,r\}$;\\
                Compute the \BLP inequality \eqref{eq:chance_lift_ineq_generic_permutation} corresponding to the parameter set $(r,l,v,P,Q)$;\\
                Evaluate its violation at $(\vc{x}^*,z^*)$ as $\Delta=\text{left-hand side}-\text{right-hand side}$;\\
                
                \If{$\Delta < \Delta^*$}
                {
                    $\Delta^* \leftarrow \Delta$;\\
                    $(r^*, l^*, v^*, P^*, Q^*) \leftarrow (r, l, v, P, Q)$;\\                 
                }
            }
		}
	}
    \If{$\Delta^* = 0$}{
        \Return No violated $Q$-symmetric \BLP inequality exists;
    }
    \Else
    {
        \Return The \BLP inequality \eqref{eq:chance_lift_ineq_generic_permutation} corresponding to the parameter set $(r^*,l^*,v^*,P^*,Q^*)$, which is the most violated $Q$-symmetric inequality separating $(\vc{x}^*,z^*)$ from $\conv(\mac{F}_c^=)$; 
    }
\end{algorithm}

\blue{
\begin{proposition} \label{prop:separation}
    Consider $\mac{F}_c^=$ and assume that the values of $\delta_j$ are fixed for all $j \in M$. For any given solution $(\vc{x}^*,z^*) \in [0,1]^m \times \Re_+$, Algorithm~\ref{alg:separation_1} identifies the most violated inequality in the $Q$-symmetric subclass of the \BLP family \eqref{eq:chance_lift_ineq_generic_permutation} in $\mathcal{O}(m\log m + p^5)$ time.
\end{proposition}
}

\proof{Proof}
\blue{
First, we show that Algorithm~\ref{alg:separation_1} is exact and identifies the most violated inequality in the $Q$-symmetric subclass. We prove this result for fixed values of $r$ and $v$, since the algorithm iterates over all choices of these parameters to determine the overall most violated inequality.
}

\blue{
By definition, every $Q$-symmetric inequality satisfies $Q=Q_1\cup Q_2$, where $Q_1$, if nonempty, consists of consecutive elements beginning at $p-v+2$ and $Q_2 \subseteq \{p+1,\dotsc,m\}$. Observe that the elements of $Q_1$ can extend only up to $p$, since any larger element can instead be included in $Q_2$. Consequently, the size of $Q_1$ ranges from $0$ to $v-1$, and this size is represented by the parameter $\tau$ in line~7 of Algorithm~\ref{alg:separation_1}. When $\tau=0$, we necessarily have $Q_1=\emptyset$, as specified in line~8. Otherwise, $Q_1= \{q_1,\dotsc,q_\tau\}$, where the elements are consecutive beginning at $p-v+2$. Since the condition $q_\iota \ge r+s_\iota+1 = p-v+1+\iota$ must hold for each $\iota=1,\dotsc,\tau$, it follows that $q_\iota=p-v+1+\iota$ for every $\iota=1,\dotsc,\tau$. This choice is enforced in line~8 of the algorithm.
}

\blue{
Now consider the set $Q_2$. When $v = 0$, Theorem~\ref{thm: chance_liftVIAgeneric_permutation} implies that $Q = \emptyset$. Thus, we must have $Q_1 = Q_2 = \emptyset$, which is enforced in lines~8 and 9 of the algorithm. Otherwise, we have $Q_2=\{q_{\tau+1},\dotsc,q_v\}$, whose elements are chosen from $\{p+1,\dotsc,m\}$. Since $q_\iota \geq p+1 \geq p-v+1+\iota = r+s_\iota+1$ for each $\iota=\tau+1,\dotsc,v$, every choice of $q_\iota$ satisfies the required ordering condition. Furthermore, when computing the coefficients $\phi_{q_\iota}$ for $\iota=\tau+1,\dotsc,v$, the condition appearing in the final summation of \eqref{alpha_2} is automatically satisfied, reducing the summation to $\sum_{k=\tau+1}^{\iota-1}\phi_{q_k}$. Consequently, the values of these coefficients depend only on the index $\iota$ and not on the particular choice of $q_\iota$. Therefore, the elements of $Q_2$ can be selected solely according to the values of $(1-x^{q\iota})$ appearing in \eqref{eq:chance_lift_ineq_generic_permutation}. Since the coefficients $\phi_{q_\iota}$ are nondecreasing by \eqref{alpha_1} and \eqref{alpha_2}, the \textit{rearrangement inequality} implies that $\sum_{\iota=\tau+1}^{v}\phi_{q_\iota}(1-x^*_{q\iota})$ is minimized by assigning the largest coefficients to the smallest values of $(1-x^*_j)$. Accordingly, the algorithm first sorts the values $x^*_{j}$ for $j = p+1, \dotsc, m$ in decreasing order with the sorted index set $\{j_1, \dotsc, j_{m-p}\}$ and then assigns $q_\iota=j_{v-\iota+1}$ for $\iota=\tau+1,\dotsc,v$. These operations are performed in lines~2 and~9 of the algorithm.
}

\blue{
It remains to determine the set $P=\{t_1,\dotsc,t_l\} \subseteq \{1,\dotsc,r\}$. The contribution of the selected elements can be interpreted as follows. Suppose that $t_\iota=i$ and $t_{\iota+1}=j$ for some $\iota=1,\dotsc,l-1$. Then the corresponding contribution to the left-hand side of \eqref{eq:chance_lift_ineq_generic_permutation} is $(h_i-h_j+\delta_i)x_i^*-\delta_i\sum_{q \in Q}(1-x^*_{q})$, where the final term results from the computation of the coefficients in \eqref{alpha_1} and \eqref{alpha_2}. When $\iota=1$, an additional term $-h_i$ is contributed by the right-hand side of \eqref{eq:chance_lift_ineq_generic_permutation}. Likewise, when $\iota=l$, the expression is adjusted by replacing $j$ with $r+s_1=p-v+1$ by definition. Therefore, identifying the most violated inequality is equivalent to finding a sequence of pairs $(i,j)$ that minimizes the total contribution.}

\blue{
To accomplish this, the algorithm constructs a directed acyclic network whose nodes correspond to the indices $1,\dotsc,r$. Each arc $(i,j)$ is assigned the corresponding contribution as its cost. A dummy source node labeled $0$ is connected to every node $i$ with arc cost $-h_i$, thereby representing the selection of the first element $t_1$. Similarly, every node $i$ is connected to a dummy sink node labeled $r+1$ to account for the last element adjustment obtained by replacing $j$ with $p-v+1$. Consequently, determining the optimal set $P$ reduces to solving a shortest-path problem on this network. These operations are performed in lines~4, 5, and~11--13 of the algorithm.
}

\blue{
Hence, for the given values of $r$, $v$, and $\tau$, the computed sets $P$ and $Q$ define the most violated $Q$-symmetric \BLP inequality (lines~14 and~15). Iterating over all possible values of $r$, $v$, and $\tau$ therefore identifies the overall most violated $Q$-symmetric \BLP inequality for the fixed $\delta$ values (lines~16--18). If the resulting violation is negative, the corresponding inequality is returned; otherwise, the algorithm concludes that no violated inequality exists in the $Q$-symmetric subclass (lines~19--22).
}

\blue{
We now analyze the running time of the algorithm. The sorting operation in line~2 can be done in $\mac{O}(m\log m)$ time. Constructing the network $G(N,A)$ in line~4 requires $\mathcal{O}(p^2)$ time. The nested iterations over the parameters $r$, $v$, and $\tau$ in lines~3, 6, and~7 perform $\mathcal{O}(p^3)$ iterations. Within each iteration, the shortest-path problem in line~13 can be solved in $\mathcal{O}(p^2)$ time. Therefore, the overall running time of the algorithm is $\mathcal{O}(m\log m + p^5)$.
}    
\hfill \Halmos
\endproof

\smallskip
\blue{The separation results presented above generalize that of \cite{kuccukyavuz2012mixing}, as they apply to a broader family of inequalities due to the inclusion of the additional parameters $\delta$ and $v$.}

\blue{Although Algorithm~\ref{alg:separation_1} solves the separation problem in polynomial time, its running time can be further reduced in certain special cases. In particular, when the cardinality of set $Q$ is bounded by a constant---a common assumption in the design of separation algorithms (see, e.g., \cite[Corollary~2]{zhao2017polyhedral}, which considers the case $|Q|=1$)---the running time decreases to $\mathcal{O}(m\log m + p^3)$.}

\blue{Finally, we note that, in large-scale practical applications where $p \ll m$, the running time of Algorithm~\ref{alg:separation_1} is dominated by the $\mathcal{O}(m\log m)$ sorting step. 
%Despite being applicable only to a subclass of inequalities, the proposed approach yields a more efficient separation algorithm than the traditional cut-generating linear programming approach associated with disjunctive programming formulations, which requires solving a linear program to identify a violated inequality. Although this linear program can be solved in polynomial time, its worst-case computational complexity grows significantly faster than that of the proposed separation algorithm and can therefore become computationally demanding for large-scale instances.
}

%%%%%%%%%%%%%%%%%%%%%%%%%%%%%%%%%%%%%%%%%%%%%%%%%%%%%%%%%%%%%%%%%%
\section{Computational Experiments} \label{sec:numerical}

In this section, we present preliminary computational results to evaluate the effectiveness of the \BLP families of inequalities derived in Sections~\ref{subsec:BLP family-general} and \ref{subsec:BLP family-special} in characterizing the convex hull of the mixing set. In particular, we compare their performance against existing explicit families of valid inequalities from the literature to assess the extent to which the proposed inequalities strengthen the relaxation and improve the approximation of the convex hull.
We consider two sets of benchmark instances for the mixing set with a cardinality constraint that have been studied in the literature, as described below: 

\begin{example}
\label{ex: Leudtke}
Instances of $\mac{F}_c^=$ with $m \in \{1, \dotsc, 10\}$, where the values of $\{h_1, \dotsc, h_m\}$ correspond to the first $m$ elements of the sequence $\{20, 18, 14, 11, 6, 5, 4, 3, 2, 1\}$, and $p$ is selected from the range $\{1,2, \dotsc, m\}$, as presented in \cite{luedtke2010integer}.
\end{example}

\begin{example}
\label{ex: Simge}
Instances of $\mac{F}_c^=$ with $m \in \{1, \dotsc, 10\}$, where the values of $\{h_1, \dotsc, h_m\}$ correspond to the first $m$ elements of the sequence $\{40, 38, 34, 31, 26, 16, 8, 4, 2, 1\}$, and $p$ is selected from the range $\{1,2, \dotsc, m\}$, as presented in \cite{kuccukyavuz2012mixing}.
\end{example}

To compute the facet-defining inequalities of the convex hull for each configuration, we used a Python wrapper for the Qhull library \cite{qhull}.
All experiments were conducted on a Linux Ubuntu 20.04 system running on a PC equipped with an Intel Core i7-9700 3.00 GHz processor and 32 GB of RAM.
For $m=10$ and $ p\geq 5$, we were unable to compute the entire convex hull within 12 hours, which further highlights the complexity of these problems. 

Tables~\ref{T: Results_Luedtke} and \ref{T: Results_Simge} illustrate the strength of existing families of valid inequalities, as characterized by Proposition~\ref{prop:chance_lift_permutation}, as well as the strength of the \BLP inequalities proposed in Theorems~\ref{thm:chance_lift_generic} and \ref{thm: chance_liftVIAgeneric_permutation}, in describing the convex hull for Examples~\ref{ex: Leudtke} and \ref{ex: Simge}, respectively.
In each table, we report the percentage of facet-defining inequalities of the convex hull that are captured by the inequalities from Proposition~\ref{prop:chance_lift_permutation} and Theorems~\ref{thm:chance_lift_generic} and \ref{thm: chance_liftVIAgeneric_permutation}, under columns labeled with the corresponding results.
\blue{The column ``Definition~\ref{def:subclass}" reports the percentage of facet-defining inequalities of the convex hull that belong to the $Q$-symmetric subclass of \BLP inequalities, as characterized in Definition~\ref{def:subclass}.}
The column ``Imp.$^{\ref{thm: chance_liftVIAgeneric_permutation}}$'' reports the improvement in convex hull coverage achieved by the family of \BLP inequalities in Theorem~\ref{thm: chance_liftVIAgeneric_permutation} over the coverage provided by Proposition~\ref{prop:chance_lift_permutation}. Similarly, the column ``Imp.$^{\ref{thm:chance_lift_generic}}$'' shows the additional coverage achieved by the general \BLP inequalities in Theorem~\ref{thm:chance_lift_generic} over those in Theorem~\ref{thm: chance_liftVIAgeneric_permutation}. Finally, the column ``Total Imp.'' reports the overall improvement obtained by combining the \BLP inequalities from Theorems~\ref{thm:chance_lift_generic} and \ref{thm: chance_liftVIAgeneric_permutation}, relative to the baseline established by Proposition~\ref{prop:chance_lift_permutation}.
We note that when $p=m$, $\mac{F}^=$ reduces to a mixing set \eqref{eq:mixing}, for which the star inequalities in Proposition~\ref{prop:mixing} are sufficient to describe the convex hull; the same holds for Theorems~\ref{thm:chance_lift_generic} and \ref{thm: chance_liftVIAgeneric_permutation}. Similarly, when $p=1$, there is only one facet-defining inequality of the form $z+ (h_1-h_2) z_1 \ge h_m$, which is captured by the strengthened star inequalities in Proposition~\ref{prop:chance} as well as by Theorems~\ref{thm:chance_lift_generic} and \ref{thm: chance_liftVIAgeneric_permutation}. Therefore, we exclude these trivial cases when $p=1$ and $p=m$ from Tables~\ref{T: Results_Luedtke} and \ref{T: Results_Simge}.

Observe that for Examples~\ref{ex: Leudtke} and \ref{ex: Simge}, the convex hull is fully characterized by existing results only when $p=2$ or $p=m-1$. This indicates that our inequalities are equally effective in these cases, as also stated in Corollary~\ref{cor:special_case}. In contrast, when the existing results do not fully characterize the convex hull, Theorems~\ref{thm:chance_lift_generic} and \ref{thm: chance_liftVIAgeneric_permutation} complement them to achieve full characterization in many additional cases. 
Specifically, Theorem~\ref{thm: chance_liftVIAgeneric_permutation} provides additional characterization ranging within $[7.69\%,25.54\%]$ for Example~\ref{ex: Leudtke} and $[7.68\%,25.54\%]$ for Example~\ref{ex: Simge}. Further, Theorem~\ref{thm:chance_lift_generic} contributes an additional $[7.69\%,46.61\%]$ and $[7.69\%,42.32\%]$ characterization for Examples~\ref{ex: Leudtke} and \ref{ex: Simge}, respectively. 
Notably, for Example~\ref{ex: Leudtke}, Theorem~\ref{thm:chance_lift_generic} fully characterizes the convex hull in all but one of the remaining 17 cases.
In the exceptional case, when $m=9$ and $p=7$, it still achieves a coverage of $99.77$\% of the convex hull. 
In these instances, our \BLP inequalities provide a further improvement of $[15.38\%,68.05\%]$ over the existing results in Proposition~\ref{prop:chance_lift_permutation}.
For Example~\ref{ex: Simge}, there are six out of 17 remaining cases where Theorem~\ref{thm:chance_lift_generic} captures at least $90.92$\% of the convex hull, offering an additional $[15.38\%,62.77\%]$ improvement over prior results.
In summary, these computational results highlight that the family of \BLP inequalities introduced in this work advances our understanding of the polyhedral structure of chance-constrained models, offering substantial improvements over existing formulations in both coverage and completeness. %For Example \ref{ex: Leudtke}, there is at least $15.38$ \% and at most $68.05$ \% of the convex hull that is further characterized, whereas for Example \ref{ex: Simge}, there is at least $15.38$ \% and at most $62.77$ \% improvement. 

\blue{In addition, a substantial proportion of the \BLP inequalities appearing in the convex hull descriptions of the studied sets via Theorem~\ref{thm: chance_liftVIAgeneric_permutation} belong to the $Q$-symmetric subclass and therefore admit an efficient separation procedure, as outlined in Algorithm~\ref{alg:separation_1}. In particular, for both Examples~\ref{ex: Leudtke} and~\ref{ex: Simge}, all valid inequalities characterized by Theorem~\ref{thm: chance_liftVIAgeneric_permutation} belong to this subclass when $m \in \{3,4,5,6,10\}$. The same holds for $m \in \{7,8,9\}$ with either $2 \le p \le 4$ or $p = m-1$. In the remaining cases, where $m \in \{7,8,9\}$ and $5 \le p \le m-2$, the proportion of valid inequalities not captured by the $Q$-symmetric subclass ranges within $[12.45\%, 23.31\%]$ for Example~\ref{ex: Leudtke} and within $[9.33\%,22.73\%]$ for Example~\ref{ex: Simge}. These results further demonstrate the importance of the $Q$-symmetric subclass within the \BLP family in describing the convex hull. Moreover, by the definition of $Q$-symmetric inequalities, and as supported by these preliminary computational results, this subclass is expected to capture an even larger proportion of the \BLP family in practical settings where $p \ll m$.}

\begin{table}
\centering
\caption{Comparison between existing and new results in characterizing the convex hull for Example \ref{ex: Leudtke}.}
\label{T: Results_Luedtke}
\begin{tabular}{ll|l|lll|ll|l}
\toprule
$m$ & $p$ & \makecell[l]{Proposition~\ref{prop:chance_lift_permutation} \\(\%)} & \makecell[l]{Theorem~\ref{thm: chance_liftVIAgeneric_permutation}\\ (\%)}  & \blue{\makecell[l]{Definition~\ref{def:subclass}\\ (\%)}} & 
\makecell[l]{Imp.$^{\ref{thm: chance_liftVIAgeneric_permutation}}$ \\(\%)}    & 
\makecell[l]{Theorem~\ref{thm:chance_lift_generic} \\ (\%)} & 
\makecell[l]{Imp.$^{\ref{thm:chance_lift_generic}}$ \\(\%)}  & 
\makecell[l]{Total Imp. \\(\%)}  \\ 
\midrule

3 & 2 & 100.0 & 100.0 & \blue{100.0} & - & 100.0 & - & - \\
\midrule
4 & 2 & 100.0 & 100.0 & \blue{100.0} & - & 100.0 & - & - \\
4 & 3 & 100.0 & 100.0 & \blue{100.0} & - & 100.0 & - & - \\
\midrule
5 & 2 & 100.0 & 100.0 & \blue{100.0} & - & 100.0 & - & - \\
5 & 3 & 84.62 & 92.31 & \blue{92.31} & 7.69 & 100.0 & 7.69 & 15.38 \\
5 & 4 & 100.0 & 100.0 & \blue{100.0} & - & 100.0 & - & - \\
\midrule
6 & 2 & 100.0 & 100.0 & \blue{100.0} & - & 100.0 & - & - \\
6 & 3 & 72.73 & 86.36 & \blue{86.36} & 13.63 & 100.0 & 13.64 & 27.27 \\
6 & 4 & 76.32 & 86.84 & \blue{86.84} & 10.52 & 100.0 & 13.16 & 23.68 \\
6 & 5 & 100.0 & 100.0 & \blue{100.0} & - & 100.0 & - & - \\
\midrule
7 & 2 & 100.0 & 100.0 & \blue{100.0} & - & 100.0 & - & - \\
7 & 3 & 64.71 & 82.35 & \blue{82.35} & 17.64 & 100.0 & 17.65 & 35.29 \\
7 & 4 & 59.3 & 76.74 & \blue{76.74} & 17.44 & 100.0 & 23.26 & 40.7 \\
7 & 5 & 63.0 & 75.0 & \blue{60.0} & 12.0 & 100.0 & 25.0 & 37.0 \\
7 & 6 & 100.0 & 100.0 & \blue{100.0} & - & 100.0 & - & - \\
\midrule
8 & 2 & 100.0 & 100.0 & \blue{100.0} & - & 100.0 & - & - \\
8 & 3 & 59.18 & 79.59 & \blue{79.59} & 20.41 & 100.0 & 20.41 & 40.82 \\
8 & 4 & 48.81 & 70.24 & \blue{70.24} & 21.43 & 100.0 & 29.76 & 51.19 \\
8 & 5 & 41.95 & 60.4 & \blue{45.3} & 18.45 & 100.0 & 39.6 & 58.05 \\
8 & 6 & 61.43 & 71.43 & \blue{55.24} & 10.0 & 100.0 & 28.57 & 38.57 \\
8 & 7 & 100.0 & 100.0 & \blue{100.0} & - & 100.0 & - & - \\
\midrule
9 & 2 & 100.0 & 100.0 & \blue{100.0} & - & 100.0 & - & - \\
9 & 3 & 55.22 & 77.61 & \blue{77.61} & 22.39 & 100.0 & 22.39 & 44.78 \\
9 & 4 & 41.98 & 65.87 & \blue{65.87} & 23.89 & 100.0 & 34.13 & 58.02 \\
9 & 5 & 31.95 & 53.39 & \blue{40.94} & 21.44 & 100.0 & 46.61 & 68.05 \\
9 & 6 & 37.78 & 54.07 & \blue{30.76} & 16.29 & 100.0 & 45.93 & 62.22 \\
9 & 7 & 60.81 & 70.95 & \blue{53.38} & 10.14 & 99.77 & 28.82 & 38.96 \\
9 & 8 & 100.0 & 100.0 & \blue{100.0} & - & 100.0 & - & - \\
\midrule
10 & 2 & 100.0 & 100.0 & \blue{100.0} & - & 100.0 & - & - \\
10 & 3 & 52.27 & 76.14 & \blue{76.14} & 23.87 & 100.0 & 23.86 & 47.73 \\
10 & 4 & 37.23 & 62.77 & \blue{62.77} & 25.54 & 100.0 & 37.23 & 62.77 \\

\bottomrule
\end{tabular}
\end{table}

\begin{table}
\centering
\caption{Comparison between existing and new results in characterizing the convex hull for Example \ref{ex: Simge}.}
\label{T: Results_Simge}
\begin{tabular}{ll|l|lll|ll|l}
\toprule
$m$ & $p$ & \makecell[l]{Proposition~\ref{prop:chance_lift_permutation} \\(\%)} & \makecell[l]{Theorem~\ref{thm: chance_liftVIAgeneric_permutation}\\ (\%)}  & \blue{\makecell[l]{Definition~\ref{def:subclass}\\ (\%)}} & 
\makecell[l]{Imp.$^{\ref{thm: chance_liftVIAgeneric_permutation}}$ \\(\%)}    & 
\makecell[l]{Theorem~\ref{thm:chance_lift_generic} \\ (\%)} & 
\makecell[l]{Imp.$^{\ref{thm:chance_lift_generic}}$ \\(\%)}  & 
\makecell[l]{Total Imp. \\(\%)}  \\ 
\midrule

3 & 2 & 100.0 & 100.0 & \blue{100.0} & - & 100.0 & - & - \\
\midrule

4 & 2 & 100.0 & 100.0 & \blue{100.0} & - & 100.0 & - & - \\
4 & 3 & 100.0 & 100.0 & \blue{100.0} & - & 100.0 & - & - \\
\midrule

5 & 2 & 100.0 & 100.0 & \blue{100.0} & - & 100.0 & - & - \\
5 & 3 & 84.62 & 92.31 & \blue{92.31} & 7.69 & 100.0 & 7.69 & 15.38 \\
5 & 4 & 100.0 & 100.0 & \blue{100.0} & - & 100.0 & - & - \\
\midrule

6 & 2 & 100.0 & 100.0 & \blue{100.0} & - & 100.0 & - & - \\
6 & 3 & 72.73 & 86.36 & \blue{86.36} & 13.63 & 100.0 & 13.64 & 27.27 \\
6 & 4 & 76.32 & 86.84 & \blue{86.84} & 10.52 & 100.0 & 13.16 & 23.68 \\
6 & 5 & 100.0 & 100.0 & \blue{100.0} & - & 100.0 & - & - \\
\midrule

7 & 2 & 100.0 & 100.0 & \blue{100.0} & - & 100.0 & - & - \\
7 & 3 & 64.71 & 82.35 & \blue{82.35} & 17.64 & 100.0 & 17.65 & 35.29 \\
7 & 4 & 59.3 & 76.74 & \blue{76.74} & 17.44 & 100.0 & 23.26 & 40.7 \\
7 & 5 & 70.87 & 80.58 & \blue{66.99} & 9.71 & 97.09 & 16.51 & 26.22 \\
7 & 6 & 100.0 & 100.0 & \blue{100.0} & - & 100.0 & - & - \\
\midrule

8 & 2 & 100.0 & 100.0 & \blue{100.0} & - & 100.0 & - & - \\
8 & 3 & 59.18 & 79.59 & \blue{79.59} & 20.41 & 100.0 & 20.41 & 40.82 \\
8 & 4 & 48.81 & 70.24 & \blue{70.24} & 21.43 & 100.0 & 29.76 & 51.19 \\
8 & 5 & 48.97 & 64.01 & \blue{51.62} & 15.04 & 92.92 & 28.91 & 43.95 \\
8 & 6 & 62.8 & 70.4 & \blue{50.8} & 7.6 & 96.0 & 25.6 & 33.2 \\
8 & 7 & 100.0 & 100.0 & \blue{100.0} & - & 100.0 & - & - \\
\midrule

9 & 2 & 100.0 & 100.0 & \blue{100.0} & - & 100.0 & - & - \\
9 & 3 & 55.22 & 77.61 & \blue{77.61} & 22.39 & 100.0 & 22.39 & 44.78 \\
9 & 4 & 41.98 & 65.87 & \blue{65.87} & 23.89 & 100.0 & 34.13 & 58.02 \\
9 & 5 & 37.51 & 54.83 & \blue{45.5} & 17.32 & 91.34 & 36.51 & 53.83 \\
9 & 6 & 35.88 & 48.6 & \blue{25.87} & 12.72 & 90.92 & 42.32 & 55.04 \\
9 & 7 & 53.59 & 61.27 & \blue{40.23} & 7.68 & 95.66 & 34.39 & 42.07 \\
9 & 8 & 100.0 & 100.0 & \blue{100.0} & - & 100.0 & - & - \\
\midrule

10 & 2 & 100.0 & 100.0 & \blue{100.0} & - & 100.0 & - & - \\
10 & 3 & 52.27 & 76.14 & \blue{76.14} & 23.87 & 100.0 & 23.86 & 47.73 \\
10 & 4 & 37.23 & 62.77 & \blue{62.77} & 25.54 & 100.0 & 37.23 & 62.77 \\

\bottomrule

\end{tabular}
\end{table}

\smallskip
\blue{
We conclude this section by emphasizing that the computational study presented herein has two primary objectives: (i) to quantify the proportion of the convex hull of the mixing set captured by the proposed closed-form inequalities, and (ii) to compare this coverage with that achieved by existing families of valid inequalities in the literature using benchmark instances that have been commonly used in previous studies. These objectives are consistent with the theoretical focus of the paper, namely, the development of a new family of closed-form valid inequalities that generalizes existing results and captures a larger portion of the convex hull. Consequently, the computational study is not intended to evaluate the efficiency of cut generation within a solver framework. Such an investigation is beyond the scope of the present work and is left for future computational research.
}

\section{Conclusion and Future Research} \label{sec:discussion}

Chance-constrained programming offers a versatile framework for modeling optimization problems under uncertainty; however, its mixed-integer reformulations introduce nonconvex mixing sets with knapsack constraints, leading to weak continuous relaxations and increased computational complexity.
Prior research has developed various families of valid inequalities, most notably {\it mixing inequalities}, to better describe these sets. However, existing approaches capture only a limited portion of the convex hull and rely on case-specific extensions of known inequalities, restricting the discovery of fundamentally new structures and systematic convexification methods. These gaps motivate the need for a unifying framework that can generate stronger and more comprehensive representations of chance-constrained sets.

To address these challenges, this paper introduces a new bilinear convexification framework that reformulates mixing sets with a knapsack constraint as bilinear sets over a simplex in a lifted space. The proposed aggregation-based procedure enables systematic derivation of facet-defining inequalities in the original space of variables, generalizing and unifying known families while revealing new classes of valid inequalities that capture much larger portions of the convex hull. Preliminary computational studies demonstrate that the new inequalities describe over 90\% of the facets of benchmark instances, offering significant strengthening over existing relaxations. 

The proposed aggregation-based convexification procedure advances the polyhedral understanding of CCPs and opens new avenues for research.
The class of valid inequalities proposed in this paper is derived from a single-row relaxation of the chance constraint, which yields a single mixing set with a knapsack constraint. Any facet-defining inequality for such a set is also facet-defining for the joint CCPs, where the key substructure is the intersection of multiple mixing sets with a shared knapsack constraint. However, there also exist facet-defining inequalities for joint CCPs that cannot be obtained from single-row relaxations. A line of research has focused on extending the strength of single-row relaxations to multi-row intersections, see, e.g., \cite{luedtke2010integer,zhao2017polyhedral,liu2019intersection}. 
An interesting direction for future research is to investigate the convexification of the intersection of multiple mixing sets with a knapsack constraint through the lens of the bilinear reformulation introduced in this work, thereby enabling the application of the proposed aggregation-based technique.

% Acknowledgments here
\ACKNOWLEDGMENT{We would like to express our sincere gratitude to the Associate Editor and anonymous reviewers for their valuable suggestions, which helped improve the manuscript.}

% References here (outcomment the appropriate case)

% CASE 1: BiBTeX used to constantly update the references
%   (while the paper is being written).
\bibliographystyle{informs2014} % outcomment this and next line in Case 1
\bibliography{ref} % if more than one, comma separated

\end{document}